\documentclass[11pt]{article}

\usepackage{amsfonts}
\usepackage{amssymb}
\usepackage{amsmath,mathrsfs}
\usepackage{amsthm}
\usepackage{latexsym}
\usepackage{layout}

\usepackage{marvosym}
\usepackage{textcomp}
\usepackage{amsbsy}
\usepackage{latexsym}
\usepackage[mathscr]{eucal}
\usepackage{amsfonts}
\usepackage[usenames]{color}

\usepackage[english]{babel}

\theoremstyle{plain}
\begingroup
\newtheorem{theorem}{Theorem}[section]
\newtheorem{lemma}[theorem]{Lemma}
\newtheorem{proposition}[theorem]{Proposition}
\newtheorem{corollary}[theorem]{Corollary}
\endgroup

\theoremstyle{definition}
\begingroup
\newtheorem{definition}[theorem]{Definition}
\newtheorem{remark}[theorem]{Remark}
\newtheorem{example}[theorem]{Example}
\endgroup
 
\theoremstyle{remark}
\begingroup

\endgroup

\mathchardef\emptyset="001F

\numberwithin{equation}{section}

\newcommand{\be}{\begin{equation}}
  \newcommand{\ee}{\end{equation}}
\newcommand{\bes}{\begin{eqnarray}}
\newcommand{\ees}{\end{eqnarray}}

\usepackage{bm}
\usepackage{graphicx}
\usepackage[hang]{subfigure}
\usepackage{stmaryrd}
\usepackage{url}
\usepackage[utf8]{inputenc}

\makeatletter \newcommand{\thmenumhspace}[1]{\sbox{\@labels}{\unhbox\@labels\hskip#1}} \makeatother

\renewcommand{\i}{\mathrm{i}} 
\newcommand{\dd}{\, \mathrm{d}} 
\newcommand{\diff}{\mathop{}\!\mathrm{d}}

\newcommand{\argmin}{\operatornamewithlimits{arg\,min}} 
\newcommand{\supp}{\operatorname{supp}}
\newcommand{\e}{\mathrm{e}}
\renewcommand\Re{\operatorname{Re}}

\newcommand{\newpar}{\par\vspace{\baselineskip}}

\usepackage{xspace}
\makeatletter
\newcommand\etc{etc\@ifnextchar.{}{.\@\xspace}}
\newcommand\ie{i.e.\@ifnextchar,{}{\@\xspace}}
\newcommand\eg{e.g.\@ifnextchar,{}{\@\xspace}}
\newcommand\wrt{with respect to\@ifnextchar,{}{\@\xspace}}
\makeatother

\newcommand\MoveEqLeft[1][2]{\kern #1em & \kern -#1em} 
\newcommand{\leadeq}[2][4]{\MoveEqLeft[#1] #2 \nonumber}
\newcommand{\leadeqnum}[2][4]{\MoveEqLeft[#1] #2}

\usepackage{xcolor}
\usepackage[normalem]{ulem}

\begin{document}
\title{Consistency of Probability Measure Quantization by Means of Power Repulsion-Attraction Potentials}
\author{
Massimo Fornasier
\footnote{Technische Universit\"at M\"unchen, Fakult\"at Mathematik, Boltzmannstrasse 3
 D-85748, Garching bei M\"unchen, Germany  ({\tt massimo.fornasier@ma.tum.de}). } \, and \,
Jan-Christian Hütter\footnote{Technische Universit\"at M\"unchen, Fakult\"at Mathematik, Boltzmannstrasse 3
 D-85748, Garching bei M\"unchen, Germany  ({\tt jan-christian.huetter@mytum.de}).}
}
\maketitle

\begin{abstract}
This paper is concerned with the study of the consistency of a variational method for probability measure quantization, deterministically 
realized by means of a minimizing principle, balancing power repulsion and attraction potentials. The proof of consistency is based 
on the construction of a target energy functional whose unique minimizer is actually the given probability measure $\omega$ to be quantized. 
Then we show that the discrete functionals, defining  the discrete quantizers as their minimizers, actually $\Gamma$-converge to the target energy with respect to the narrow topology 
on the space of probability measures.
A key ingredient is the reformulation of the target functional by means of a Fourier representation, which extends the characterization
of conditionally positive semi-definite functions from points in generic position to probability measures. As a byproduct of the Fourier
representation, we also obtain compactness of sublevels of the target energy in terms of uniform moment bounds, which 
already found applications in the asymptotic analysis of corresponding gradient flows. To model situations where the given probability is 
affected by noise, we additionally consider a modified energy, with the addition of a regularizing total variation term and we investigate 
again its point mass approximations in terms of $\Gamma$-convergence. We show that such a discrete measure 
representation of the total variation can be interpreted as an additional nonlinear potential, repulsive at a short range, attractive at a 
medium range, and at a long range not having effect, promoting a uniform distribution of the point masses.   
\end{abstract}

\section{Introduction}
\label{cha:introduction}

\subsection{Variational measure quantization and main results of the paper}

Quantization of $d$-dimensional probability measures deals with constructive methods to define atomic  probability measures 
supported on a finite number of discrete points, which best approximate a given (diffuse) probability measure \cite{00_Graf_Luschgy_quantization, 04-Gruber-Quantization}. 
Here the space of all probability measures is endowed with the Wasserstein or
Kantorovich metric, which is usually the measure of the distortion of the approximation.
The main motivations come from 
two classical relevant applications. The first we mention is information theory.
In fact the problem of the quantization of a $d$-dimensional measure $\omega$ can be re-interpreted as the
best approximation of a random $d$-dimensional vector $X$ with distribution $\omega$ by means of 
a random vector $Y$ which has at most $N$ possible values in its image. This is a classical
way of considering the digitalization of an analog signal, 
for the purpose of optimal data storage or parsimonious transmission of impulses via a channel. 
As we shall recall in more detail below, image dithering \cite{scgwbrwe11,testgwscwe11} is a modern
example of such an application in signal processing. 
The second classical application is numerical integration \cite{pa98}, where integrals  with respect to certain probability measures  need to be
well-approximated by corresponding quadrature rules defined on the possibly optimal quantization points with respect to classes of continuous functions.
Numerical integration belongs to the standard problems of numerical analysis with numerous applications. It is often needed as a relevant subtask for 
solving more involved problems, for instance, the numerical approximation of solutions of partial differential equations.
Additionally a number of problems in physics, e.g., in quantum physics, as well as any expectation in a variety of stochastic models 
require the computation of high-dimensional integrals as main (observable) quantities of interest.
However, let us stress that the range of applications of measure quantization has nowadays become more far 
reaching, including mathematical models in economics (optimal location of service centers) or biology (optimal foraging and
population distributions).

In absence of special structures of the underlying probability measure, for instance being well-approximated
by finite sums of  tensor products of  lower dimensional measures, the problem of optimal quantization of measures, especially when defined on high-dimensional domains, can 
be hardly solved explicitely by deterministic methods.  
In fact, one may need to define optimal tiling of the space into Voronoi cells, based again on testing the space by suitable high-dimensional
integrations, see Section \ref{sec:discr-kern-estim} below
for an explicit deterministic construction of high-dimensional tilings for approximating probability measures
by discrete measures. When the probability distribution can be empirically ``tested'', by being able to draw at random samples from it, measure quantization can be realized by means of 
empirical processes. This way of generating natural quantization points leads to the consistency of the approximation, 
in the sense of almost sure convergence of the empirical processes to the original probability measure as the number of draws goes to infinity, see Lemma \ref{lem:3} below.
Other results address also the approximation rate of such a randomized quantization, measuring the expected valued of the
Wasserstein distance between the empirical processes and original probability measure, see for instance \cite{descsc13} and references therein.
Unfortunately, in those situations where the probability distribution is given but it is too expensive or even impossible to be sampled, also
the use of simple empirical processes might not be viable. A concrete example of this situation is image dithering\footnote{\url{http://en.wikipedia.org/wiki/Dither}}, see
Figure \ref{fig:figure2}. In this case the image represents the given probability distribution, which we do actually can access, but it is evidently impossible
to sample  random draws from it, unless one designs actually a quantization of the image by means of deterministic methods, which again may leads
us to tilings, and eventually making use of pseudorandom number generators\footnote{\url{http://en.wikipedia.org/wiki/Pseudorandom_number_generator}. One practical way to sample randomly an image would be first to generate (pseudo-)randomly a finite number
of points according to the uniform distribution from which one eliminates points which do not realize locally an integral over a prescribed threshold.}.  To overcome this difficulty,  a variational approach has been proposed in a series of papers \cite{scgwbrwe11,testgwscwe11}.
\begin{figure}[t]
 \centering
  \subfigure[Original image]{\includegraphics[width=6cm]{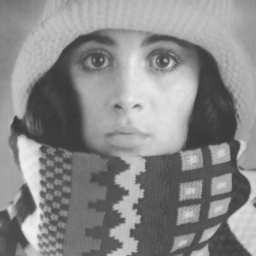}}
 \hspace{0.5cm}
  \subfigure[Dithered image]{\includegraphics[width=6cm]{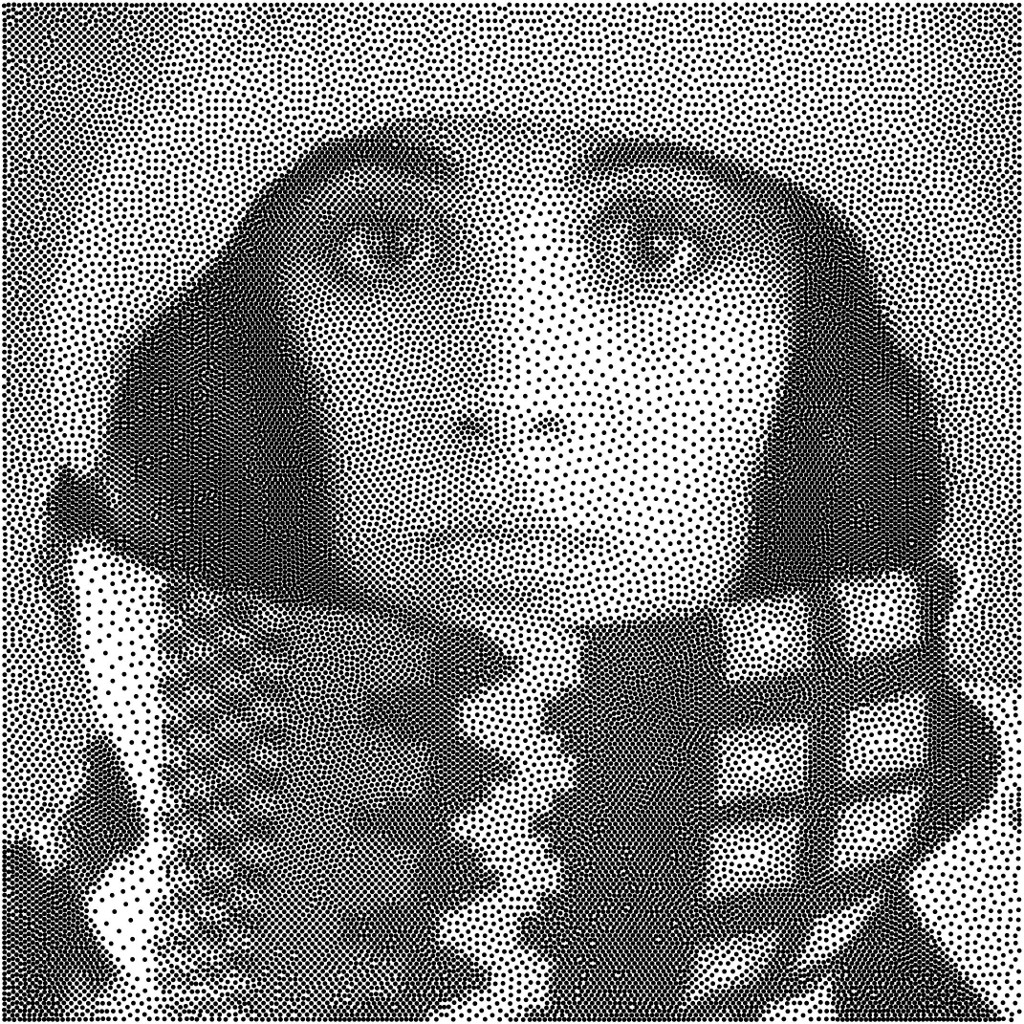}}
  \caption{Un-dithered and dithered image from \cite{testgwscwe11} with the kind permission of the authors.}
  \label{fig:figure2}
\end{figure}
 While there are many ways to determine the proximity of two probability measures (for a brief summary over some relevant alternatives, see \cite{07_Carillo_Toscani_prob-metrics}), the interesting idea in  \cite{testgwscwe11} consists in employing variational principles for that purpose. Namely, we consider the points \( x=(x_i)_{i=1,\ldots,N} \) to be attracted by the high-intensity locations of the black-and-white image, which represents our probability distribution $\omega$, by introducing an attraction potential
\begin{equation*}
  \mathcal{V}_N(x) := \frac{1}{N} \sum_{i=1}^{N} \int_{\mathbb{R}^2} \left| x_i - y \right| \diff \omega(y)
\end{equation*}
which is to be minimized. If left as it is, the minimization of this term will most certainly not suffice to force the points into an intuitively good position, as the
minimizer would consist of all the points being at the median of \( \omega \). Therefore, we shall enforce the spread of the points by adding a pairwise repulsion term
\begin{equation*}
  \mathcal{W}_N(x) := - \frac{1}{2N^2} \sum_{i,j=1}^{N} \left| x_i-x_j \right|,
\end{equation*}
leading to the minimization of the composed functional
\begin{equation}
  \label{eq:7}
  \mathcal{E}_N(x) := \mathcal{V}_N(x) + \mathcal{W}_N(x),
\end{equation}
which produces the visually appealing results in Figure \ref{fig:figure2}. By considering more general kernels $K_a(x_i,y)$ and $K_r(x_i,x_j)$ in the attraction and repulsion terms
\begin{equation}
  \label{eq:7'}
\mathcal{E}_N^{K_a,K_r}(x) =\frac{1}{N} \sum_{i=1}^{N} \int_{\mathbb{R}^2}  K_a(x_i,y) \diff \omega(y)- \frac{1}{2N^2} \sum_{i,j=1}^{N} K_r(x_i,x_j),
\end{equation}
as already mentioned above, an attraction-repulsion functional of this type can easily be prone to other interesting interpretations. For instance, one could also consider the particles as a population subjected to attraction to a nourishment source \( \omega \), modeled by the attraction term, while at the same time being repulsed by internal competition. As one can see in the numerical experiments reported in \cite[Section 4]{FHS12}, the interplay of different powers of attraction and repulsion forces can lead some individuals of the population to fall out of the domain of the resource (food), which can be interpreted as an interesting mathematical model of social exclusion.
The relationship between functionals of the type \eqref{eq:7'} and optimal numerical integration in reproducing kernel Hilbert spaces has been highlighted in \cite{grpost11}, also showing once again the relevance of (deterministic) measure quantization
towards designing efficient quadrature rules for numerical integration.\\

However, the generation of optimal quantization points by the minimization of functionals of the type \eqref{eq:7'} might also be subjected to criticism. First of all the functionals are in general nonconvex, rendering their global optimization, especially in high-dimension, a problem
of high computational complexity, although, being the functional the difference of two convex terms, numerical methods based on the alternation of descent and ascent algorithms proved to be rather efficient  in practice, see \cite{testgwscwe11} for details.
Especially one has to notice that for kernels generated by radial symmetric functions applied on the Euclidean distance of their arguments, the evaluation of the functional and of its subgradients may result in the computation of convolutions
which can be rapidly implemented by non-equispaced fast Fourier transforms \cite{fest04, post03}. Hence, this technical advantage makes it definitively a promising alternative (for moderate dimensions $d$) with respect to deterministic methods
towards optimal space tiling, based on local integrations and greedy placing, as it is for instance done in the strategy proposed in Section \ref{sec:discr-kern-estim} below.
Nevertheless, while for both empirical processes and deterministic constructions consistency results are available, see for instance Lemma \ref{lem:3} and Lemma \ref{lem:13} below, and the broad literature on these techniques \cite{00_Graf_Luschgy_quantization},  so far
no similar results have been provided for discrete measures supported on optimal points generated as minimizers of functionals of the type \eqref{eq:7'}, which leads us to the scope of this paper.

We shall prove that, for a certain type of kernels $K_a(x,y)=\psi_a(x-y)$ and $K_r(x,y)=\psi_r(x-y)$, where \( \psi_a \colon \mathbb{R}^d \rightarrow \mathbb{R}_+ \)  and  \( \psi_r \colon \mathbb{R}^d \rightarrow \mathbb{R}_+ \)
are radially symmetric functions, the empirical measures $\mu_N = \frac{1}{N} \sum_{i=1}^N \delta_{x_i}$ constructed over points  \( x=(x_i)_{i=1,\ldots,N} \) minimizing \eqref{eq:7'} converges narrowly to the given  probability measure $\omega$, showing
the consistency of this quantization method. The technique we intend to use to achieve this result makes use of the so-called $\Gamma$-convergence \cite{93-Dal_Maso-intro-g-conv}, which is a type of variational convergence of sequences of functionals over
metrizable spaces, which allows for simultaneous convergence of their respective minimizers. The idea  is to construct a ``target functional'' $\mathcal E$ whose unique minimizer is actually the given probability measure $\omega$. Then one needs
to show that the functionals $\mathcal{E}_N^{K_a,K_r}$ actually $\Gamma$-converge to  $\mathcal E$ with respect to the narrow topology on the space of probability measures, leading eventually to the convergence of the corresponding minimizers to $\omega$.
We immediately reveal that the candidate target functional for this purpose is, in the first instance, given by
\begin{equation}
  \label{eq:9}
  \mathcal{E}[\mu] := \int_{\Omega \times \Omega} \psi_a(x-y) \diff \omega(x) \diff \mu(y) - \frac{1}{2} \int_{\Omega \times \Omega} \psi_r(x-y)  \diff \mu(x) \diff \mu(y),
\end{equation}
where we consider from now on a more general domain $\Omega \subset \mathbb R^d$ as well as measures $\mu, \omega \in \mathcal P(\Omega)$, where $\mathcal P(\Omega)$ is the space of probability measures.
The reason for this natural choice comes immediately by observing that $\mathcal{E}[\mu_N] = \mathcal{E}_N^{K_a,K_r}(x)$. For later use we denote 
$$\mathcal V[\mu] := \int_{\Omega\times \Omega} \psi_a(x-y) \diff \omega(x) \diff \mu(y) \text { and } \mathcal W[\mu]:=-\frac{1}{2} \int_{\Omega\times \Omega} \psi_r(x-y)  \diff \mu(x) \diff \mu(y).
$$
However, this natural choice poses several mathematical questions. First of all, as the functional is composed by the difference of two positive terms which might be simultaneously not finite
over the set of probability measures, its well-posedness has to be justified. This will be done by restricting the class of radial symmetric functions $\psi_a$ and $\psi_r$ to those
with at most quadratic polynomial growth and the domain of the functional to probability measures with bounded second moment. This solution, however, conflicts with the natural
topology of the problem, which is the one induced by the narrow convergence. In fact, the resulting functional will not be necessarily lower semi-continuous with respect to the
narrow convergence and this property is well-known to be necessary for a target functional to be a $\Gamma$-limit \cite{93-Dal_Maso-intro-g-conv}.
Thus, we need to extend the functional $\mathcal{E}$ from the probability measures with bounded second moment $\mathcal P_2(\Omega)$ to the entire $\mathcal P(\Omega)$,
by means of a functional \(\widehat{\mathcal{E}}\) which is also lower semi-continuous with respect to the narrow topology. The first relevant result of this paper is to prove 
that such a lower semi-continuous relaxation \(\widehat{\mathcal{E}}\) can be explicitly expressed, up to an additive constant terms, for $\psi(\cdot)=\psi_a(\cdot)=\psi_r(\cdot) = |\cdot|^q$, and $1 \leq q < 2$, in terms
of the Fourier formula

\begin{equation}
  \label{eq:01}
  \widehat{\mathcal{E}}[\mu]= -2^{-1} (2\pi)^{-d}\int_{\mathbb{R}^d}\left| 
    \widehat{\mu}(\xi) - \widehat{\omega}(\xi) \right|^2
  \widehat{\psi}(\xi) \dd \xi,
\end{equation}
where for any \(\mu \in \mathcal{P}(\mathbb{R}^{d})\), \(\widehat{\mu}\) denotes its Fourier-Stieltjes transform,
and \(\widehat{\psi}\) is the \emph{generalized Fourier-transform of \(\psi\)}, \ie, a Fourier transform with respect to a certain duality, which allows to cancel the singularities of the Fourier transform of kernel $\psi$ at $0$. We have gathered most of the important facts about it in Appendix \ref{cha:cond-posit-semi}.
The connection between functionals composed of repulsive and attractive power terms and Fourier type formulas \eqref{eq:01} is novel and requires to extend the theory of conditionally positive semi-definite functions from discrete points 
to probability measures \cite{Wend05}. 
This crucial result is fundamental for proving as a consequence the well-posedness in $\mathcal P(\Omega)$ and the uniqueness of the minimizer $\omega$, as it is now evident by the form \eqref{eq:01}, and eventually the $\Gamma$-convergence
of the particle approximations. Another very relevant result which follows from the Fourier representation is the uniform $r$th-moment bound for $r< \frac{q}{2}$ of the sublevels of \(\widehat{\mathcal{E}}\)  leading to their
compactness in certain Wasserstein distances. This result plays a major role, for instance, in the analysis of the convergence to steady states of corresponding gradient flows (in dimension $d=1$), which are studied in the follow up paper \cite{dffohuma13}.
\\

Another relevant consequence of the Fourier representation is to allow us to add regularizations to the optimization problem. While for other quantization methods mentioned above, such as deterministic tiling and random draw of empirical processes,
it may be hard to filter the possible noise on the probability distribution, the variational approach based on the minimization of particle functionals of the type \eqref{eq:7} is amenable to easy mechanisms of regularization. Differently from the path followed in the
reasoning above, where we developed a limit from discrete to continuous functionals, here we proceed in the opposite direction, defining first the expected continuous regularized functional and then designing candidate discrete functional approximations, proving then the consistency again by $\Gamma$-convergence. One effective way of filtering noise and still preserving the structure of the underlying measure $\omega$ is the addition to the discrepancy functional of a term of {\it total variation}. This technique was introduced by Rudin, Osher, and Fatemi in the seminal paper \cite{rof92} for the denoising
of digital images, leading to a broad literature on variational methods over functions of bounded variations. We refer to \cite[Chapter 4]{10_Chambolle_ea_tv_intro} for an introduction to the subject and to the references therein for a broad view. Inspired by this well-established theory, we shall consider a regularization of \( \mathcal{E} \) by a total variation term,
\begin{equation}
  \label{eq:355}
  \mathcal{E}^\lambda[\mu] := \widehat{\mathcal{E}}[\mu] + \lambda \left| D\mu \right|(\Omega),
\end{equation}
where \( \lambda > 0 \) is a regularization parameter and \( \mu \) is assumed to be in \( L^1(\Omega) \), having distributional derivative \( D\mu \) which is a finite Radon measure with total variation \( \left| D\mu \right| \).
Beside providing existence of minimizers of $\mathcal{E}^\lambda$ in $\mathcal P(\Omega) \cap BV(\Omega)$ and its $\Gamma$-convergence to $\widehat{\mathcal{E}}$ for $\lambda \to 0$, we also formulate particle approximations to $\mathcal{E}^\lambda$.
While the approximation to the first term $\mathcal{E}$ is already given by its restriction to atomic measures, the consistent discretization in terms of point masses of the total variation term $\left| D\mu \right|(\Omega)$ is our last
result of the present paper. By means of  kernel estimators \cite{12-Wied-Weissbach-Kernel-Estimators}, we show in arbitrary dimensions that the total variation can be interpreted at the level of point masses as an additional attraction-repulsion
potential, actually repulsive at a short range, attractive at a medium range, and at a long range not having effect, which tends to locate the point masses into uniformly distributed configurations.
We conclude with the proof of consistency of such a discretization by means of $\Gamma$-convergence. To our knowledge this interpretation of the total variation in terms of point masses has never been pointed out before in the literature.

\subsection{Further relevance to other work}

Besides the aforementioned relationship to measure quantization in information theory, numerical integration, and the theory of conditionally positive semi-definite functions, energy functionals such as \eqref{eq:9}, being composed of a quadratic and a linear integral term, arise as well in a variety of mathematical models in biology and physics, describing the limit of corresponding particle descriptions. 
In particular the quadratic term, in our case denoted by \( \mathcal{W} \), corresponding to the self-interaction between particles, has emerged in modeling biological aggregation. We refer to the survey paper \cite{13-Carrillo-Choi-Hauray-MFL} and the references therein for a summary on the mathematical results related to the mean-field limit of large ensembles of interacting particles with applications in swarming models, with particular emphasis on existence and uniqueness of aggregation gradient flow equations. We also mention that in direct connection to \eqref{eq:9}, in the follow up paper \cite{dffohuma13} we review the global well-posedness of gradient flow equations associated to the energy $\mathcal E$ in one dimension, providing a simplified proof of existence and uniqueness, and we
address the difficult problem of describing the asymptotic behavior of their solutions. In this respect we stress once more that the moment bounds derived in Section \ref{sec:moment-bound-symm} of the present paper play a fundamental role for that analysis.
\\
Although here derived as a model of regularization of the approximation process to a probability measure, also functionals like \eqref{eq:355} with other kernels than polynomial growing ones  appear in the literature in various contexts.
The existence and characterization of their minimizers are in fact of great independent interest. When restricted to characteristic functions of finite perimeter sets, a functional of the type  \eqref{eq:355} with Coulombic-like repulsive interaction models the so-called non-local isoperimetric problem studied in \cite{2012-Muratov,2013-Muratov} and \cite{13-Cicalese-droplet-minimizers}. Non-local Ginzburg-Landau energies modeling diblock polymer systems with kernels given by the Neumann Green's function of the Laplacian are studied in \cite{12-Serfaty-DropletI,12-Serfaty-DropletII}. The power potential model studied in the present paper is contributing to this interesting constellation.
\newpar

\subsection{Structure of the paper}
\label{sec:organisation-thesis}

In Section \ref{sec:prel-obs}, we start with a few theoretical preliminaries, followed by examples and counterexamples of the existence of minimizers for \( \mathcal{E} \) in the case of power potentials, depending on the powers and on the domain \( \Omega \), where elementary estimates for the behavior of the power functions are used in conjunction with appropriate notions of compactness for probability measures, \ie, uniform integrability of moments and moment bounds.

Starting from Section \ref{sec:prop-funct-mathbbrd}, we study the limiting case of coinciding powers for attraction and repulsion, where there is no longer an obvious confinement property given by the attraction term. To regain compactness and lower semi-continuity, we consider the lower semi-continuous envelope of the functional  \( \mathcal{E} \), which can be proven to coincide, up to an additive constant, with the  Fourier representation \eqref{eq:01}, see Corollary \ref{cor:lower-semi-cont} in Section \ref{sec:extension-p}, which is at first derived on \( \mathcal{P}_2(\mathbb{R}^d) \) in Section \ref{sec:fourier-formula-p2}. The main ingredient to find this representation is the generalized Fourier transform in the context of the theory of conditionally positive definite functions, which we briefly recapitulated in Appendix \ref{cha:cond-posit-semi}.

Having thus established a problem which is well-posed for our purposes, we proceed to prove one of our main results, namely the convergence of the minimizers of the discrete functionals to \( \omega \), Theorem \ref{thm:cons-part-appr} in Section \ref{sec:part-appr}. This convergence will follow in a standard way from the \( \Gamma \)-convergence of the corresponding functionals. Furthermore, again applying the techniques of Appendix \ref{cha:cond-posit-semi} used to prove the Fourier representation  allows us to derive compactness of the sublevels of \( \mathcal{E} \) in terms of a uniform moment bound in Section \ref{sec:moment-bound-symm}.

Afterwards, in Section \ref{sec:tv-reg}, we shall introduce the total variation regularization of \( \mathcal{E} \). Firstly, we prove consistency in terms of \( \Gamma \)-convergence for vanishing regularization parameter in Section \ref{sec:cons-regul-cont}. Then, in Section \ref{sec:discrete-version-tv-1}, we propose two ways of computing a version of it on particle approximations and again prove consistency for \( N \to \infty \). One version consists of employing kernel density estimators, while, in the other one, each point mass is replaced by an indicator function extending up to the next point mass with the purpose of computing explicitly the total variation. In Section \ref{sec:numer-exper}, we exemplify the $\Gamma$-limits of the first approach by numerical experiments.

\section{Preliminary observations}
\label{sec:prel-obs}

\subsection{Narrow convergence and Wasserstein-convergence}
\label{sec:conv-p-pp}

We  begin with a brief summary of measure theoretical results which will be needed in the following. 
Let $\Omega \subset \mathbb R^d$ be fixed and denote with $\mathcal P_p(\Omega)$ the set of probability measures $\mu$ with finite $p$th-moment
$$
\int_{\Omega} |x|^p d\mu(x) < \infty. 
$$
For an introduction to the narrow topology in spaces of probability measures $\mathcal P(\Omega)$, see \cite[Chapter 5.1]{AGS08}.
Let us only briefly recall a few relevant facts, which will turn out to be useful later on. First of all let us recall the definition of narrow convergence. A sequence of probability measures $(\mu_n)_{n \in \mathbb N}$ narrowly converges
to $\mu \in \mathcal P(\Omega)$ if
$$
\lim_{n \to \infty} \left | \int_{\Omega} g(x) d \mu_n(x) - \int_{\Omega} g(x) d\mu(x) \right | = 0, \mbox{ for all } g \in C_b(\Omega).
$$
It is immediate to show that $L^1$ convergence of absolutely continuous probability measures in $\mathcal P(\Omega)$ implies narrow convergence. Moreover, as recalled in \cite[Remark 5.1.1]{AGS08}, there is a sequence of continuous functions \( (f_k)_{k \in \mathbb{N}} \) on \( \Omega \) and \( \sup_{x \in \Omega} \left| f_k(x) \right| \leq 1 \) such that the narrow convergence in \( \mathcal{P}(\Omega) \) can be metrized by 
  \begin{equation}
    \label{eq:318}
    \delta(\mu,\nu) := \sum_{k = 1}^{\infty} 2^{-k} \left| \int_{\Omega} f_k(x) \diff \mu(x) - \int_{\Omega} f_k(x) \diff \nu(x) \right|.
 \end{equation}
It will turn out to be useful also to observe that narrow convergences extends to tensor products. From \cite[Theorem 2.8]{68-Billingsley-Conv-Proba} it follows that if \( \left( \mu_n \right)_n \), \( \left( \nu_n \right)_n \) are two sequences in \( \mathcal{P}(\Omega) \) and \( \mu, \nu \in \mathcal{P}(\Omega) \), then
  \begin{equation*}
    \mu_n \otimes \nu_n \rightarrow \mu \otimes \nu \text{ narrowly} \mbox{ if and only if } \mu_n \rightarrow \mu \text{ and } \nu_n \rightarrow \nu \text{ narrowly}.
  \end{equation*}

Finally, we include some results about the continuity of integral functionals with respect to Wasserstein-convergence.

\begin{definition}[Wasserstein distance]
  \label{def:wasserstein-distance}
  Let \( \Omega \subseteq \mathbb{R}^d \), \( p \in [1,\infty) \) as well as \( \mu_1, \mu_2 \in \mathcal{P}_p(\Omega) \) be two probability measures with finite \( p \)th moment. Denoting by \( \Gamma(\mu_1, \mu_2) \) the probability measures on \( \Omega \times \Omega \) with marginals \( \mu_1 \) and \( \mu_2 \), then we define
  \begin{equation}
    \label{eq:322}
    W_p^p(\mu_1, \mu_2) := \min \left\{ \int_{\Omega^2} \left| x_1 - x_2 \right|^p \diff \bm{\mu}(x_1, x_2) : \bm{\mu} \in \Gamma(\mu_1, \mu_2) \right\},
  \end{equation}
  the \emph{Wasserstein}-\( p \) distance between \( \mu_1 \) and \( \mu_2 \).

\end{definition}

\begin{definition}[Uniform integrability]
  A measurable function \( f : \Omega \rightarrow [0,\infty] \) is \emph{uniformly integrable} with respect to a family of finite measures \( \left\{ \mu_i : i \in I \right\} \), if
  \begin{equation*}
    \lim_{M \rightarrow \infty} \sup_{i \in I} \int_{\left\{ f(x) \geq M\right\}} f(x) \diff \mu_i(x) = 0.
  \end{equation*}
\end{definition}

\begin{lemma}[Topology of Wasserstein spaces]
  \label{lem:26}
  \cite[Proposition 7.1.5]{AGS08}
  For \( p  \geq 1\) and a subset \( \Omega \subseteq \mathbb{R}^d \), \( \mathcal{P}_p (\Omega) \) endowed with the Wasserstein-\( p \) distance is a separable metric space which is complete if \( \Omega \) is closed. A set \(  \mathcal{K} \subseteq \mathcal{P}_p(\Omega) \) is relatively compact if and only if it is \( p \)-uniformly integrable (and hence tight by Lemma \ref{lem:24} just below). In particular, for a sequence \( (\mu_n)_{n \in \mathbb{N}} \subseteq \mathcal{P}_p(\Omega) \), the following properties are equivalent:
\begin{itemize}
\item[(i) ] $ \lim_{n\rightarrow\infty} W_p(\mu_n, \mu) = 0$;
\item[(ii) ] $\mu_n \rightarrow \mu$ narrowly and  $(\mu_n)_n$ has uniformly integrable $p$-moments.
\end{itemize}
  
\end{lemma}

\begin{lemma}[Continuity of integral functionals]
  \label{lem:4}
  \cite[Lemma 5.1.7]{AGS08}
  Let \( (\mu_n)_{n \in \mathbb N} \) be a sequence in  $\mathcal{P}(\Omega)$  converging narrowly to \( \mu \in \mathcal{P}(\Omega) \), \( g : \Omega\rightarrow \mathbb{R} \) lower semi-continuous and \( f : \Omega \rightarrow \mathbb{R} \) continuous. If \( \left| f \right|, g^- := -\min \left\{ g, 0 \right\} \) are uniformly integrable \wrt \( \left\{ \mu_{n}: n \in \mathbb N \right\} \), then
  \begin{align*}
    \liminf_{n\rightarrow\infty} \int_{\Omega} g(x) \diff \mu_n(x) \geq {} & \int_{\Omega} g(x) \diff \mu(x)\\
    \lim_{n\rightarrow\infty} \int_{\Omega} f(x) \diff \mu_n(x) = {} & \int_{\Omega} f(x) \diff \mu(x)
  \end{align*}
\end{lemma}

\begin{lemma}[Uniform integrability of moments]
  \label{lem:24}
  \cite[Corollary to Theorem 25.12]{95-Billingsley-Proba_and_Measuer}
  Given \( r > 0 \) and a family \( \left\{ \mu_i :i \in I \right \} \) of probability measures in $\mathcal{P}(\Omega)$ with
  \begin{equation*}
    \sup_{i \in I} \int_{\Omega} \left| x \right|^r \diff \mu_i(x) < \infty,
  \end{equation*}
  then the family \( \left\{ \mu_{i} : i \in I \right \} \) is tight and for all \( 0 < q < r \), \( x \mapsto \left| x \right|^q \) is uniformly integrable \wrt \( \left\{ \mu_i : i \in I \right\} \). 
\end{lemma}
\begin{proof} 
  For the uniform integrability, let \( M > 0 \). By the monotonicity of the power functions \( t \mapsto t^p \) for \( t > 0 \) and \( p > 0 \), we have
  \begin{align*}
    \int_{\left\{ \left| x \right|^q \geq M \right\}} \left| x \right|^q \diff \mu_i = {} & \int_{\left\{ \left| x \right|^q \geq M \right\}} \left| x \right|^q \frac{M^{(r-q)/q}}{M^{(r-q)/q}} \diff \mu_i \\
    \leq {} & M^{-(r-q)/q} \int_{\left\{ \left| x \right|^q \geq M \right\}} \left| x \right|^r \diff \mu_i \\
    \leq {} & M^{-(r-q)/q} \int_{\Omega} \left| x \right|^r \diff \mu_i \rightarrow 0,
  \end{align*}
  for \( M \rightarrow \infty \), uniformly in \( i \in I \).

  Similarly, for the tightness, 
  \begin{equation*}
    \mu_i\left( \left\{ \left| x \right| \geq M \right\} \right) \leq M^{-r} \int_{\Omega} \left| x \right|^r \diff \mu_i(x) \rightarrow 0
  \end{equation*}
  for \( M \rightarrow \infty \).
\end{proof}

\subsection{Examples and counterexamples to existence of minimizers for discordant powers $q_a \neq q_r$}

We recall the definition of \( \mathcal{E} \):
\begin{equation*}
    \mathcal{E}[\mu] := \int_{\Omega\times\Omega} \psi_a(x-y) \diff \omega(x) \diff \mu(y) - \frac{1}{2} \int_{\Omega\times\Omega} \psi_r(x-y)  \diff \mu(x) \diff \mu(y),
\end{equation*}
for \( \omega \), \( \mu \in \mathcal{P}_2(\Omega) \) (at least for now) and
\begin{equation*}
    \psi_a(x) := \left| x \right|^{q_a}, \quad \psi_r(x) := \left| x \right|^{q_r}, \quad x \in \mathbb{R}^d,
\end{equation*}
where \( q_a \), \( q_r \in [1,2] \). Furthermore, denote for a vector-valued measure \( \nu \) its \emph{total variation} by \( \left| \nu \right| \) and by \( BV(\Omega) \) the space of functions \( f \in L^1_{\mathrm{loc}}(\Omega) \) whose distributional derivatives \( Df \) are finite Radon measures. With abuse of terminology, we call \( \left| Df \right|(\Omega) \) the total variation of \( f \). Now, we define the \emph{total variation regularization} of \( \mathcal{E} \) by
\begin{equation*}
  \mathcal{E}^\lambda[\mu] := \mathcal{E}[\mu] + \lambda \left| D\mu \right|(\Omega),
\end{equation*}
where \( \mu \in \mathcal{P}_2(\Omega) \cap BV(\Omega) \).

We shall briefly state some results which are in particular related to the asymmetric case of \( q_a \) and \( q_r \)  not necessarily being equal.

\subsubsection{Situation on a compact set}
\label{sec:situ-comp-set}

From now on, let \( q_a, q_r \in [1,2] \).

\begin{proposition}
  Let \(\Omega \)  be a compact subset of \(\mathbb{R}^{d}\). Then, the functionals \(\mathcal{E}\) and \(\mathcal{E}^{\lambda}\) are well-defined on \(\mathcal{P}(\Omega)\) and  \(\mathcal{P}(\Omega) \cap BV(\Omega)\), respectively, and \(\mathcal{E}\) admits a minimizer.

  If additionally \(\Omega\) is an extension domain, then \(\mathcal{E}^{\lambda}\) admits a minimizer as well.
\end{proposition}

\begin{proof}
  Note that since the mapping
  \begin{equation}
    \label{eq:15}
    (x,y) \mapsto \left| y - x \right|^{q},\quad x,y \in \mathbb{R}^{d},
  \end{equation}
  is jointly continuous in \(x\) and \(y\), it attains its maximum on the compact set \(\Omega \times \Omega\). Hence, the kernel \eqref{eq:15} is a bounded continuous function, which, on the one hand, implies that the functional \(\mathcal{E}\) is bounded (and in particular well-defined) on \( L^1(\Omega) \) and on the other hand that it is continuous with respect to the narrow topology. Together with the compactness of \( \mathcal{P}(\Omega) \), this implies existence of a minimizer for \(\mathcal{E}\).

  The situation for \(\mathcal{E}^{\lambda}\) is similar. Due to the boundedness of \(\Omega\) and the regularity of its boundary, sub-levels of \( \left| D\, \cdot\, \right|(\Omega) \) are relatively compact in \( L^{1}(\Omega) \cap \mathcal{P}(\Omega) \) by \cite[Chapter 5.2, Theorem 4]{92-Evans-fine-properties}. As the total variation is lower semi-continuous with respect to \(L^{1}\)-convergence by \cite[Chapter 5.2, Theorem 1]{92-Evans-fine-properties} and \(L^{1}\)-convergence implies narrow convergence, we get lower semi-continuity of \( \mathcal{E}^\lambda \) and therefore again existence of a minimizer.
\end{proof}


\subsubsection{Existence of minimizers for stronger attraction on arbitrary domains}
\label{sec:exist-minim-strong}

Note that from here on, the constants \( C \) and \( c \) are generic and may change in each line of a calculation.
In the following we shall make use of the following elementary inequalities: for \( q \geq 1 \) and \( x,y \in \mathbb{R}^d \), there exist \( C,c > 0 \) such that
  \begin{equation}
    \label{eq:16}
    \left| x + y \right|^q \leq C \left( \left| x \right|^q + \left| y \right|^q \right),
  \end{equation}
  and
  \begin{equation}
    \label{eq:17}
    \left| x - y \right|^q \geq  \left( c \left| x \right|^q - \left| y \right|^q \right).
  \end{equation}

\begin{theorem}
  \label{thm:exist-min-strong}
  Let \( q_a, q_r \in [1,2] \), \( \Omega \subseteq \mathbb{R}^d \) closed and \( q_a > q_r \). If \( \omega \in \mathcal{P}_{q_a}(\Omega) \), then the sub-levels of \( \mathcal{E} \) have uniformly bounded \( q_a \)th moments and \( \mathcal{E} \) admits a minimizer on \( \mathcal{P}_{q_r}(\Omega) \).
\end{theorem}

\begin{proof}
 
  \emph{Ad moment bound:} Let \( \mu \in \mathcal{P}_{q_r}(\Omega) \). By estimate \eqref{eq:17}, we have
  \begin{align}   
    \mathcal{V}[\mu] = {} & \int_{\Omega \times \Omega} \left| x - y \right|^{q_a} \diff \mu(x) \diff \omega(y) \nonumber \\
    \geq {} &  \int_{\Omega \times \Omega} \left( c \left| x \right|^{q_a} - \left| y \right|^{q_a} \right) \diff \mu(x) \diff \omega(x)\nonumber \\
    = {} &  c \int_{\Omega} \left| x \right|^{q_a} \diff \mu(x) - \int_{\Omega} \left| y \right|^{q_a} \diff \omega(y). \label{eq:23}
  \end{align}
  On the other hand, by estimate \eqref{eq:16}
  \begin{align}
    \mathcal{W}[\mu] = {} & - \frac{1}{2} \int_{\Omega \times \Omega} \left| x - y \right|^{q_r} \diff \mu(x) \diff \mu(y) \nonumber \\
    \geq {} & - C \int_{\Omega \times \Omega} \left( \left| x \right|^{q_r} + \left| y \right|^{q_r} \right) \diff \mu(x) \diff \mu(y)\nonumber \\
    \geq {} & - C \int_{\Omega} \left| x \right|^{q_r} \diff \mu(x).\label{eq:25}
  \end{align}
  Combining \eqref{eq:23} and \eqref{eq:25}, we obtain
  \begin{align}
    \mathcal{E}[\mu] + \int_{\Omega} \left| x \right|^{q_a} \diff \omega(x) \geq {} &\int_{\Omega} \left( c \left| x \right|^{q_a} - C \left| x \right|^{q_r} \right) \diff \mu(x) \nonumber \\
    \geq {} &\int_{\Omega} \left( c - C \left| x \right|^{q_r - q_a} \right) \left| x \right|^{q_a} \diff \mu(x). \nonumber
  \end{align}
  Since \( q_a > q_r \), there is an \( M > 0 \) such that
  \begin{equation}
    c - C \left| x \right|^{q_r - q_a} \geq \frac{c}{2}, \quad \left| x \right| \geq M,  \nonumber
  \end{equation}
  and hence
  \begin{align}
    \int_{\Omega} \left| x \right|^{q_a} \diff \mu(x) = {} &   \left[ \int_{B_M(0)} \left| x \right|^{q_a} \diff \mu(x) + \int_{\Omega \setminus B_M(0)} \left| x \right|^{q_a} \diff \mu(x) \right]  \nonumber\\
    \leq {} &  M^{q_a} + \frac{2}{c} \left [\mathcal{E}[\mu] + \int_{\Omega} \left| x \right|^{q_a} \diff \omega(x) \right] \label{eq:358}
  \end{align}
As we can show that the sub-levels of \( \mathcal{E} \) have a uniformly bounded \( q_a \)th moment, so that they are also Wasserstein-\( q \) compact for any \( q < q_a \) by Lemma \ref{lem:26} and Lemma \ref{lem:24}, given a minimizing sequence, we can extract a narrowly converging subsequence \( \left( \mu_n \right)_n \) with uniformly integrable \( q_r \)th moments. With respect to that convergence, which  also implies the narrow convergence of \( \left( \mu_n \otimes \mu_n \right)_n \) and \( \left( \mu_n \otimes \omega \right)_n \), the functional \( \mathcal{W} \) is continuous and the functional \( \mathcal{V} \) is lower semi-continuous by Lemma \ref{lem:4}, so we shall be able to apply the direct method of calculus of variations to show existence of a minimizer in \( \mathcal{P}_{q_r}(\Omega) \).

\end{proof}

\subsubsection{Counterexample to the existence of minimizers for stronger repulsion}
\label{sec:absence-minim-strong}

Now, let \(q_{a},q_{r} \in [1,2]\) with \( q_r > q_a \). On \(\Omega=\mathbb{R}^{d}\), this problem need not have a minimizer.

\begin{example}[Nonexistence of minimizers for stronger repulsion]
  \label{exa:count-exist-minim}
  Let \(\Omega=\mathbb{R}\), \(q_{r} > q_{a}\), \(\omega = \mathcal{L}^{1}\llcorner{[-1,0]}\) and
  consider the sequence \(\mu_{n} := n^{-1}\mathcal{L}^{1}\llcorner{[0,n]}\). Computing
  the values of the functionals used to define \(\mathcal{E}\) and
  \(\mathcal{E}^{\lambda}\) yields
  \begin{align}
    \mathcal{V}[\mu_{n}] = {} &\frac{1}{n} \int_{-1}^{0} \int_{0}^{n} \left| y -
      x \right|^{q_{a}} \diff x \diff y \nonumber \\
    \leq {} &\frac{1}{n} \int_{0}^{n} (y+1)^{q_{a}} \diff y \nonumber \\
    = {} &\frac{1}{n (q_{a}+1)} \left(n+1\right)^{q_{a}+1} - \frac{1}{n (q_{a}+1)}\nonumber \\
    \leq {} &\frac{\left(n+1\right)^{q_{a}}}{q_{a}+1};\nonumber \\
    \mathcal{W}[\mu_{n}] = {} & - \frac{1}{2n^{2}} \int_{0}^{n}
    \int_{0}^{n}\left| y-x \right|^{q_{r}} \diff x \diff y \nonumber \\
    = {} & - \frac{1}{2n^{2}(q_{r}+1)} \int_{0}^{n} \left[(n - y)^{q_{r}+1} +
      y^{q_{r}+1}\right] \diff y\nonumber \\
    = {} & - \frac{1}{2n^{2}(q_{r}+1)(q_{r}+2)} 2n^{q_{r}+2} =
    \frac{n^{q_{r}}}{(q_{r}+1)(q_{r}+2)};\nonumber \\
    \left\| D\mu_{n} \right\| = {} &- \frac{2}{n}. \nonumber
  \end{align}
  By considering the limit of the corresponding sums, we see that
  \begin{equation}
    \mathcal{E}[\mu_{n}]\rightarrow -\infty, \quad \mathcal{E}^{\lambda}[\mu_{n}]\rightarrow -\infty \quad \text{for } n\rightarrow \infty,\nonumber
  \end{equation}
  which means that there are no minimizers in this case.
\end{example}

\section{Properties of the functional on $\mathbb{R}^d$}
\label{sec:prop-funct-mathbbrd}

Now, let us consider \(\Omega=\mathbb{R}^{d}\) and 
\begin{equation}
  \label{eq:317}
  q := q_{a} = q_{r}, \quad \psi(x) := \psi_a(x) = \psi_r(x) = \left| x \right|^q, \quad x \in \mathbb{R}^d,
\end{equation}
for \( 1 \leq q < 2 \).

Here, neither the well-definedness of \(\mathcal{E}[\mu]\) for all \(\mu\in\mathcal{P}(\mathbb{R}^{d})\) nor the narrow compactness of the sub-levels as in the case of a compact \( \Omega \) in Section \ref{sec:situ-comp-set} are clear, necessitating additional conditions on \( \mu \) and \( \omega \). For example, if we assume the finiteness of the second moments, i.e., \(\mu,\omega \in \mathcal{P}_{2}(\mathbb{R}^{d})\), we can a priori see that both \(\mathcal{V}[\mu]\) and \(\mathcal{W}[\mu]\) are finite.

Under this restriction, we shall show a formula for \(\mathcal{E}\) involving the Fourier-Stieltjes transform of the measures \(\mu\) and \(\omega\). Namely, there is a constant \(C = C(q,\omega) \in \mathbb{R}\) such that
\begin{equation}
  \label{eq:31}
  \mathcal{E}[\mu] + C = -2^{-1} (2\pi)^{-d}\int_{\mathbb{R}^d}\left| 
    \widehat{\mu}(\xi) - \widehat{\omega}(\xi) \right|^2
  \widehat{\psi}(\xi) \dd \xi =: \widehat{\mathcal{E}}[\mu],
\end{equation}
where for any \(\mu \in \mathcal{P}(\mathbb{R}^{d})\), \(\widehat{\mu}\) denotes its Fourier-Stieltjes transform,
\begin{equation}
  \label{eq:32}
  \widehat{\mu}(\xi) = \int_{\mathbb{R}^{d}} \exp(-\i x^{T}\xi) \dd \mu(x),
\end{equation}
and \(\widehat{\psi}\) is the \emph{generalized Fourier-transform of \(\psi\)}, \ie, a Fourier transform with respect to a certain duality, which allows to cancel the singularities of the Fourier transform of the kernel $\psi$ at $0$. We have gathered most of the important facts about it in Appendix \ref{cha:cond-posit-semi}. In this case, it can be explicitly computed to be
\begin{equation}
  \label{eq:33}
      \widehat{\psi}(\xi) := -2\cdot(2\pi)^{d} D_{q} \, \left| \xi \right|^{-d-q}, \quad \text{with a } D_{q} > 0,
\end{equation}
where
\begin{equation}
  D_{q} := -(2\pi)^{-d/2}\frac{2^{q + d/2}\, \Gamma((d+q)/2)}{2\Gamma(-q/2)} > 0,
\end{equation}
so that
\begin{equation}
  \label{eq:34}
  \widehat{\mathcal{E}}[\mu] = D_{q} \int_{\mathbb{R}^d}\left| 
    \widehat{\mu}(\xi) - \widehat{\omega}(\xi) \right|^2
  \left| \xi \right|^{-d-q} \dd \xi,
\end{equation}
which will be proved in Section \ref{sec:fourier-formula-p2}.

Notice that, while $\mathcal E$ might not be well-defined on $\mathcal P(\mathbb R^d)$, formula \eqref{eq:34} makes sense on the whole space \(\mathcal{P}(\mathbb{R}^{d})\) and the sub-levels of \( \widehat{\mathcal{E}} \) can be proved to be narrowly compact as well as lower semi-continuous \wrt the narrow topology (see Proposition \ref{prp:compctness-sublvls}), motivating the proof in Section \ref{sec:extension-p} that up to a constant, this formula is exactly the lower semi-continuous envelope of \(\mathcal{E}\) on \(\mathcal{P}(\mathbb{R}^{d})\) endowed with the narrow topology.

\subsection{Fourier formula in $\mathcal{P}_{2}(\mathbb{R}^{d})$}
\label{sec:fourier-formula-p2}

Assume that \(\mu,\omega \in \mathcal{P}_{2}(\mathbb{R}^{d})\) and observe that by using the symmetry of $\psi$, $\mathcal{E}[\mu]$ can be written as
\begin{align}
  \mathcal{E}[\mu] = {} & -\frac{1}{2}\int_{\mathbb{R}^d \times \mathbb{R}^d} \psi(y - x) \dd \mu(x) \dd \mu(y)
  + 
  \frac{1}{2}\int_{\mathbb{R}^d \times \mathbb{R}^d} \psi(y - x) \dd \omega(x) \dd \mu(y) \nonumber \\
  & +\frac{1}{2}\int_{\mathbb{R}^d \times \mathbb{R}^d} \psi(y - x) \dd \omega(y) \dd \mu(x) -
  \frac{1}{2}\int_{\mathbb{R}^d \times \mathbb{R}^d} \psi(y - x) \dd \omega(x) \dd \omega(y)  \nonumber  \\
  & +\frac{1}{2}\int_{\mathbb{R}^d \times \mathbb{R}^d} \psi(y - x) \dd \omega(x) \dd \omega(y)  \nonumber  \\
  = {} &- \frac{1}{2}\int_{\mathbb{R}^d \times \mathbb{R}^d} \psi(y - x) \dd [\mu - \omega](x) \dd [\mu - \omega](y) +
  C,
\end{align}
where
\begin{equation}
  \label{eq:36}
  C = \frac{1}{2}\int_{\mathbb{R}^d \times \mathbb{R}^d} \psi(y - x) \dd \omega(x) \dd \omega(y).
\end{equation}
In the following, we shall mostly work with the symmetrized variant and denote it
by
\begin{equation}
  \label{eq:37}
  \widetilde{\mathcal{E}}[\mu] := - \frac{1}{2}\int_{\mathbb{R}^d \times \mathbb{R}^d} \psi(y - x) \dd [\mu -
  \omega](x) \dd [\mu - \omega](y).
\end{equation}

\subsubsection{Representation for point-measures}
\label{sec:four-repr-point}

Our starting point is a Fourier-type representation of $\widetilde{\mathcal{E}}$ in the case
where $\mu$ and $\omega$ are atomic measures, which has been derived in \cite{Wend05}.

\begin{lemma}
  Let $\mu$ and $\omega$ be a linear combination of Dirac measures so that
  \begin{equation*}
    \mu - \omega = \sum_{j = 1}^N \alpha_j \delta_{x_j},
  \end{equation*}
  for a suitable $N \in \mathbb{N}$, $\alpha_j \in \mathbb{R}$, and pairwise distinct $x_j \in \mathbb{R}^d$ for all $j = 1,\ldots,N$.
  Then
  \begin{equation}
    \label{eq:39}
    \widetilde{\mathcal{E}}[\mu] = -2^{-1}(2\pi)^{-d}\int_{\mathbb{R}^d}\left| \sum_{j = 1}^N \alpha_j
      \exp(\i x_j^{T}\xi) \right|^2 \widehat{\psi}(\xi) \dd \xi,
  \end{equation}
  where
  \begin{equation*}
    \widehat{\psi}(\xi) := -2\cdot(2\pi)^{d} D_{q} \, \left| \xi \right|^{-d-q}, \quad \text{with a } D_{q} > 0.
  \end{equation*}
\end{lemma}

\begin{proof}
  The claim is an application of a general representation theorem for conditionally positive semi-definite functions. An extensive introduction can be found in \cite{Wend05}, of which we have included a brief summary in Appendix \ref{cha:cond-posit-semi} for the sake of completeness. Here, we use Theorem \ref{thm:repr-thm-cond-semi} together with the explicit computation of the generalized Fourier transform of \( \psi \) in Theorem \ref{thm:cond-ft-power}.
\end{proof}

\begin{remark}
  By $\overline{\exp(\i x)} = \exp(-\i x)$, for $x \in \mathbb{R}$,  we can also write the above formula \eqref{eq:39} as
  \begin{equation*}
    \widetilde{\mathcal{E}}[\mu] = D_q\int_{\mathbb{R}^d}\left| \widehat{\mu}(\xi) - \widehat{\omega}(\xi) \right|^2 \left| \xi \right|^{-d-q} \dd \xi, \quad \xi \in \mathbb{R}^d.
  \end{equation*}
\end{remark}

\subsubsection{Point approximation of probability measures by the empirical process}
\label{sec:point-appr-empir}

\begin{lemma}[Consistency of empirical process]
  \label{lem:3}
  Let \(\mu \in \mathcal{P}(\mathbb{R}^{d})\) and \((X_{i})_{i \in \mathbb{N}}\) be  a sequence of i.i.d. random variables with \(X_{i} \sim \mu\) for all \(i \in \mathbb{N}\). Then the empirical distribution
  \begin{equation*}
    \mu_{N} := \frac{1}{N} \sum_{i=1}^{N} \delta_{X_{i}}
  \end{equation*}
  converges with probability \(1\) narrowly to \(\mu\), i.e.,
  \begin{equation*}
    P(\{\mu_{N}\rightarrow \mu \text{ narrowly}\}) = 1.
  \end{equation*}
  
  Additionally, if for a \(p \in [1,\infty)\), \(\int_{\mathbb{R}^{d}}\left| x \right|^{p}\dd \mu < \infty\), then \( x \mapsto \left| x \right|^p \) is almost surely uniformly integrable \wrt \( \left\{ \mu_N: N  \in \mathbb N \right \}\), which by Lemma \ref{lem:26} implies almost sure convergence of \( \mu_N \rightarrow \mu \) in the \( p \)-Wasserstein topology.
\end{lemma}

\begin{proof}
  By the metrizability of narrow convergence as in \eqref{eq:318}, it is sufficient to prove convergence of the integral functionals associated to a sequence of bounded continuous functions \( \left( f_k \right)_{k \in \mathbb{N}} \). But
  \begin{equation*}
    \int_{\mathbb{R}^d} f_k(x) \diff \mu_N(x) = \frac{1}{N}\sum_{i = 1}^{N} f_k(X_i) \xrightarrow{N\rightarrow\infty} E[f_k(X)] = \int_{\mathbb{R}^d} f_k(x) \diff \mu(x),
  \end{equation*}
  almost surely by the strong law of large numbers, \cite[Theorem 2.4.1]{Dur10}, leading to construction of a countable collection of null sets \( A_k \) where the above convergence fails. Since a countable union of null sets is again a null set, the first claim follows.

  For the second claim, we apply the strong law of large numbers to the functions \( f_M(x) := \left| x \right|^p \cdot 1_{\left\{ \left| x \right|^p \geq M \right\}}\) for \( M > 0 \) to get the desired uniform integrability: for a given \( \varepsilon > 0 \), choose \( M > 0 \) large enough such that
  \begin{equation*}
    \int_{\mathbb{R}^d} f_M(x) \diff \mu(x) < \frac{\varepsilon}{2},
  \end{equation*}
  and then \( N_0 \in \mathbb{N} \) large enough such that
  \begin{equation*}
    \left| \int_{\mathbb{R}^d} f_M(x) \diff \mu_N(x) - \int_{\mathbb{R}^d} f_M(x) \diff \mu(x) \right| < \frac{\varepsilon}{2}, \quad N \geq N_0, \text{ almost surely}.
  \end{equation*}
  Now we choose \( M' \geq M \) sufficiently large so to ensure that \( \left| X_{i} \right|^p < M' \) almost surely for all \( i < N_0 \). By the monotonicity of \( \int_{\mathbb{R}^d} f_M(x) \diff \mu(x) \) in \( M \), this ensures
  \begin{align*}
    \sup_{N\in\mathbb{N}} \int_{\mathbb{R}^d} f_{M'}(x) \diff \mu_N = {} & \sup_{N \geq N_0} \int_{\mathbb{R}^d} f_{M'}(x) \diff \mu_N \leq {} \sup_{N\geq N_0} \int_{\mathbb{R}^d} f_M(x) \diff \mu_N(x) \\
    < {} & \frac{\varepsilon}{2} + \frac{\varepsilon}{2} = \varepsilon
  \end{align*}
\end{proof}

\subsubsection{Representation for $\mathcal{P}_{2}(\mathbb{R}^{d})$}
\label{sec:four-repr-gener}

Now we establish continuity in both sides of \eqref{eq:39} with respect to the
\( 2 \)-Wasserstein-convergence to obtain \eqref{eq:31} in $\mathcal P_2(\mathbb R^d)$.

\begin{lemma}[Continuity of \(\widetilde{\mathcal{E}}\)]
  \label{lem:5}
  Let
  \begin{equation*}
    \mu_{k} \rightarrow  \mu, \quad \omega_{k} \rightarrow  \omega \quad     \text{for } k\rightarrow \infty \text{ in } \mathcal{P}_{2}(\mathbb{R}^{d}),
  \end{equation*}
\wrt the \( 2 \)-Wasserstein-convergence.
  Then,
  \begin{align}
    \label{eq:54}
    \leadeq\int_{\mathbb{R}^d \times \mathbb{R}^d} \psi(y - x) \dd [\mu_k - \omega_k](x) \dd [\mu_k - \omega_k](y) \\
    \rightarrow  {} &\int_{\mathbb{R}^d \times \mathbb{R}^d} \psi(y - x) \dd [\mu - \omega](x) \dd [\mu - \omega](y), \quad
    \text{for } k \rightarrow \infty.
  \end{align}
\end{lemma}

\begin{proof}
  By the particular choice of $\psi$, we have the estimate
  \begin{equation*}
    \left| \psi(y - x) \right| \leq C(1 + \left| y - x \right|^2) \leq 2C(1 + \left| x \right|^2 +
    \left| y \right|^2).
  \end{equation*}
  After expanding the expression to the left of \eqref{eq:54} so that we only have to deal with integrals with respect to probability measures, we can use this estimate to get the uniform integrability of the second moments of \( \mu \) and \( \omega \) by Lemma \ref{lem:26} and are then able to apply Lemma \ref{lem:4} to obtain convergence.
\end{proof}

\begin{lemma}[Continuity of \(\widehat{\mathcal{E}}\)]
  \label{lem:6}
  Let
  \begin{equation*}
    \mu_{k}\rightarrow \mu,\quad \omega_{k}\rightarrow \omega \quad \text{for } k\rightarrow \infty\, \text{ in } \mathcal{P}_{2}(\mathbb{R}^{d}),
  \end{equation*}
\wrt the \( 2 \)-Wasserstein-convergence,
  such that
  \begin{equation*}
    \mu_k - \omega_k = \sum_{j = 1}^{N_k} \alpha_j^k \delta_{x_j^k}
  \end{equation*}
  for suitable $\alpha_j^k \in \mathbb{R}$ and pairwise distinct $x_j^k \in \mathbb{R}^{d}$. Then,
  \begin{align*}
    \leadeq{\int_{\mathbb{R}^d}\left| \sum_{j = 1}^{N_k} \alpha_j^k \exp(\i \xi \cdot x_j^k) \right|^2
    \widehat{\psi}(\xi) \dd \xi} \\ \rightarrow  {}
    &\int_{\mathbb{R}^d}\left| \int_{\mathbb{R}^d} \exp(\i \xi \cdot x) \dd [\mu - \omega](x) \right|^2
    \widehat{\psi}(\xi) \dd \xi \quad \text{for } k\rightarrow \infty.
  \end{align*}
\end{lemma}

\begin{proof}
  By the narrow convergence of $\mu_k$ and $\omega_k$, we get pointwise convergence of the Fourier transforms, i.e.,
  \begin{equation*}
    \sum_{j = 1}^{N_k} \alpha_j^k \exp(\i \xi \cdot x_j^k) \rightarrow  \int_{\mathbb{R}^d} \exp(\i \xi \cdot x) \dd [\mu - \omega](x) \quad \text{for all } \xi \in \mathbb{R}^d \text{ and } k\rightarrow \infty.
  \end{equation*}
  We want to use the dominated convergence theorem: The Fourier transform  of \(\mu - \omega\) is bounded in \(\xi\), so that the case $\xi\rightarrow \infty$ poses no problem due to the integrability of $\widehat{\psi}(\xi) = C \left| \xi \right|^{-d-q}$ away from $0$. In order to justify the necessary decay at $0$, we use the control of the first moments (since we even control the second moments by the \(\mathcal{P}_{2}\) assumption): Inserting the Taylor expansion of the exponential function of order \(0\),
  \begin{equation*}
    \exp(\i \xi \cdot x) = 1 + \i \xi \cdot x \int_0^1 \exp(\i \xi \cdot tx) \dd t,
  \end{equation*}
  into the expression in question and using the fact that \(\mu_{k}\) and \(\omega_{k}\) are probability measures results in
  \begin{align*}
    \leadeq{\left| \int_{\mathbb{R}^d} \exp(\i \xi \cdot x) \dd [\mu_k - \omega_k](x) \right| }\\
    = {} & \left| \int_{\mathbb{R}^d} \left( 1 + \i \xi \cdot x \int_0^1 \exp(\i \xi \cdot tx) \dd t \right) \dd
      [\mu_k - \omega_k](x) \right|\\
    = {} & \left| \int_{\mathbb{R}^d} \left( \i \xi \cdot x \int_0^1 \exp(\i \xi \cdot tx) \dd t \right) \dd [\mu_k - \omega_k](x) \right|\\
    \leq {} &|\xi| \underbrace{\left(\int_{\mathbb{R}^d} |x| \dd \mu_k(x) + \int_{\mathbb{R}^d} |x|\dd \omega_k(x) \right)}_{:= C}.
  \end{align*}
  Therefore, we have a $k$-uniform bound $C$ such that
  \begin{equation*}
    \left| \sum_{j = 1}^{N_k} \alpha_j^k \exp(\i \xi \cdot x_j^k) \right|^2 \leq C \, |\xi|^2,
  \end{equation*}
  compensating the singularity of $\widehat{\psi}$ at the origin, hence together
  with the dominated convergence theorem proving the claim.
\end{proof}

Combining the two Lemmata above with the approximation provided by Lemma \ref{lem:3} yields

\begin{corollary}[Fourier-representation for $\widetilde{\mathcal{E}}$ on
  \(\mathcal{P}_{2}(\mathbb{R}^{d})\)]
  \label{cor:four-repr-widet}
  \begin{equation*}
    \widetilde{\mathcal{E}}[\mu] = \widehat{\mathcal{E}}[\mu], \quad \mu \in \mathcal{P}_{2}(\mathbb{R}^{d}).
  \end{equation*}
\end{corollary}

\subsection{Extension to $\mathcal{P}(\mathbb{R}^{d})$}
\label{sec:extension-p}

While the well-definedness of \(\mathcal{E}[\mu]\) is not clear for all \(\mu \in \mathcal{P}(\mathbb{R}^{d})\), since the sum of two integrals with values \( \pm \infty \) may occur instead, for each such \(\mu\) we can certainly assign a value in \(\mathbb{R}\cup \left\{ \infty \right\}\) to \(\widehat{\mathcal{E}}[\mu]\). In the following, we want to justify in which sense it is possible to consider \(\widehat{\mathcal{E}}\) instead of the original functional, namely that \( \widehat{\mathcal{E}} \) is, up to an additive constant, the \emph{lower semi-continuous envelope} of \( \mathcal{E} \).

Firstly, we prove that \(\widehat{\mathcal{E}}\) has compact sub-levels in \(\mathcal{P}(\mathbb{R}^{d})\) endowed with the narrow topology, using the following lemma as a main ingredient.

\begin{lemma}
  \label{lem:7}
   Given a probability measure \(\mu \in \mathcal{P}(\mathbb{R}^{d})\) with Fourier transform
  \(\widehat{\mu}\colon \mathbb{R}^{d} \rightarrow  \mathbb{C}\), there are \(C_{1} = C_{1}(d) > 0\) and \(C_{2} = C_{2}(d) > 0\) such that for all
  \(u > 0\),
  \begin{equation}
    \label{eq:64}
    \mu\left(\left\{x : \left| x \right| \geq u^{-1}\right\}\right) \leq
    \frac{C_{1}}{u^{d}}\int_{\left| \xi \right|
      \leq C_{2} u} (1 - \Re \widehat{\mu}(\xi)) \dd \xi.
  \end{equation}
\end{lemma}

\begin{proof}
The proof for the case \( d = 1 \) can be found in  \cite[Theorem 3.3.6]{Dur10} and we generalize it below to any $d \geq 1$.
  Let \(u > 0\). Firstly, note that
  \begin{equation*}
    1 - \Re \widehat{\mu}(\xi) = \int_{\mathbb{R}^{d}}(1 - \cos(\xi \cdot x)) \dd \mu(x) 
    \geq 0 \quad \text{for all } \xi \in \mathbb{R}^{d}.
  \end{equation*}
  By starting with the integral on the right-hand side of \eqref{eq:64} (up to a constant in the integration domain) and using Fubini-Tonelli as well as integration in spherical coordinates, we get
  \begin{align}
    \leadeq{\int_{\left| \xi \right|\leq u} (1 - \Re \widehat{\mu}(\xi)) \dd \xi} \nonumber \\
    = {} & \int_{\mathbb{R}^{d}} \int_{\left| \xi \right| \leq u} (1 - \cos(\xi \cdot x)) \dd \xi \dd \mu(x) \nonumber\\
    = {} & \int_{\mathbb{R}^{d}} \int_{\left| \widetilde{\xi} \right| = 1} \int_{0}^{u} (1 - \cos(r\widetilde{\xi} \cdot x)) r^{d-1}\dd r \dd \sigma(\widetilde{\xi}) \dd \mu(x) \label{eq:356}\\
    = {} & \int_{\mathbb{R}^{d}} \int_{\left| \widetilde{\xi} \right| = 1} \left[ \frac{u^d}{d} - \int_{0}^{u} \cos(r\widetilde{\xi} \cdot x) r^{d-1}\dd r \right] \dd \sigma(\widetilde{\xi}) \dd \mu(x)\label{eq:67}
  \end{align}
  If \(d \geq 2\), integrating the integral over \(\cos(r\widetilde{\xi} \cdot x)r^{d-1}\) in
  \eqref{eq:67} by parts yields
  \begin{align*}
    \int_{0}^{u} \cos(r\widetilde{\xi} \cdot x) r^{d-1}\dd r = {} &
    \frac{\sin(u\widetilde{\xi} \cdot x) u^{d-1}}{\widetilde{\xi} \cdot x} - (d-1)
    \int_{0}^{u}\frac{\sin(r\widetilde{\xi} \cdot x)}{\widetilde{\xi} \cdot x}r^{d-2} \dd r,
  \end{align*}
  which can also be considered true for \(d = 1\) if the second part is assumed to be zero because of the factor \((d-1)\).

  We now prove \eqref{eq:64} by estimating the integrand in \eqref{eq:67} suitably from below. Using \(\left| \sin(x) \right| \leq 1\) for all \(x \in \mathbb{R}\) and dividing by \(u^{d}\), we get
  \begin{align*}
    \leadeq{d^{-1} - u^{-d}\int_{0}^{u} \cos(r\widetilde{\xi} \cdot x) r^{d-1}\dd r} \\
    = {} & d^{-1} - \frac{\sin(u \widetilde{\xi} \cdot x)}{u\widetilde{\xi} \cdot x} + \frac{(d-1)}{u^{d}} \int_{0}^{u}\frac{\sin(r\widetilde \xi \cdot x)}{\widetilde{\xi} \cdot x} r^{d-2}\dd r\\
    \geq {} & d^{-1} - \frac{1}{u\left| \widetilde{\xi} \cdot x \right|} - \frac{(d-1)}{u^{d}} \int_{0}^{u}\frac{1}{\left| \widetilde{\xi} \cdot x \right|} r^{d-2}\dd r\\
    = {} & d^{-1} - \frac{2}{u\left| \widetilde{\xi} \cdot x \right|}.
  \end{align*}
  As we want to achieve an estimate from below, by the non-negativity of the integrand \( 1 - \cos(\xi \cdot x) \), we can restrict the integration domain in \eqref{eq:356} to
    \begin{equation*}
      \widetilde{S}(x) := \left\{ \widetilde \xi \in S^{d-1} : \left| \widetilde \xi \cdot x \right| \geq \frac{1}{2}\left| x \right| \right\} \quad
    \text{and} \quad D(u) := \left\{ x : \left| x \right| \geq \frac{8d}{u} \right\},
  \end{equation*}
  yielding
  \begin{equation}
    \frac{1}{d} - \frac{1}{u^{d}}\int_{0}^{u} \cos(r\widetilde{\xi} \cdot x) r^{d-1}\dd r \geq \frac{1}{2d}, \quad x \in D(u),\, \widetilde{\xi} \in \widetilde{S}(x).\label{eq:72}
  \end{equation}
  Combining \eqref{eq:72} with \eqref{eq:67} gives us
  \begin{equation*}
    \frac{1}{u^{d}} \int_{\left| \xi \right|\leq u} (1 - \Re \widehat{\mu}(\xi)) \dd \xi \geq \frac{1}{C_{3}} \mu\left(\left\{ \left| x \right| \geq 8D's^{-1} \right\}\right)
  \end{equation*}
  with
  \begin{equation*}
    C_{3} := \frac{1}{2d} \operatorname{vol} (\widetilde S(x)),
  \end{equation*}
  where \(\operatorname{vol} (\widetilde S(x))\) is independent of \(x\). Finally, we substitute \(\widetilde u := (8d)^{-1}u\) to get
  \begin{equation*}
    \mu\left(\left\{x : \left| x \right| \geq \widetilde{u}^{-1}\right\}\right) \leq \frac{C_{1}}{\widetilde{u}^{d}}\int_{\left| \xi \right| \leq C_{2}\widetilde{u}} (1 - \Re \widehat{\mu}(\xi)) \dd \xi
  \end{equation*}
  with
  \begin{equation*}
    C_{1} := \frac{C_{3}}{(8d)^{d}} \quad \text{and} \quad  C_{2} := 8d.\qedhere
  \end{equation*}
\end{proof}

\raggedbottom
\begin{proposition}
  \label{prp:compctness-sublvls}
  \(\widehat{\mathcal{E}}\colon \mathcal{P}(\mathbb{R}^{d}) \rightarrow  \mathbb{R}_{\geq 0} \cup \{\infty\}\) is lower
  semi-continuous with respect to the narrow convergence and its sub-levels are
  narrowly compact.
\end{proposition}

\begin{proof}
  Lower semi-continuity and thence closedness of the sub-levels follows from Fatou’s lemma, because narrow convergence corresponds to pointwise convergence of the Fourier transform and the integrand in the definition of \(\widehat{\mathcal{E}}\) is non-negative.
  
  Now, assume we have a \(K > 0\) and
  \begin{equation*}
    \mu \in N_{K}(\widehat{\mathcal{E}}) := \{\mu \in \mathcal{P}(\mathbb{R}^{d}) : \widehat{\mathcal{E}}[\mu] \leq K\}.
  \end{equation*}
  We show the tightness of the family of probability measures
  \(N_{K}(\widehat{\mathcal{E}})\) using Lemma \ref{lem:7}. Let \(0 < u \leq 1\). Then,
  \begin{align}
    \leadeq[2]{\frac{1}{u^{d}} \int_{\left| \xi \right|\leq C_{2}u} \left(1 - \Re \widehat{\mu}(\xi)\right) \dd \xi} \\
    \leq {} & C_2^d \int_{\left| \xi \right|\leq C_{2}u} \left| \xi \right|^{-d} \left(1 - \Re \widehat{\mu}(\xi)\right) \dd \xi  \label{eq:78} \\
    \leq {} & C_2^d \int_{\left| \xi \right|\leq C_{2}u} \left| \xi \right|^{-d} \left(\left| 1 - \Re \widehat{\omega}(\xi) \right| + \left| \Re \widehat{\omega}(\xi) - \Re \widehat{\mu}(\xi) \right| \right)\dd \xi \nonumber \\
    \leq {} & C_2^d \int_{\left| \xi \right|\leq C_{2}u} \left| \xi \right|^{-d} \left(\left| 1 - \widehat{\omega}(\xi) \right| + \left| \widehat{\omega}(\xi) - \widehat{\mu}(\xi) \right| \right)\dd \xi \nonumber \\
    = {} & C_2^d \int_{\left| \xi \right|\leq C_{2}u} \left| \xi \right|^{(-d-q)/2} \cdot \left| \xi \right|^{(-d+q)/2} \left(\left| 1 - \widehat{\omega}(\xi) \right| + \left| \widehat{\omega}(\xi) - \widehat{\mu}(\xi) \right| \right)\dd \xi \nonumber \\
\nonumber
    \end{align}
\begin{align}
  \leq {} & C_2^d \, {\underbrace{\left(\int_{\left| \xi \right|\leq C_{2}u} \left| \xi \right|^{-d+q} \dd \xi\right)}_{\text{\(=:f(u)\)}}}^{1/2} \cdot \Bigg[{\underbrace{\left(\int_{\left| \xi \right|\leq C_{2}u} \left| \xi \right|^{-d-q} \left| 1 - \widehat{\omega}(\xi) \right|^{2}\dd \xi\right)}_{\text{\(= C \cdot \widehat{\mathcal{E}}[\delta_0] < \infty\)}}}^{1/2}\label{eq:79}\\
     & + {\underbrace{\left(\int_{\left| \xi \right|\leq C_{2}u}\left| \xi \right|^{-d-q} \left| \widehat{\omega}(\xi) - \widehat{\mu}(\xi) \right|^{2} \dd \xi\right)}_{\text{\(\leq D_{q}^{-1}K\)}}}^{1/2}\Bigg] \quad \text{(Hölder’s inequality)} \label{eq:80}\\
    \leq {} & C_2^d \, (f(u))^{1/2} \left( C^{1/2} + \left(D_{q}^{-1}K\right)^{1/2} \right), \nonumber
  \end{align}
  where in equations \eqref{eq:79} and \eqref{eq:80} we used the boundedness of the first summand in \eqref{eq:79} by a constant \(C > 0\), which is justified because \(\omega\) has an existing second moment. But
  \begin{equation*}
    f(u) = \int_{\left| \xi \right|\leq C_{2} u} \left| \xi \right|^{-d+q} \dd \xi = O(u^{q})
    \quad \text{for } u \rightarrow  0,
  \end{equation*}
  giving a uniform control of the convergence to zero of the left-hand side of \eqref{eq:78}. Together with Lemma \ref{lem:7}, this yields tightness of \(N_{K}(\widehat{\mathcal{E}})\), hence relative compactness with respect to narrow convergence. Compactness then follows from the aforementioned lower semi-continuity of \(\widehat{\mathcal{E}}\).
\end{proof}

From this proof, we cannot deduce a stronger compactness, so that the limit of a minimizing sequence for the original functional \(\widetilde{\mathcal{E}}\) (which coincides with \(\widehat{\mathcal{E}}\) on \(\mathcal{P}_{2}(\mathbb{R}^{d})\) by Corollary \ref{cor:four-repr-widet}) need not lie in the set \(\mathcal{P}_{2}(\mathbb{R}^{d})\) (actually, in Section \ref{sec:moment-bound-symm}, we shall see that we can prove a slightly stronger compactness). To apply compactness arguments, we hence need an extension of \(\widetilde{\mathcal{E}}\) to the whole of \(\mathcal{P}(\mathbb{R}^{d})\). For the direct method and later $\Gamma$-convergence to be applied, this extension should also be lower semi-continuous; therefore the natural candidate is the \emph{lower semi-continuous envelope} of \(\widetilde{\mathcal{E}}\), now defined on the whole of \( \mathcal{P}(\mathbb{R}^d) \) by
\begin{equation*}
  \widetilde{\mathcal{E}}[\mu] =
  \begin{cases}
    \widetilde{\mathcal{E}}[\mu], \quad & \mu \in \mathcal{P}_2(\mathbb{R}^d), \\
    \infty, & \mu \in \mathcal{P}(\mathbb{R}^d) \setminus \mathcal{P}_2(\mathbb{R}^d),
  \end{cases}
\end{equation*}
which in our case can be defined as
\begin{equation*}
  \widetilde{\mathcal{E}}^{-}[\mu] := \inf_{\substack{\mu_{n}\rightarrow \mu \text{        narrowly}\\\mu_{n} \in \mathcal{P}_{2}(\mathbb{R}^{d})}} \liminf_{n\rightarrow \infty}  \widetilde{\mathcal{E}}[\mu_{n}],
\end{equation*}
or equivalently as the largest lower semi-continuous function smaller than  \( \widetilde{\mathcal{E}} \).
This corresponds to \cite[Definition 3.1]{93-Dal_Maso-intro-g-conv} if we consider our functional initially to be \( +\infty \) for \( \mu \in \mathcal{P}(\mathbb{R}^d) \setminus \mathcal{P}_2(\mathbb{R}^d) \).

In order to show that actually \(\widetilde{\mathcal{E}}^{-} = \widehat{\mathcal{E}}\), which is the content of Corollary \ref{cor:lower-semi-cont} below, we need a sequence along which there is continuity in the values of \( \widetilde{\mathcal{E}} \), which we find by dampening an arbitrary \( \mu \) by a Gaussian.

\flushbottom
\begin{proposition}
  \label{prp:rec-seq}
  For \( \omega \in \mathcal{P}_2(\mathbb{R}^d) \) and \(\mu \in \mathcal{P}(\mathbb{R}^{d})\), there exists a sequence \((\mu_{n})_{n \in \mathbb{N}} \subseteq \mathcal{P}_{2}(\mathbb{R}^{d})\) such that
  \begin{alignat*}{2}
    \mu_{n}&\rightarrow \mu \text{ narrowly} \quad&&\text{ for } n\rightarrow \infty,\\
    \widehat{\mathcal{E}}[\mu_{n}]&\rightarrow \widehat{\mathcal{E}}[\mu] &&\text{ for } n\rightarrow \infty.
  \end{alignat*}
\end{proposition}

\begin{proof}
  \emph{1. Definition of \(\mu_{n}\).} Define
  \begin{equation*}
    \eta(x) := (2\pi)^{-d/2}\exp\left(-\frac{1}{2}\left| x \right|^{2}\right), \quad \eta_{\varepsilon}(x) := \varepsilon^{-d}\eta(\varepsilon^{-1}x), \quad x \in \mathbb{R}^{d}.
  \end{equation*}
  Then \((2\pi)^{-d} \widehat{\widehat{\eta_{\varepsilon}}} = \eta_\varepsilon \) is a non-negative approximate identity with respect to the convolution and \( \widehat{\eta}_\varepsilon = \exp(-\varepsilon^2 \left| x \right|^2/2) \). To approximate \(\mu\), we use a smooth dampening of the form
  \begin{equation*}
    \mu_{n} := \widehat{\eta}_{n^{-1}}\cdot \mu + \left(1 - (\widehat{\eta}_{n^{-1}}\cdot \mu)(\mathbb{R}^{d})\right)\delta_{0},
  \end{equation*}
  such that the resulting \( \mu_n \) are in \( \mathcal{P}_2 \), with Fourier transforms
  \begin{equation*}
    \widehat{\mu}_{n}(\xi) =  (\widehat{\mu} \ast
    \eta_{n^{-1}})(\xi) - (\widehat{\mu} \ast
    \eta_{n^{-1}})(0) + 1, \quad \xi \in \mathbb{R}^{d}.
  \end{equation*}
  Note that because \(\widehat{\mu}\) is continuous, \(\widehat{\mu}_{n}(\xi) \rightarrow \widehat{\mu}(\xi)\) for all \(\xi \in \mathbb{R}^{d}\). We want to use the dominated convergence theorem to deduce that
  \begin{equation*}
    \widehat{\mathcal{E}}[\mu_{n}] = D_{q} \int_{\mathbb{R}^{d}} \left| \xi \right|^{-d-q}
    \left| \widehat{\mu}_{n}(\xi)-\widehat{\omega}(\xi) \right|^{2} \dd \xi \rightarrow  \widehat{\mathcal{E}}[\mu] \quad \text{ for }
    n\rightarrow \infty.
  \end{equation*}
  
  \emph{2. Trivial case and dominating function.} Firstly, note that if \(\widehat{\mathcal{E}}[\mu] = \infty\), then Fatou's lemma ensures that \(\widehat{\mathcal{E}}[\mu_{n}]\rightarrow \infty\) as well.

  Secondly, by the assumptions on \( \omega \), it is sufficient to find a dominating function for
  \begin{equation*}
    \xi \mapsto \left| \xi \right|^{-d-q} \left| \widehat{\mu}_n(\xi) - 1 \right|^2,
  \end{equation*}
  which will only be problematic for \( \xi \) close to \( 0 \). We can estimate the behavior of \(\widehat{\mu}_{n}\) by that of \(\widehat{\mu}\) as
  \begin{align}
    \left| \widehat{\mu}_{n}(\xi) - 1 \right| \leq {} & \int_{\mathbb{R}^{d}} \int_{\mathbb{R}^{d}} \eta_{n^{-1}}(\zeta) \left| \exp(\i(\zeta-\xi)\cdot x) - \exp(\i\zeta\cdot x) \right| \dd \mu(x) \dd \zeta \nonumber \\
    = {} & \int_{\mathbb{R}^{d}} \underbrace{\int_{\mathbb{R}^{d}} \eta_{n^{-1}}(\zeta) \dd \zeta}_{\text{\(=1\)}}\,\left| \exp(-i\xi\cdot x) - 1 \right| \dd \mu(x) \nonumber \\
    \leq {} &C\bigg[(1 - \Re \widehat{\mu}(\xi)) + \underbrace{\int_{\mathbb{R}^{d}} \left| \sin (\xi\cdot x) \right| \dd \mu(x)}_{\text{\(f(\xi):=\)}}\bigg],\label{eq:92}
  \end{align}
  where the right-hand side \eqref{eq:92} is to serve as the dominating function. Note that we can estimate each summand in \eqref{eq:92} separately to justify integrability due to the elementary inequality
  \begin{equation*}
    \left| a+b \right|^{2} \leq 2\left(\left| a \right|^{2} + \left| b \right|^{2}\right) \quad
    \text{for all } a,b \in \mathbb{C}.
  \end{equation*}
  Taking the square of \eqref{eq:92} yields
  \begin{equation}
    \label{eq:94}
    \left| \widehat \mu_{n}(\xi) - 1 \right|^{2} \leq C\left[(1 - \Re \widehat{\mu}(\xi))^{2} + \left(\int_{\mathbb{R}^{d}} \left| \sin (\xi\cdot x) \right| \dd \mu(x)\right)^{2}\right].
  \end{equation}
  Now, by the existence of the second moment of \(\omega\), we know that
  \begin{align}
    \leadeq{\int_{\mathbb{R}^{d}} \left| \xi \right|^{-d-q} (1 - \Re \widehat{\mu}(\xi))^{2} \dd \xi} \nonumber \\
    \leq {} & \int_{\mathbb{R}^{d}} \left| \xi \right|^{-d-q} \left| \widehat{\mu}(\xi) - 1 \right|^{2} \dd \xi \nonumber \\
    \leq {} & 2\int_{\mathbb{R}^{d}} \left| \xi \right|^{-d-q}\left| \widehat{\mu}(\xi) - \widehat{\omega}(\xi) \right|^{2} \dd \xi + 2 \int_{\mathbb{R}^{d}} \left| \xi \right|^{-d-q} \left| \widehat{\omega}(\xi) - 1 \right|^{2} \dd \xi < \infty\label{eq:97}
  \end{align}
  This yields the integrability condition for the first term in equation \eqref{eq:94}. What remains is to show the integrability for the term \(f\) in \eqref{eq:92}, which will occupy the rest of the proof.
  
  \emph{3. Splitting \(f\).} We apply the estimate
  \begin{equation*}
    \left| \sin(y) \right| \leq \min\{\left| y \right|,1\} \quad \text{for } y \in \mathbb{R},
  \end{equation*}
  resulting in
  \begin{equation*}
    f(\xi) = \int_{\mathbb{R}^{d}} \left| \sin(\xi\cdot x) \right| \dd \mu(x) \leq \underbrace{\left| \xi \right|
      \int_{\left| x \right| \leq \left| \xi \right|^{-1}} \left| x \right| \dd \mu(x)}_{\text{\(:= f_{1}(\xi)\)}}
    + \underbrace{\int_{\left| x \right| \geq
        \left| \xi \right|^{-1}} \dd \mu(x)}_{\text{\(:= f_{2}(\xi)\)}}.
  \end{equation*}
  
  \emph{4. Integrability of \(f_{2}\):}\label{item:1} By Lemma \ref{lem:7}
  and Hölder's inequality, we can estimate \(f_{2}\) as follows:
  \begin{align}
    f_{2}(\xi) \leq {} &\frac{C_{1}}{\left| \xi \right|^{d}} \int_{\left| y \right| \leq C_{2}\left| \xi \right|} (1
    - \Re \widehat{\mu}(y)) \dd y \nonumber \\
    \leq {} & \frac{C_{1}}{\left| \xi \right|^{d}} {\underbrace{\left(\int_{\left| y \right| \leq
            C_{2}\left| \xi \right|} 1 \dd y\right)^{1/2}}_{\text{\(= C\left| \xi \right|^{d/2}\)}}}
    \left(\int_{\left| y \right| \leq C_{2}\left| \xi \right|}(1 - \Re \widehat{\mu}(y))^{2} \dd
      y\right)^{1/2}\label{eq:101}
  \end{align}
  Hence, inserting \eqref{eq:101} into the integral which we want to show to be finite and applying Fubini-Tonelli yields
  \begin{align*}
    \int_{\mathbb{R}^{d}} \left| \xi \right|^{-d-q} f_{2}(\xi)^{2} \dd \xi \leq {} &C \int_{\mathbb{R}^{d}} \left| \xi \right|^{-2d-q} \int_{\left| y \right|\leq C_{2}\left| \xi \right|} (1 - \Re \widehat{\mu}(y))^{2}\dd y \dd \xi\\
    \leq {} & C\int_{\mathbb{R}^{d}} (1 - \Re \widehat{\mu}(y))^{2} \underbrace{\int_{C_{2}\left| \xi \right| \geq \left| y \right|} \left| \xi \right|^{-2d-q} \dd \xi}_{\text{\(=C\left| y \right|^{-d-q}\)}} \dd y\\
    \leq {} & C\int_{\mathbb{R}^{d}} \left| y \right|^{-d-q}(1 - \Re \widehat{\mu}(y))^{2} \dd y < \infty
  \end{align*}
  by \eqref{eq:97}.
  
  \emph{5. Integrability of \(f_{1}\):} We use Fubini-Tonelli to get a well-known estimate for the first moment, namely
  \begin{align*}
    f_{1}(\xi) = {} & \left| \xi \right| \int_{\left| x \right|\leq \left| \xi \right|^{-1}} \left| x \right| \dd
    \mu(x)\\
    = {} & \left| \xi \right| \int_{\left| x \right|\leq \left| \xi \right|^{-1}} \int_{0}^{\left| x \right|} 1
    \dd z \dd \mu(x)\\
    = {} & \left| \xi \right| \int_{0}^{\infty} \int_{\mathbb{R}^{d}} 1_{\text{\(\{z \leq \left| x \right| \leq \left| \xi \right|^{-1}\}\)}} \dd \mu(x) \dd z\\
    \leq {} & \left| \xi \right| \int_{0}^{\left| \xi \right|^{-1}} \mu(\{z \leq \left| x \right|\}) \dd z.
  \end{align*}
  Next, we use Lemma \ref{lem:7} and Hölder's inequality (twice) to obtain
  (remember that \(1 \leq q < 2\) which ensures integrability)
  \begin{align*}
    f_{1}(\xi) \leq {} & C_{1} \left| \xi \right| \int_{0}^{\left| \xi \right|^{-1}} z^{d} \int_{\left| \zeta \right| \leq C_{2}z^{-1}} (1 - \Re \widehat{\mu}(\zeta)) \dd \zeta \dd z\\
    \leq {} & C_{1} \left| \xi \right| \int_{0}^{\left| \xi \right|^{-1}} {z^{d}}\underbrace{\left(\int_{\left| \zeta \right| \leq C_{2} z^{-1}} 1 \dd \zeta\right)^{1/2}}_{\text{\(=C\,z^{-d/2} = C\,z^{q/4 + (-d/2 - q/4)}\)}} \left(\int_{\left| \zeta \right| \leq C_{2} z^{-1}} (1 - \Re \widehat{\mu}(\zeta))^{2} \dd \zeta\right)^{1/2} \dd z\\
    \leq {} & C \left| \xi \right| \underbrace{\left(\int_{0}^{\left| \xi \right|^{-1}} z^{-q/2}\dd z\right)^{1/2}}_{\text{\(=C\left| \xi \right|^{q/4-1/2}\)}} \left(\int_{0}^{\left| \xi \right|^{-1}} \int_{\left| \zeta \right| \leq C_{2}z^{-1}} z^{d+q/2}(1 - \Re \widehat{\mu}(\zeta))^{2} \dd \zeta \dd z\right)^{1/2}.
  \end{align*}
  Squaring the expression and using Fubini-Tonelli on the second term, we obtain
  \begin{align}
    f_{1}(\xi)^{2} \leq {} & C \left| \xi \right|^{1+q/2} \int_{\mathbb{R}^{d}} (1 - \Re \widehat{\mu}(\zeta))^{2} \int_{0}^{\left| \xi \right|^{-1}} 1_{\{z \leq C_{2}\left| \zeta \right|^{-1}\}} z^{d+q/2} \dd z \dd \zeta \nonumber \\
    \leq {} & C \left| \xi \right|^{1+q/2} \int_{\mathbb{R}^{d}} (1 - \Re \widehat{\mu}(\zeta))^{2} \min\left\{\left| \xi \right|^{-d-q/2-1},\left| \zeta \right|^{-d-q/2-1}\right\} \dd \zeta \nonumber \\
    = {} & C \left| \xi \right|^{-d} \int_{\left| \zeta \right| \leq \left| \xi \right|} (1 - \Re \widehat{\mu}(\zeta))^{2} \dd \zeta\label{eq:106}\\
    &+ \underbrace{C \left| \xi \right|^{1+q/2} \int_{\left| \zeta \right|\geq\left| \xi \right|}\left| \zeta \right|^{-d-q/2-1} (1 - \Re \widehat{\mu}(\zeta))^{2} \dd \zeta}_{:= f_{3}(\xi)}\label{eq:107}
  \end{align}
  The integrability against \(\xi \mapsto \left| \xi \right|^{-d-q}\) of the term \eqref{eq:106} can now be shown analogously to \eqref{eq:101} in Step 2. Inserting the term \eqref{eq:107} into the integral and again applying Fubini-Tonelli yields
  \begin{align*}
    \leadeq{\int_{\mathbb{R}^{d}} \left| \xi \right|^{-d-q} f_{3}(\xi)^{2} \dd \xi} \\
    \leq {} & C \int_{\mathbb{R}^{d}} \left| \xi \right|^{-d-q/2+1} \int_{\left| \zeta \right|\geq\left| \xi \right|} \left| \zeta \right|^{-d-q/2-1} (1 - \Re \widehat{\mu}(\zeta))^{2} \dd \zeta \dd \xi\\
    = {} & C \int_{\mathbb{R}^{d}} \left| \zeta \right|^{-d-q/2-1} (1 - \Re \widehat{\mu}(\zeta))^{2} \underbrace{\int_{\left| \xi \right|\leq\left| \zeta \right|} \left| \xi \right|^{-d-q/2+1}\dd \xi}_{\text{\(=C\left| \zeta \right|^{-q/2+1}\)}} \dd \zeta\\
    = {} & C \int_{\mathbb{R}^{d}} \left| \zeta \right|^{-d-q} (1 - \Re \widehat{\mu}(\zeta))^{2} \dd \zeta < \infty,
  \end{align*}
  because of \eqref{eq:97}, which ends the proof. \qedhere
\end{proof}

\begin{corollary}
  \label{cor:lower-semi-cont}
  We have that
  \begin{equation*}
    \widetilde{\mathcal{E}}^{-}[\mu] = \widehat{\mathcal{E}}[\mu], \quad \mu \in \mathcal{P}(\mathbb{R}^{d})
  \end{equation*}
  and that \( \omega \) is the unique minimizer of \( \widetilde{\mathcal{E}}^- \).
\end{corollary}

\begin{proof}
  For \(\mu \in \mathcal{P}(\mathbb{R}^{d})\) and any sequence \((\mu_{n})_{n\in\mathbb{N}} \subseteq
  \mathcal{P}_{2}(\mathbb{R}^{d})\) with \(\mu_{n}\rightarrow \mu\) narrowly, we have
  \begin{equation*}
    \liminf_{n\rightarrow \infty} \widetilde{\mathcal{E}}[\mu_{n}] = \liminf_{n\rightarrow \infty}
     \widehat{\mathcal{E}}[\mu_{n}] \geq \widehat{\mathcal{E}}[\mu],
  \end{equation*}
  by the lower semi-continuity of \(\widehat{\mathcal{E}}\). By taking the infimum over all the sequences converging narrowly to $\mu$, we conclude
  \begin{equation}
    \label{eq:111}
    \widetilde{\mathcal{E}}^{-}[\mu] \geq \widehat{\mathcal{E}}[\mu] \quad \text{for all } \mu \in
    \mathcal{P}(\mathbb{R}^{d}).
  \end{equation}
  
  Conversely, for \(\mu \in \mathcal{P}(\mathbb{R}^{d})\), employing the sequence \((\mu_{n})_{n\in\mathbb{N}}\subseteq \mathcal{P}_{2}(\mathbb{R}^{d})\) of Proposition \ref{prp:rec-seq} allows us to see that
  \begin{equation}
    \label{eq:112}
    \widehat{\mathcal{E}}[\mu] = \lim_{n\rightarrow \infty} \widehat{\mathcal{E}}[\mu_{n}] = \lim_{n\rightarrow \infty}
    \widetilde{\mathcal{E}}[\mu_{n}] \geq \widetilde{\mathcal{E}}^{-}[\mu].
  \end{equation}
  Combining \eqref{eq:112} with \eqref{eq:111} yields the first claim, while the characterization of the minimizer follows from the form of \( \widehat{\mathcal{E}}\) in \eqref{eq:34}.
\end{proof}

Having verified this, in the following we shall work with the functional
\(\widehat{\mathcal{E}}\) instead of \(\mathcal{E}\) or \(\widetilde{\mathcal{E}}\).

\begin{remark}
  Notice that the lower semi-continuous envelope and therefore \(\widehat{\mathcal{E}}\) is also the \(\Gamma\)-limit, see Definition \ref{def:gamma-conv} below, of a regularization of \(\widetilde{\mathcal{E}}\) using the second moment, i.e., by considering 
  \begin{equation*}
    \mathcal{I}_{\varepsilon}[\mu] := \widetilde{\mathcal{E}}[\mu] + \varepsilon \int_{\mathbb{R}^{d}}\left| x \right|^{2} \dd \mu,
  \end{equation*}
  we have
  \begin{equation*}
    \mathcal{I}_{\varepsilon} \xrightarrow{\Gamma} \widetilde{\mathcal{E}}^{-} \quad \text{for
    } \varepsilon \rightarrow 0.
  \end{equation*}
\end{remark}

\subsection{Consistency of the particle approximations}
\label{sec:part-appr}

Let \(N \in \mathbb{N}\) and define
\begin{equation*}
  \mathcal{P}^{N}(\mathbb{R}^{d}) := \left\{ \mu \in \mathcal{P}(\mathbb{R}^{d}) : \mu =
    \frac{1}{N}\sum_{i = 1}^{N}\delta_{x_{i}} \text{ for some } \{x_{i}: i=1, \dots N \}
    \subseteq \mathbb{R}^{d} \right\}
\end{equation*}
and consider the restricted minimization problem
\begin{equation}
  \label{eq:115}
  \widehat{\mathcal{E}}_{N}[\mu] :=
  \begin{cases}
    \widehat{\mathcal{E}}[\mu], &\mu \in \mathcal{P}^{N}(\mathbb{R}^{d}),\\
    \infty, &\text{otherwise}
  \end{cases}\rightarrow \min_{\mu \in \mathcal{P}(\mathbb{R}^{d})}.
\end{equation}

We want to prove consistency of the restriction in terms of \(\Gamma\)-convergence of \(\widehat{\mathcal{E}}_{N}\) to \(\widehat{\mathcal{E}}\). This implies that the  discrete measures minimizing
 \(\widehat{\mathcal{E}}_{N}\) will converge to the unique minimizer $\omega$ of \(\widehat{\mathcal{E}}\), in other words the measure quantization of $\omega$ via the minimization of
\(\widehat{\mathcal{E}}_{N}\) is consistent.

\begin{definition} [\( \Gamma \)-convergence]
  \label{def:gamma-conv}
  \cite[Definition 4.1, Proposition 8.1]{93-Dal_Maso-intro-g-conv}
  Let \( X \) be a metrizable space and \( F_N \colon X \rightarrow (-\infty,\infty] \), \( N \in \mathbb{N} \) be a sequence of functionals. Then we say that \( F_N \) \emph{\( \Gamma \)-converges} to \( F \), written as \( F_N \xrightarrow{\Gamma} F \), for an \( F \colon X \rightarrow (-\infty,\infty] \), if
  \begin{enumerate}
  \item \emph{\( \liminf \)-condition:} For every \( x \in X \) and every sequence \( x_N \rightarrow x \),
    \begin{equation*}
      F(x) \leq \liminf_{N\rightarrow\infty} F_N(x_N);
    \end{equation*}
  \item \emph{\( \limsup \)-condition:} For every \( x \in X \), there exists a sequence \( x_N \rightarrow x \), called \emph{recovery sequence}, such that
    \begin{equation*}
      F(x) \geq \limsup_{N\rightarrow\infty} F_N(x_N).
    \end{equation*}
  \end{enumerate}

  Furthermore, we call the sequence \( (F_N)_N \) \emph{equi-coercive} if for every \( c \in \mathbb{R} \) there is a compact set \( K \subseteq X \) such that \( \left\{ x : F_N(x) \leq c \right\} \subseteq K \) for all \( N \in \mathbb{N} \). As a direct consequence,  choosing \( x_N \in \argmin F_N \) for all $N \in \mathbb N$, there is a subsequence \( (x_{N_k})_k \) and \( x^\ast \in X \) such that
  \begin{equation*}
    x_{N_k} \rightarrow x^\ast \in \argmin F.
  \end{equation*}
\end{definition}
We shall need a further simple lemma justifying the existence of minimizers for the problem \eqref{eq:115}.

\begin{lemma}
  \label{lem:9}
  For all \(N \in \mathbb{N}\), \(\mathcal{P}^{N}(\mathbb{R}^{d})\) is closed in the narrow topology.
\end{lemma}

\begin{proof}
  Note that \(\mathcal{P}(\mathbb{R}^{d})\) endowed with the narrow topology is a metrizable space, hence it is a Hausdorff space and we can characterize its topology by sequences. Let \(N \in \mathbb{N}\) and \((\mu_{k})_{k \in \mathbb{N}} \subseteq \mathcal{P}^{N}(\mathbb{R}^{d})\) with
  \begin{equation*}
    \mu_{k} \rightarrow  \mu \in \mathcal{P}(\mathbb{R}^{d}) \quad \text{narrowly for } k\rightarrow \infty.
  \end{equation*}
  By ordering the points composing each measure, for example using a lexicographical ordering, we can identify the measures \(\mu_{k}\) with a collection of points \(x^{k} \in \mathbb{R}^{d\times N}\). As the sequence \((\mu_{k})_{k}\) is convergent, it is tight, whence the columns of \((x^{k})_{k}\) must all lie in a compact set \(K \subseteq \mathbb{R}^{d}\). So we can extract a subsequence \((x^{k_{l}})_{l \in \mathbb{N}}\) such that 
  \begin{equation*}
    x^{k_{l}} \rightarrow  x^{\ast} = (x^{\ast}_{i})_{i = 1}^{N} \in \mathbb{R}^{d\times N} \quad \text{for } l\rightarrow \infty.
  \end{equation*}
  This implies that
  \begin{equation*}
    \mu_{k_{l}} \rightarrow  \mu^{\ast} = \frac{1}{N} \sum_{i}^{N}\delta_{x^{\ast}_{i}} \quad \text{narrowly for } l\rightarrow \infty.
  \end{equation*}
  Since \(\mathcal{P}(\mathbb{R}^{d})\) is a Hausdorff space, \(\mu = \mu^{\ast} \in \mathcal{P}^{N}(\mathbb{R}^{d})\), concluding the proof.
\end{proof}

\begin{theorem}[Consistency of particle approximations]
  \label{thm:cons-part-appr}
  The functionals \((\widehat{\mathcal{E}}_{N})_{N \in \mathbb{N}}\) are equi-coercive and
  \begin{equation*}
    \widehat{\mathcal{E}}_{N} \xrightarrow{\Gamma} \widehat{\mathcal{E}} \quad \text{for } N\rightarrow \infty,
  \end{equation*}
  with respect to the narrow topology. In particular,
  \begin{equation*}
    \argmin_{\mu \in \mathcal{P}(\mathbb{R}^{d})} \widehat{\mathcal{E}}_{N}[\mu] \ni \widetilde{\mu}_{N} \rightarrow \widetilde{\mu} = \argmin_{\mu \in \mathcal{P}(\mathbb{R}^{d})} \widehat{\mathcal{E}}[\mu] = \omega,
  \end{equation*}
  for any choice of minimizers \(\widetilde \mu_{N}\).
\end{theorem}

\begin{proof}
  \emph{1. Equi-coercivity:} This follows from the fact that \(\widehat{\mathcal{E}}\) has compact sub-levels by Proposition \ref{prp:compctness-sublvls}, together with \(\widehat{\mathcal{E}}_{N} \geq \widehat{\mathcal{E}}\).

  \emph{2. \( \liminf \)-condition:} Let \(\mu_{N} \in \mathcal{P}(\mathbb{R}^{d})\) such that \(\mu_{N} \rightarrow \mu\) narrowly for \(N\rightarrow \infty\). Then
  \begin{equation*}
    \liminf_{N\rightarrow \infty} \widehat{\mathcal{E}}_{N}[\mu_{N}] \geq \liminf_{N\rightarrow \infty} \widehat{\mathcal{E}}[\mu_{N}] \geq \widehat{\mathcal{E}}[\mu],
  \end{equation*}
  by the lower semi-continuity of \(\widehat{\mathcal{E}}\).

  \emph{3. \( \limsup \)-condition:} Let \(\mu \in \mathcal{P}(\mathbb{R}^{d})\). By Proposition \ref{prp:rec-seq}, we can find a sequence \((\mu^{k})_{k \in \mathbb{N}} \subseteq \mathcal{P}_{2}(\mathbb{R}^{d})\), for which \(\widehat{\mathcal{E}}[\mu^{k}] \rightarrow \widehat{\mathcal{E}}[\mu]\). Furthermore, by Lemma \ref{lem:3}, we can approximate each \(\mu^{k}\) by \((\mu^{k}_{N})_{N\in\mathbb{N}} \subseteq \mathcal{P}_{2}(\mathbb{R}^{d})\cap\mathcal{P}^{N}(\mathbb{R}^{d})\), a realization of the empirical process of \(\mu^{k}\). This has a further subsequence which converges in the \(2\)-Wasserstein distance by Lemma \ref{lem:26}, for which we have continuity of \(\widehat{\mathcal{E}}\) by Lemma \ref{lem:6}. A diagonal argument then yields a sequence \(\mu_{N} \in \mathcal{P}^{N}(\mathbb{R}^{d})\) for which
  \begin{equation*}
    \widehat{\mathcal{E}}_{N}[\mu_{N}] = \widehat{\mathcal{E}}[\mu_{N}] \rightarrow  \widehat{\mathcal{E}}[\mu] \quad \text{for } N\rightarrow \infty.
  \end{equation*}

  \emph{4.Convergence of minimizers:} We find minimizers for \(\widehat{\mathcal{E}}_{N}\) by applying the direct method, which is justified because the \((\widehat{\mathcal{E}}_{N})_{N}\) are equi-coercive and each \(\widehat{\mathcal{E}}_{N}\) is lower semi-continuous by Fatou’s lemma and Lemma \ref{lem:9}. The convergence of the minimizers \(\widetilde{\mu}_{N}\) to a minimizer \(\widetilde{\mu}\) of \(\widehat{\mathcal{E}}\) then follows. But \(\widetilde{\mu} = \omega\) because \(\omega\) is the unique minimizer of \(\widehat{\mathcal{E}}\).\qedhere
\end{proof}

\section{Moment bound in the symmetric case}
\label{sec:moment-bound-symm}

Let \( q=q_a = q_r \in (1,2) \) be strictly larger than \( 1 \) now. We want to prove that in this case, we have a stronger compactness than the one showed in Proposition \ref{prp:compctness-sublvls}, namely that the sub-levels of \( \widehat{\mathcal{E}} \) have a uniformly bounded \( r \)th moment for \( r < q/2 \).

In the proof, we shall be using the theory developed in Appendix \ref{cha:cond-posit-semi} in a more explicit form than before, in particular the notion of the generalized Fourier transform (Definition \ref{def:gen-fourier-transform}) and its computation in the case of the power function (Theorem \ref{thm:cond-ft-power}).

\begin{theorem}
  \label{thm:moment-bound-symmetric}
  Let \( \omega \in \mathcal{P}_2(\mathbb{R}^d) \). For \( r < q/2 \) and a given \( M \in \mathbb{R} \), there exists an \( M' \in \mathbb{R} \) such that
  \begin{equation*}
    \int_{\mathbb{R}^d} \left| x \right|^{r} \diff \mu(x) \leq M', \quad \text{for all } \mu \text{ such that } \widehat{\mathcal{E}}[\mu] \leq M.
  \end{equation*}
\end{theorem}

\begin{proof}
  Let \( \mu \in \mathcal{P}(\mathbb{R}^d) \). If \( \widehat{\mathcal{E}}[\mu] \leq M \), then we also have
  \begin{align*}
    M \geq {} & \widehat{\mathcal{E}}[\mu] = D_q \int_{\mathbb{R}^d} \left| \widehat{\mu}(\xi) - \widehat{\omega}(\xi) \right|^2 \left| \xi \right|^{-d-q} \diff \xi\\
    \geq {} &  \left( c \int_{\mathbb{R}^d} \left| \widehat{\mu}(\xi) - 1 \right|^2 \left| \xi \right|^{-d-q} \diff \xi - \int_{\mathbb{R}^d} \left| \widehat{\omega}(\xi) - 1 \right|^2 \left| \xi \right|^{-d-q} \diff \xi \right),
  \end{align*}
  so that there is an \( M'' > 0\) such that 
  \begin{equation*}
    \int_{\mathbb{R}^d} \left| \widehat{\mu} - 1 \right|^2 \left| \xi \right|^{-d-q} \diff \xi \leq M''.
  \end{equation*}
  Now approximate \( \mu \) by the sequence of Proposition \ref{prp:rec-seq}, again denoting it
  \begin{equation*}
    \mu_{n} := \widehat{\eta}_{n^{-1}}\cdot \mu + \left(1 - (\widehat{\eta}_{n^{-1}}\cdot \mu)(\mathbb{R}^{d})\right)\delta_{0},
  \end{equation*}
 and then \( \mu_n \) by a Gaussian mollification with \( \eta_{k_n^{-1}} \) to obtain the diagonal sequence \( \mu_{n}' := \mu_n \ast \eta_{k_n^{-1}}\), so that we have convergence \( \widehat{\mathcal{E}}[\mu_n'] \rightarrow \widehat{\mathcal{E}}[\mu] \). We set \( \widehat{\nu}_n := (\mu_n' - \eta_{k_n^{-1}}) \).

Then, \( \widehat{\nu}_n \in \mathcal{S}(\mathbb{R}^d) \), the space of Schwartz functions: by the dampening of Proposition \ref{prp:rec-seq}, the underlying measures have finite moment of any order, yielding decay of \( \widehat{\nu_n}(x) \) of arbitrary polynomial order for \( \left| x \right| \to \infty \), and the mollification takes care of \( \widehat{\nu}_n \in C^\infty(\mathbb{R}^d) \). Furthermore, set \( \nu_n = \widehat{\nu}^\vee_n \) and recall that the inverse Fourier transform can also be expressed as the integral of an exponential function. By expanding this exponential function in its power series, we see that for each \( n \),
  \begin{equation*}
    \widehat{\nu}_n(\xi) = O(\left| \xi \right|) \quad \text{for } \xi \rightarrow 0,
  \end{equation*}
  by the fact that \( \mu_n' \) and \( \delta_0 \) have the same mass, namely \( 1 \). Therefore, \( \widehat{\nu}_n \in \mathcal{S}_1(\mathbb{R}^d) \), see Definition \ref{def:restr-schwartz}, and we can apply Theorem \ref{thm:cond-ft-power}.2, to get
  \begin{align*}
    \leadeq{\int_{\mathbb{R}^d} \left| x \right|^{r} \widehat{\nu}_n(x) \diff x}\\
    = {} & C \int_{\mathbb{R}^d} \left| \xi \right|^{-d-r} \nu_n(\xi) \diff \xi\\
    \leq {} & C \Bigg[ \int_{\left| \xi \right| \leq 1} \underbrace{\left| \xi \right|^{-d-r}}_{= \left| \xi \right|^{-\frac{d-q+2r}{2}} \left| \xi \right|^{-\frac{d + q}{2}}} \left| \nu_n(\xi) \right| \diff \xi + \underbrace{\int_{\left| \xi \right| > 1} \left| \xi \right|^{-d-r} \left| \nu_n(\xi) \right| \diff \xi}_{\leq C < \infty} \Bigg]\\
    \leq {} & C \Bigg[ \underbrace{\left( \int_{\left| \xi \right| \leq 1} \left| \xi \right|^{-d+(q - 2r)} \diff \xi \right)^{1/2}}_{\smash{< \infty}} \left( \int_{\mathbb{R}^d} \left| \xi \right|^{-d-q} \left| \nu_n \right|^2 \diff \xi \right)^{1/2} + 1 \Bigg]\\
    \leq {} & C \left[ \left( \int_{\mathbb{R}^d} \left| \xi \right|^{-d-q} \left| \nu_n \right|^2 \diff \xi \right)^{1/2} + 1 \right].
  \end{align*}
  
  Now, we recall again the continuity of \( \widehat{\mathcal{E}} \) for \( \omega = \delta_0 \) along \( \mu_n \) by Proposition \ref{prp:rec-seq},  and its continuity \wrt the Gaussian mollification. The latter can be seen either by the \( 2 \)-Wasserstein-convergence of the mollification for \( n \) fixed or by using the dominated convergence theorem together with the power series expansion of \( \exp \), similarly to Lemma \ref{lem:11} below. To summarize, we see that
  \begin{equation*}
    \lim_{n\rightarrow\infty} \int_{\mathbb{R}^d} \left| \xi \right|^{-d-q} \left| \nu_n \right|^2 \diff \xi = (2 \pi)^{-d} \int_{\mathbb{R}^d} \left| \xi \right|^{-d-q} \left| \widehat{\mu} - 1 \right|^2 \diff \xi \leq (2 \pi)^{-d} M'',
  \end{equation*}
  while on the other hand we have
  \begin{align*}
    \liminf_{n\rightarrow\infty} \int_{\mathbb{R}^d} \left| x \right|^{r} \widehat{\nu}_n(x) \diff x = {} & \liminf_{n\rightarrow\infty} \int_{\mathbb{R}^d} \left| x \right|^{r} \diff \mu_n(x) - \underbrace{\lim_{n\rightarrow\infty} \int_{\mathbb{R}^d} \left| x \right|^r \eta_{k_n^{-1}} (x) \diff x}_{=0}\\
    \geq {} & \int_{\mathbb{R}^d} \left| x \right|^{r} \diff \mu(x)
  \end{align*}
  by Lemma \ref{lem:4}, concluding the proof.
\end{proof}

\section{Regularization by using the total variation}
\label{sec:tv-reg}

We shall regularize the functional \( \widehat{\mathcal{E}} \) by an additional total variation term, for example to reduce the possible effect of  noise on the given datum \( \omega \). In particular, we expect the minimizer of the corresponding functional to be piecewise smooth or even piecewise constant while any sharp edges discontinuities in \( \omega \) should be preserved, as it is the case for the regularization of a quadratic fitting term, see for example \cite[Chapter 4]{10_Chambolle_ea_tv_intro}.

In the following, we begin by introducing this regularization and prove that for a vanishing regularization parameter, the minimizers of the regularizations converge to the minimizer of the original functional. One effect of the regularization will be to allow us to consider approximating or regularized minimizers of \( \widehat{\mathcal{E}}[\mu] \) in \( \mathcal{P}(\mathbb{R}^d) \cap BV(\mathbb{R}^d) \), where \( BV(\mathbb{R}^d) \) is the space of bounded variation functions. In the classical literature, one finds several approaches to discrete approximations to functionals including total variation terms as well as to their \( BV \)-minimizers, e.g., by means of finite element type approximations, see for example \cite{12-Bartels-TotalVariation}. Here however, we propose an approximation which depends on the position of (freely moving) particles in \( \mathbb{R}^d \), which can be combined with the particle approximation of Section \ref{sec:part-appr}. To this end, in Section \ref{sec:discrete-version-tv-1}, we shall present two ways of embedding into $L^1$ the Dirac masses which are associated to particles.

\subsection{Consistency of the regularization for the continuous functional}
\label{sec:cons-regul-cont}

For \(\mu \in \mathcal{P}(\mathbb{R}^{d})\), define
\begin{equation}
  \label{eq:135}
  \widehat{\mathcal{E}}^{\lambda}[\mu] :=
  \begin{cases}
    \widehat{\mathcal{E}}[\mu] + \lambda \left| D\mu \right|(\mathbb{R}^d), &\mu \in \mathcal{P}(\mathbb{R}^{d}) \cap BV(\mathbb{R}^{d}),\\
    \infty, &\text{otherwise},
  \end{cases}
\end{equation}
where \(D\mu\) denotes the distributional derivative of \(\mu\), being a finite Radon-measure, and \(\left| D \mu \right|(\mathbb{R}^d)\) its total variation. We present two useful Lemmata before proceeding to prove the \( \Gamma \)-convergence \( \widehat{\mathcal{E}}^\lambda \xrightarrow{\Gamma} \widehat{\mathcal{E}} \).

\begin{lemma}
  [Continuity of \(\widehat{\mathcal{E}}\) \wrt Gaussian mollification]
  \label{lem:11}
  Let \( \omega \in \mathcal{P}_2(\mathbb{R}^d) \), \(\mu \in \mathcal{P}(\mathbb{R}^{d})\) and set again
  \begin{equation*}
    \eta(x) := (2\pi)^{-d/2}\exp\left(-\frac{1}{2}\left| x \right|^{2}\right), \quad \eta_{\varepsilon}(x) := \varepsilon^{-d}\eta(\varepsilon^{-1}x), \quad x \in \mathbb{R}^{d}.
  \end{equation*}
  Then,
  \begin{equation*}
    \widehat{\mathcal{E}}[\eta_{\varepsilon}\ast \mu] \rightarrow \widehat{\mathcal{E}}[\mu], \quad \text{for } \varepsilon \rightarrow 0.
  \end{equation*}
\end{lemma}

\begin{proof}
  If \(\widehat{\mathcal{E}}[\mu] = \infty\), then the claim is true by the lower semi-continuity of \(\widehat{\mathcal{E}}\) together with the fact that \(\eta_{\varepsilon} \ast \mu \rightarrow \mu\) narrowly.

  If \(\widehat{\mathcal{E}}[\mu] < \infty\), we can estimate the difference \(\left|\widehat{\mathcal{E}}[\eta_{\varepsilon}\ast \mu]-\widehat{\mathcal{E}}[\mu]\right|\) (which is well-defined, but for now may be \(\infty\)) by using
  \begin{equation*}
    \left| a^{2} - b^{2} \right| \leq \left| a - b \right| \cdot \big( \left| a \right| + \left| b \right|\big), \quad a,b \in \mathbb{\mathbb{C}},
  \end{equation*}
  and
  \begin{equation*}
    \widehat{\eta_{\varepsilon}\ast\mu}(\xi) = \exp\left( -\frac{\varepsilon^2}{2} \left| \xi \right|^{2} \right) \widehat\mu(\xi),
  \end{equation*}
  as
  \begin{align}
    \leadeq{\left|\widehat{\mathcal{E}}[\eta_{\varepsilon}\ast \mu]-\widehat{\mathcal{E}}[\mu]\right|} \nonumber \\
    \leq {} & D_{q}\int_{\mathbb{R}^{d}} \left| \left| \widehat\eta_{\varepsilon}(\xi) \widehat\mu(\xi) - \widehat\omega(\xi) \right|^{2} - \left| \widehat\mu(\xi) - \widehat\omega(\xi) \right|^{2}\right| \left| \xi \right|^{-d-q} \diff \xi \nonumber \\
    \leq {} & D_{q}\int_{\mathbb{R}^{d}} \underbrace{\left(\left| \widehat\eta_{\varepsilon}(\xi) \widehat\mu(\xi) - \widehat\omega(\xi) \right| + \left| \widehat\mu(\xi) - \widehat\omega(\xi) \right| \right)}_{\leq 4} \underbrace{\left| \widehat\eta_{\varepsilon}(\xi) \widehat\mu(\xi) - \widehat\mu(\xi) \right|}_{=\left( 1 - \exp\left( -(\varepsilon^2/2)\left| \xi \right|^{2} \right) \right) \widehat\mu(\xi)} \left| \xi \right|^{-d-q} \diff \xi \nonumber \\
    \leq {} & C\int_{\mathbb{R}^{d}}\left( 1 - \exp\left( - \frac{\varepsilon^2}{2}\left| \xi \right|^{2} \right) \right) \left| \xi \right|^{-d-q}\diff \xi,\label{eq:141}
  \end{align}
  which converges to \(0\) by the dominated convergence theorem: On the one hand we can estimate
  \begin{equation*}
    \exp\left( -\frac{\varepsilon^2}{2}\left| \xi \right|^{2} \right) \geq 0, \quad \xi \in \mathbb{R}^{d},
  \end{equation*}
  yielding a dominating function for the integrand in \eqref{eq:141} for \(\xi\) bounded away from \(0\) because of the integrability of \( \xi \mapsto \left| \xi \right|^{-d-q} \) there. On the other hand
  \begin{align*}
    1 - \exp\left( -\frac{\varepsilon^2}{2}\left| \xi \right|^{2} \right) {} = & -\sum_{n = 1}^{\infty} \frac{1}{n!} \left(-\frac{\varepsilon^2}{2}\left| \xi \right|^{2}\right)^{n}\\
    = {} &\frac{\varepsilon^2}{2}\left| \xi \right|^{2} \sum_{n = 0}^{\infty} \frac{1}{(n+1)!} \left(-\frac{\varepsilon^2}{2}\left| \xi \right|^{2}\right)^{n},
  \end{align*}
  where the sum on the right is uniformly bounded for \(\varepsilon\rightarrow 0\) as a convergent power-series, which, combined with \(q < 2\), renders the integrand in \eqref{eq:141} dominated for \(\xi\) near \(0\) as well.
%
\end{proof}



\begin{proposition}
  [Consistency]
  \label{prp:consist-cont-bv}
Let $(\lambda_N)_{N \in \mathbb N}$ be a vanishing sequence of positive parameters.
  The functionals \((\widehat{\mathcal{E}}^{\lambda_N})_{N \in \mathbb{N}}\) are equi-coercive and
  \begin{equation*}
    \widehat{\mathcal{E}}^{\lambda_N} \xrightarrow{\Gamma} \widehat{\mathcal{E}} \quad \text{for } N\rightarrow 0,
  \end{equation*}
  with respect to the narrow topology. In particular, the limit point of minimizers of $\widehat{\mathcal{E}}^{\lambda_N}$ coincides with the unique minimizer $\omega$ of  $\widehat{\mathcal{E}}$.
\end{proposition}

\begin{proof}
  Firstly, observe that equi-coercivity follows from the narrow compactness of the sub-levels of \(\widehat{\mathcal{E}}\) shown in Proposition \ref{prp:compctness-sublvls}, and that the \( \liminf \)-condition is a consequence of the lower semi-continuity of \(\widehat{\mathcal{E}}\) as in the proof of Theorem \ref{thm:cons-part-appr}.

  \emph{Ad existence of minimizers:} We again want to apply the direct method of the calculus of variations.

  Let \( (\mu_k)_k \) be a minimizing sequence for \( \widehat{\mathcal{E}}^\lambda \), so that the \( \mu_k \) are all contained in a common sub-level of the functional. Now, for a given \( \lambda \), the sub-levels of \( \widehat{\mathcal{E}}^\lambda \) are relatively compact in \( L^1(\mathbb{R}^d) \), which can be seen by combining the compactness of the sub-levels of the total variation in \( L^1_{\text{loc}}(\mathbb{R}^d) \) with the tightness gained by \( \widehat{\mathcal{E}} \): if \( \widehat{\mathcal{E}}^\lambda[\mu_k] \leq M < \infty \), we can consider \( (\theta_l \mu_k)_k \) for a smooth cut-off function \( \theta_l \) having its support in \( [-l, l]^d \). By standard arguments we have the product formula
  \begin{equation*}
    D\left(\theta_l \mu_k\right) = D \theta_l \mu_k + \theta_l D \mu_k
  \end{equation*}
  and therefore
  \begin{align*}
    \left| D \left( \theta_l \mu_k \right) \right| (\mathbb{R}^d) \leq {} &  \int_{\mathbb{R}^d} \mu_k(x) \left| D\theta_l (x) \right| \diff x + \int_{\mathbb{R}^d}  \theta_l (x) \diff \left| D \mu_k \right| (x)\\
    \leq {} & C_l + \left| D\mu_k \right|(\mathbb{R}^d),
  \end{align*}
  so that for each \( l \), by the compactness of the sub-levels of the total variation in \( L^1_{\text{loc}} \), see \cite[Chapter 5.2, Theorem 4]{92-Evans-fine-properties}, we can select an \( L^{1} \)-convergent subsequence, which we denote again by\( (\theta_l \mu_{k})_k \). Then, by extracting further subsequences for $m \geq l$ and applying a diagonal argument we can construct a subsequence, again denoted $\mu_k$ such that
  \begin{align*}
    \left\| \mu_{k} - \mu_{h} \right\|_{L^1} \leq {} & \left\| (1-\theta_{l}) (\mu_{k} - \mu_{h}) \right\|_{L^1} + \left\| \theta_{l} \mu_{k} - \theta_{l} \mu_{h} \right\|_{L^1},
  \end{align*}
 and 
$$
 \left\| (1-\theta_{l}) (\mu_{k} - \mu_{h}) \right\|_{L^1} \leq \frac{\varepsilon}{2},
$$
for $l \geq \l_0(\varepsilon)$ large enough, because of the tightness of $(\mu_k)_k$, and
$$
\left\| \theta_{l} \mu_{k} - \theta_{l} \mu_{h} \right\|_{L^1} \leq \frac{\varepsilon}{2},
$$
by the local convergence in $L^1_{\mathrm{loc}}$ for $h,k \geq k_0(l)$ large enough. From this we conclude that $(\mu_k)_k$ is a Cauchy subsequence in $L^1$, hence, convergent.

  The lower semi-continuity of \( \widehat{\mathcal{E}}^\lambda \) follows from the lower semi-continuity of the total variation with respect to \( L^{1} \)-convergence and the lower semi-continuity of \( \widehat{\mathcal{E}} \) with respect to narrow convergence. Summarizing, we have compactness and lower semi-continuity, giving us that \( (\mu_k)_k \) has a limit point which is a minimizer.

  \emph{Ad \( \limsup \)-condition:} Let \(\mu \in \mathcal{P}(\mathbb{R}^{d})\) and write \(\mu_{\varepsilon} := \widehat{\eta}_{\varepsilon}\ast\mu\) for the mollification of Lemma \ref{lem:11}. Now, by Fubini’s theorem,
  \begin{align*}
    \left| D (\widehat{\eta}_{\varepsilon}\ast \mu) \right|(\mathbb{R}^d) = {} &\int_{\mathbb{R}^{d}} \left| \left(\nabla \widehat{\eta}_{\varepsilon}\ast \mu\right)(x) \right| \diff x \\
    \leq {} & \left\| \nabla \widehat{\eta}_{\varepsilon} \right\|_{L^{1}(\mathbb{R}^{d})} \mu(\mathbb{R}^{d})\\
    = {} & \varepsilon^{-d} \left\| \nabla\widehat{\eta} \right\|_{L^{1}(\mathbb{R}^{d})},
  \end{align*}
  so if we choose \(\varepsilon(\lambda)\) such that \(\lambda = o(\varepsilon^{d})\), for example \(\varepsilon(\lambda) := \lambda^{1/(d+1)}\), then
  \begin{equation*}
    \lambda \left| D\mu_{\varepsilon(\lambda)} \right|(\mathbb{R}^d) \rightarrow 0, \quad \text{for } \lambda \rightarrow 0.
  \end{equation*}
  On the other hand, \(\widehat{\mathcal{E}}[\mu_{\varepsilon(\lambda)}] \rightarrow \widehat{\mathcal{E}}[\mu]\) by Lemma \ref{lem:11}, yielding the required convergence \(\widehat{\mathcal{E}}^{\lambda}[\mu_{\varepsilon(\lambda)}] \rightarrow \widehat{\mathcal{E}}[\mu]\).

  The convergence of the minimizers then follows.
\end{proof}

\subsection{Discrete versions of the total variation regularization}
\label{sec:discrete-version-tv-1}

As one motivation for the study of the functional \( \mathcal{E} \) was to compute its particle minimizers, we shall also here consider a discretized version of the total variation regularization, for example to be able to compute the minimizers of the regularized functional directly on the level of the point approximations. We propose two techniques for this discretization.

The first technique is well-known in the non-parametric estimation of \( L^1 \) densities and consists of replacing each point with a small ``bump'' instead of interpreting it as a point measure. In order to get the desired convergence properties, we have to be careful when choosing the corresponding scaling of the bump. For an introduction to this topic, see \cite[Chapter 3.1]{85-Devroye-Gyoerfi-non-param-den-est}.

The second technique replaces the Dirac deltas by indicator functions which extend from the position of one point to the next one. Unfortunately, this poses certain difficulties in generalizing it to higher dimensions, as the set on which we extend would have to be replaced by something like a Voronoi cell, an object well-known in the theory of optimal quantization of measures, see for example \cite{00_Graf_Luschgy_quantization}.

In the context of attraction-repulsion functionals, it is of importance to note that the effect of the additional particle total variation term can again be interpreted as an attractive-repulsive-term. See Figure \ref{fig:discrete-tv} for an example in the case of kernel density estimation with a piecewise linear estimation kernel, where it can be seen that each point is repulsive at a short range, attractive at a medium range, and at a long range does not factor into the total variation any more.
This interpretation of the action of the total variation as a potential acting on particles to promote their uniform distribution is, to our knowledge, new.

\begin{figure}[t]
  \centering
  \subfigure[Linear \( K_1(x) \) and corresponding \( K_1'(x) \) ]{\includegraphics[]{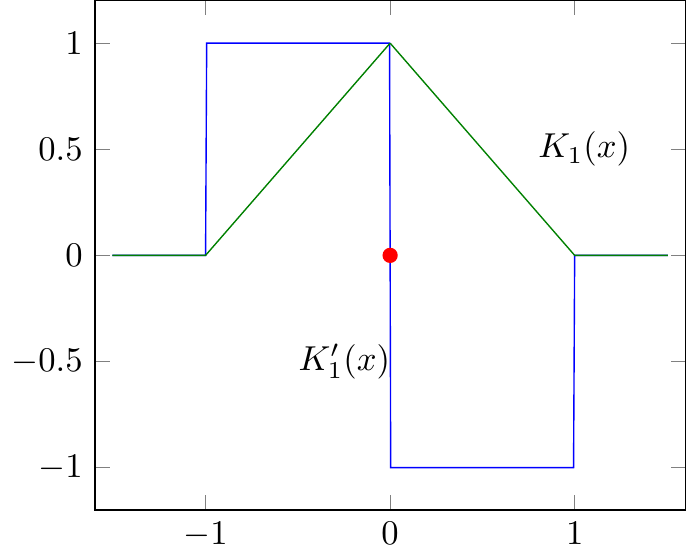}}
  \hspace{0.5cm}
  \subfigure[Discrete total variation for \( x_1 = 0 \) fixed (red), \( x_2 \) free, \( h = 0.75 \)]{\includegraphics[]{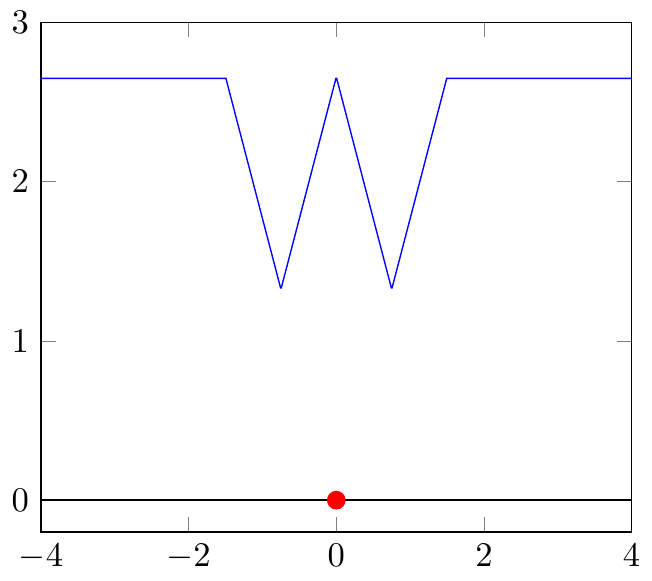}} 
  \caption{Example for the discrete total variation functional}
  \label{fig:discrete-tv}
\end{figure}

\subsubsection{Discretization by kernel estimators and quantization on deterministic tilings}
\label{sec:discr-kern-estim}

\begin{definition}
  [Discrete total variation via kernel estimate]
  \label{def:dens-est}
  For a \(\mu_{N} = \frac{1}{N}\sum_{i=1}^{N} \delta_{x_{i}} \in \mathcal{P}^{N}(\mathbb{R}^{d})\), a \emph{scale parameter} \(h = h(N)\) and a \emph{density estimation kernel} \( K \in W^{1,1}(\mathbb{R}^{d}) \) such that \( \nabla K \in BV(\mathbb{R}^d, \mathbb{R}^d) \), as well as
  \begin{equation*}
    K \geq 0, \quad \int_{\mathbb{R}^{d}}K(x) \diff x = 1,
  \end{equation*}
  we set
  \begin{equation*}
    K_{h}(x) := \frac{1}{h^{d}}K\left(\frac{x}{h}\right)
  \end{equation*}
  and define the corresponding \emph{\(L^{1}\)-density estimator} by
  \begin{equation*}
    Q_h[\mu_N](x) := K_{h}\ast\mu_{N}(x) = \frac{1}{N h^{d}}\sum_{i=1}^{N} K\left(\frac{x-x_{i}}{h}\right),
  \end{equation*}
  where the definition has to be understood for almost every \( x \). Then, we can introduce a discrete version of the regularization in \eqref{eq:135} as
  \begin{equation}
    \label{eq:151}
    \widehat{\mathcal{E}}_{N}^{\lambda}[\mu_N] := \widehat{\mathcal{E}}[\mu_N] + \lambda \left| DQ_{h(N)}[\mu_N]\right|(\mathbb{R}^d), \quad \mu_N \in \mathcal{P}^N(\mathbb{R}^d).
  \end{equation}
\end{definition}

We want to prove consistency of this approximation in terms of \( \Gamma \)-convergence of the functionals \( \widehat{\mathcal{E}}_N^\lambda \) to \( \widehat{\mathcal{E}}^\lambda \). For a survey on the consistency of kernel estimators in the probabilistic case under various sets of assumptions, see \cite{12-Wied-Weissbach-Kernel-Estimators}. Here however, we want to give a proof using deterministic and explicitly constructed point approximations.

\begin{figure}[t]
  \centering
  \subfigure[Notation of Definition \ref{def:good-tiling} for \( N = 5 \) ]{\includegraphics{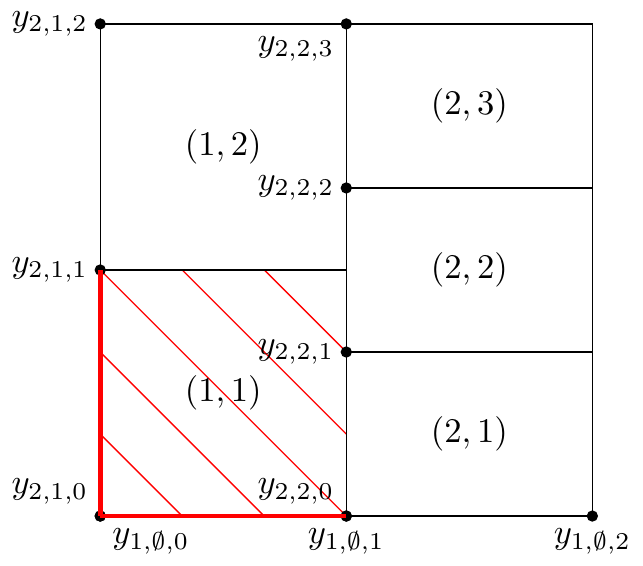}}
  \hspace{0\textwidth}
  \subfigure[{Tiling as in Example \ref{exp:constr-2d} for a uniform measure on two squares in \( [0,1]^2 \)}]{\includegraphics{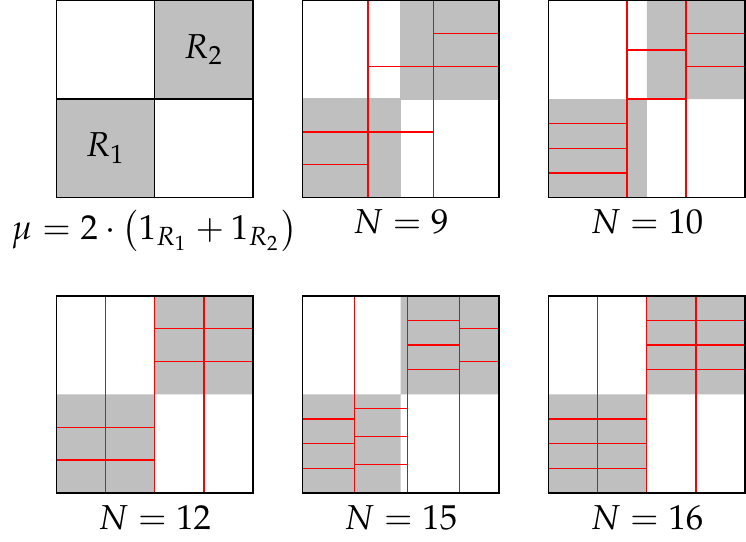}}
  \caption{Illustration of the tiling}
  \label{fig:tiling}
\end{figure}

In order to find a recovery sequence for the family of functionals \eqref{eq:151}, we have to find point approximations to a given measure with sufficiently good spatial approximation properties. For this, we suggest using a generalization of the quantile construction to higher dimensions. Let us state the properties we expect from such an approximation:

\begin{definition}
  [Tiling associated to a measure]
  \label{def:good-tiling}
  Let \( \mu \in \mathcal{P}_c(\mathbb{R}^d)\cap L^1(\mathbb{R}^d) \) compactly supported, where \( \mathcal{P}_c(\mathbb{R}^d) \) denotes the space of compactly supported probability measures, such that \( \operatorname{supp}(\mu) \subseteq [-R_\mu,R_\mu]^d \) and let \( N \in \mathbb{N} \). Set \( \widetilde{n} := \lfloor N^{1/d} \rfloor \). A good tiling (for our purposes) will be composed of an index set \( I \) and an associated tiling \( \left( T_i \right)_{i \in I} \) such that (see Figure \ref{fig:tiling} for an example of the notation):

  \begin{enumerate}
  \item \( I \) has \( N \) elements, \( \# I = N \), and in each direction, we have at least \( \widetilde{n} \) different indices, \ie,
    \begin{equation}
      \label{eq:155}
      \left\{ 1,\ldots,\widetilde{n} \right\}^d \subseteq I \subseteq \left\{ 1,\ldots,\widetilde{n}+1 \right\}^d.\\
    \end{equation}
    Additionally, for all \( k \in {1,\ldots,d} \) and \( (i_1,\ldots,i_{k-1},i_k,\ldots,i_d) \in I \),
    \begin{equation*}
      n_{k,i_1,\ldots,i_{k-1}} := \# \left\{ j_k : j \in I, \, (j_1,\ldots,j_{k-1}) = (i_1,\ldots,i_{k-1}) \right\} \in \left\{ \widetilde{n},\widetilde{n}+1 \right\}.
    \end{equation*}
  \item There is a family of ordered real numbers only depending on the first \( k \) coordinates,
    \begin{gather*}      
      y_{k,i_1,\ldots,i_k} \in [-R_\mu,R_\mu], \quad y_{k,i_1,\ldots,i_k-1} < y_{k,i_1,\ldots,i_k},\\
      \text{for all } k \in \left\{1,\ldots,d\right\} \text{ and } (i_1,\ldots,i_k,i_{k+1},\ldots,i_d) \in I, \nonumber
    \end{gather*}
    with fixed end points,
    \begin{equation*}
      y_{k,i_1,\ldots,i_{k-1},0}=-R_{\mu}, \quad y_{k,i_1,\ldots,i_{k-1},n_{k,i_1,\ldots,i_{k-1}}} = R_{\mu},
    \end{equation*}
    associated tiles
    \begin{equation*}
      T_i := \vartimes_{k=1}^{d} \left[y_{k,i_1,\ldots,(i_k-1)},y_{k,i_1,\ldots,i_k}\right],
    \end{equation*}
    and such that the mass of \( \mu \) is equal in each of them,
    \begin{equation*}
      \mu \left( T_i \right) = \frac{1}{N}, \quad \text{for all }i \in I.
    \end{equation*}
  \end{enumerate}
\end{definition}

Such a construction can always be found by generalizing the quantile construction. Let us show the construction explicitly for \( d = 2 \) as an example.

\begin{example}
  [Construction in 2D]
  \label{exp:constr-2d}
  Given \(N \in \mathbb{N}\), let \(\widetilde{n} := \lfloor \sqrt{N} \rfloor\). We can write \( N \) as
  \begin{equation}
    \label{eq:159}
    N = \widetilde{n}^{2-m}\left( \widetilde{n} +1 \right)^m + l,
  \end{equation}
with unique \( m \in \left\{ 0, 1 \right\} \) and \( l \in \left\{ 0,\ldots,\widetilde{n}^{1-m} \left( \widetilde{n} + 1 \right)^m - 1 \right\}. \)
  Then we get the desired tiling by setting
  \begin{alignat}{2}
    n_{1,\emptyset} := {} &
    \begin{cases}
      \widetilde{n} + 1 \quad &\text{if } m = 1,\\
      \widetilde{n} &\text{if } m = 0,
    \end{cases} \nonumber \\
    n_{2,i_1} := {} &
    \begin{cases}
      \widetilde{n}+1 \quad &\text{if } i_1 \leq l,\\
      \widetilde{n} &\text{if } i_1 \geq l + 1,
    \end{cases} \quad && i_1 = 1,\ldots,n_{1,\emptyset},\label{eq:375}\\
    w_{2,i_1,i_2} := {} &\frac{1}{n_{2,i_1}}, \quad && i_1 = 1,\ldots,n_{1,\emptyset},\; i_2 = 1,\ldots,n_{2,i_1}, \nonumber \\
    w_{1,i_1} := {} &\frac{n_{2,i_1}}{\sum_{j_1} n_{2,j_1}}, \quad &&i_1 = 1,\ldots,n_{1,\emptyset}, \nonumber
  \end{alignat}
  and choosing the end points of the tiles such that
  \begin{align}
    \label{eq:161}
    \sum_{j_1 = 1}^{i_1} w_{1,j_1} = {} &\int_{-R_\mu}^{y_{1,i_1}} \int_{-R_\mu}^{R_\mu} \diff \mu(x_1,x_2),\\
    \sum_{j_1 = 1}^{i_1} \sum_{j_2 = 1}^{i_2} w_{1,j_1} w_{2,j_1,j_2} = {} &\int_{-R_\mu}^{y_{1,i_1}} \int_{-R_\mu}^{y_{2,i_1,i_2}} \diff \mu(x_1,x_2). \nonumber
  \end{align}
  Now, check that indeed \(\sum_{j_1} n_{2,j_1} = N\) by \eqref{eq:159} and \eqref{eq:375} and that we have
  \begin{equation*}
    \mu(T_{i_1,i_2}) = w_{1,i_1} w_{2,i_1,i_2} = \frac{1}{N} \quad \text{for all } i_1, i_2,
  \end{equation*}
  by the choice of the weights \( w_{1,j_1} \), \( w_{2,j_1,j_2} \) as desired.
\end{example}

The general construction now consists of choosing a subdivision in \( \widetilde{n}+1 \) slices uniformly in as many dimensions as possible, while keeping in mind that in each dimension, we have to subdivide in at least \( \widetilde{n} \) slices. There will again be a rest \( l \), which is filled up in the last dimension.

\begin{proposition}
  [Construction for arbitrary \( d \)]
  \label{prp:constr-gen-dim}
  A tiling as defined in Definition \ref{def:good-tiling} exists for all \( d \in \mathbb{N} \).
\end{proposition}

\begin{proof}
  Analogously to Example \ref{exp:constr-2d}, let \( \widetilde{n}:= \lfloor N^{1/d} \rfloor \) and set
  \begin{equation*}
    N = \widetilde{n}^{d-m}\left( \widetilde{n} +1 \right)^m + l,
  \end{equation*}
  with unique \( m \in \left\{ 0, \ldots, d-1 \right\} \) and \( l \in \left\{ 0,\ldots,\widetilde{n}^{d-1-m} \left( \widetilde{n} + 1 \right)^m - 1 \right\} \). Then, we get the desired ranges by
  \begin{alignat}{2}
    n_{k,i_1,\ldots,i_{k-1}} &{}:= \widetilde{n} + 1, &&\text{for } k \in \left\{ 1,\ldots,m \right\} \text{ and all relevant indices};\\ 
    n_{k,i_1,\ldots,i_{k-1}}, &{}:= \widetilde{n}, &&\text{for } k \in \left\{ m+1,\ldots,d-1 \right\} \text{ and all relevant indices};\\
    n_{d,i_1,\ldots,i_{d-1}} &{}\in \left\{ \widetilde{n}, \widetilde{n}+1 \right\}, \quad&&\text{such that exactly \( l \) multi-indices are } \widetilde{n} + 1.
  \end{alignat}
  The weights can then be selected such that we get equal mass after multiplying them, and the tiling is found by iteratively using a quantile construction similar to \eqref{eq:161} in Example \ref{exp:constr-2d}.
\end{proof}

\begin{lemma}
  [Consistency of the approximation]
  \label{lem:13}
  For \(\mu\in \mathcal{P}_c(\mathbb{R}^d) \cap BV(\mathbb{R}^d)\), let \( (T_i)_{i \in I} \) be a tiling as in Definition \ref{def:good-tiling}, and \( x_i \in T_i \) for all \( i \in I \) an arbitrary point in each tile. Then, \(\mu_N = \frac{1}{N} \sum_{i=1}^{N} \delta_{x_i}\) converges narrowly to \(\mu\) for \(N \rightarrow\infty\).
  Furthermore, if
  \begin{equation}
    \label{eq:172}
    h = h(N)\rightarrow 0 \quad \text{and} \quad h^{2d}N \rightarrow \infty \quad \text{for } N\rightarrow \infty,
  \end{equation}
  then \(Q_{h(N)}[\mu_N] \rightarrow\mu\) strictly in \(BV(\mathbb{R}^d)\) (as defined in \cite[Definition 3.14]{00-Amb-Fusco-Pallara-BV}).
\end{lemma}

\begin{proof}
  Suppose again that
  \begin{equation*}
    \operatorname{supp} \mu \subseteq [-R_{\mu},R_{\mu}]^d.
  \end{equation*}
  
  \emph{Ad narrow convergence:}
  By \cite[Theorem 3.9.1]{Dur10}, it is sufficient to test convergence for bounded, Lipschitz-continuous functions. So let \(\varphi \in C_b(\mathbb{R}^d) \) be a Lipschitz function with Lipschitz constant \(L\). Then,
  \begin{align*}
    \leadeq{\left| \int_{\mathbb{R}^d} \varphi(x) \diff \mu_N (x) - \int_{\mathbb{R}^d} \varphi(x) \diff \mu(x) \right|}\\
    = {} & \left| \frac{1}{N} \sum_{i=1}^{N} \varphi(x_i) - \int_{\mathbb{R}^d} \varphi(x) \diff \mu(x) \right|\\
    \leq {} & \sum_{i \in I} \int_{T_i} \left| \varphi(x) - \varphi(x_i) \right| \diff \mu(x)\\
    \leq {} & L \sum_{i \in I} \int_{T_i} \left| x - x_i \right| \diff \mu(x).
  \end{align*}
  Denote by
  \begin{equation*}
    \widehat{\pi}_k (i_1,\ldots,i_d) := (i_1,\ldots,i_{k-1},i_{k+1},i_d)
  \end{equation*}
  the projection onto all coordinates except the \( k \)th one. Now, we exploit the uniformity of the tiling in all dimensions, \eqref{eq:155}: By using the triangle inequality and grouping the summands,
  \begin{align}
    \leadeq{\sum_{i \in I} \int_{T_i} \left| x - x_i \right| \diff \mu(x)} \\
    \leq {} & \sum_{i \in I} \sum_{k=1}^{d} \int_{T_i} \left| x^k - x_i^k \right| \diff \mu(x)\label{eq:176}\\
    = {} & \sum_{k=1}^{d} \sum_{i \in \widehat{\pi}_k(I)} \sum_{j=1}^{n_{k,i_1,\ldots,i_{k-1}}} \int_{T_i} \left| x^k - x_{i_1,\ldots,i_{k-1},j,i_k,\ldots,i_{d-1}}^k \right| \diff \mu(x) \nonumber \\
    \leq {} & \sum_{k=1}^{d} \sum_{i \in \widehat{\pi}_k(I)} \underbrace{\sum_{j=1}^{n_{k,i_1,\ldots,i_{k-1}}} \left( y_{k,i_1,\ldots,i_{k-1},(j-1)}-y_{k,i_1,\ldots,i_{k-1},j} \right)}_{=2R_\mu} \underbrace{\int_{T_i} \diff \mu(x)}_{=1/N} \nonumber \\
    \leq {} & 2R_{\mu}\,d \frac{\left(\widetilde{n}+1\right)^{d-1}}{N} \leq 2R_{\mu}\,d \frac{\left(\widetilde{n}+1\right)^{d-1}}{\widetilde{n}^{d}} \leq \frac{C}{\widetilde{n}} \rightarrow 0 \quad \text{for } N\rightarrow\infty.\label{eq:376}
  \end{align}
  
  \emph{Ad \(L^1\)-convergence:} As \( K \in W^{1,1}(\mathbb{R}^d) \subseteq BV(\mathbb{R}^d)\), we can approximate it by \( C^1 \) functions which converge \( BV \)-strictly, so let us additionally assume \( K \in C^1 \) for now. Then,
  \begin{align}
    \label{eq:177}
    \leadeq{\int_{\mathbb{R}^d} \left| K_h \ast \mu_N(x) - \mu(x) \right| \diff x}\\
    \leq {} & \int_{\mathbb{R}^d} \left| K_h \ast \mu_N(x) - K_h \ast \mu (x) \right| \diff x + \int_{\mathbb{R}^d} \left| K_h \ast \mu(x) - \mu(x) \right| \diff x.
  \end{align}
  By \(h \rightarrow 0\), the second term goes to \( 0 \) (see \cite[Chapter 5.2, Theorem 2]{92-Evans-fine-properties}), so it is sufficient to consider
  \begin{align}
    \leadeqnum{\int_{\mathbb{R}^d} \left| K_h \ast \mu_N(x) - K_h \ast \mu(x)  \right| \diff x}\label{eq:178}\\
    \leq {} & \sum_{i \in I} \int_{T_i} \int_{\mathbb{R}^d} \left| K_h(x - x_i) - K_h(x-y) \right| \diff x \diff \mu(y) \nonumber \\
    = {} & \sum_{i \in I} \int_{T_i} \int_{\mathbb{R}^d} \left| \int_0^1 \nabla K_h(x - y + t (y - x_i)) \cdot (y - x_i) \diff t \right| \diff x \diff \mu(y) \nonumber \\
    \leq {} & \sum_{i \in I} \int_{T_i} \int_0^1 \left| y - x_i \right| \int_{\mathbb{R}^d} \left| \nabla K_h(x - y + t (y - x_i)) \right| \diff x \diff t \diff \mu(y) \nonumber \\
    = {} & \frac{1}{h} \left\| \nabla K \right\|_{L^1} \sum_{i \in I} \int_{T_i}  \left| y - x_i \right| \diff \mu(y).\label{eq:179}
  \end{align}
  Since the left-hand side \eqref{eq:178} and the right-hand side \eqref{eq:179} of the above estimate are continuous with respect to strict BV convergence (by Fubini-Tonelli and convergence of the total variation, respectively), this estimate extends to a general \( K \in BV(\mathbb{R}^d) \) and
  \begin{equation*}
    \frac{1}{h} \sum_{i \in I} \int_{T_i}  \left| y - x_i \right| \diff \mu(y) \leq \frac{C}{\widetilde{n}h} \rightarrow 0, \quad \text{for } N \rightarrow \infty,
  \end{equation*}
  %
  by the calculation in \eqref{eq:176} and condition \eqref{eq:172}.

  \emph{Ad convergence of the total variation:} Similarly to the estimate in \eqref{eq:177}, by \( h \rightarrow 0 \) it is sufficient to consider the \(L^1\) distance between \( \nabla K_h \ast \mu_N \) and \( \nabla K_h \ast \mu \) if we approximate a general \( K \) by a \( K \in C^{2}(\mathbb{R}^d) \). By a calculation similar to \eqref{eq:178} -- \eqref{eq:179} as well as \eqref{eq:376} and using \( \nabla K_h(x) = h^{-d-1} \nabla K(x/h) \), we get
  \begin{align*}
    \leadeq{\int_{\mathbb{R}^d} \left| \nabla K_h \ast \mu_N(x) - \nabla K_h \ast \mu(x)  \right| \diff x}\\
    \leq {} & C \frac{1}{h} \sum_{i \in I} \int_{T_i} \int_{\mathbb{R}^d} \left| \nabla K_h(x-x_i) - \nabla K_h(x-y) \right| \diff x \diff \mu(y)\\
    \leq {} & C \left\| D^2 K \right\|_{L^1} \frac{1}{\widetilde{n}h^2} \rightarrow 0 \quad \text{for } N\rightarrow \infty,
  \end{align*}
  by the condition \eqref{eq:172} we imposed on \( h \).
\end{proof}

Since we associate to each \( \mu_N \in \mathcal{P}_N \) an \( L^1 \)-density \( Q_{h(N)}[\mu_N] \) and want to analyze both the behavior of \( \mathcal{E}[\mu_N] \) and \( \left| DQ_{h(N)}[\mu_N] \right|(\mathbb{R}^d) \), we need to incorporate the two different topologies involved, namely the narrow convergence of \( \mu \) and \( L^1 \)-convergence of \( Q_{h(N)}[\mu] \), into the concept of \( \Gamma \)-convergence. This can be done by using a slight generalization introduced in \cite{94-Anzellotti-GeneralizedGamma}, named \( \Gamma(q,\tau^-) \)-convergence.

\begin{definition}[\( \Gamma(q,\tau^-) \)-convergence]
  \label{def:-gammaq-tau}
  \cite[Definition 2.1]{94-Anzellotti-GeneralizedGamma}
  For \( N \in \mathbb{N} \), let \( X_N \) be a set and \( F_N \colon X_N \to \mathbb{R} \) a function. Furthermore, let \( Y \) be a topological space with topology \( \tau \) and \( q = ( q_N )_{N \in \mathbb{N}} \) a sequence of embedding maps \( q_N \colon X_N \to Y \). Then, \( F_N \) is said to \( \Gamma(q,\tau^-) \)-converge to a function \( F \colon Y \to \overline{\mathbb{R}} \) at \( y \in Y \), if
  \begin{enumerate}
  \item \emph{\( \liminf \)-condition}: For each sequence \( x_N \in X_N \) such that \( q_N(x_N) \xrightarrow{\tau} y \),
    \begin{equation*}
      F(y) \leq \liminf_{N \to \infty} F_N(x_N).
    \end{equation*}
  \item \emph{\( \limsup \)-condition}: There is a sequence \( x_N \in X_N \) such that \( q_N(x_N) \xrightarrow{\tau} y \) and
    \begin{equation*}
      F(y) \geq \limsup_{N \to \infty} F_N(x_N).
    \end{equation*}
  \end{enumerate}

  Furthermore, we say that the \( F_N \) \( \Gamma(q,\tau^-) \)-converge on a set \( \mathcal{D} \subseteq Y \) if the above is true for all \( y \in \mathcal{D} \) and we call the sequence \( F_N \) equi-coercive, if for every \( c \in \mathbb{R} \), there is a compact set \( K \subseteq Y \) such that \( q_N\left( \left\{ x : F_N(x) \leq c \right\} \right) \subseteq K \).
\end{definition}

\begin{remark}
  The main consequence of \( \Gamma \)-convergence, which is of interest to us, is the convergence of minimizers. This remains true also for \( \Gamma(q,\tau^-) \)-convergence, see \cite[Proposition 2.4]{94-Anzellotti-GeneralizedGamma}.
\end{remark}

Here, we are going to consider
\begin{equation}
  \label{eq:379}
  Y :=  \mathcal{P}(\mathbb{R}^d) \times BV(\mathbb{R}^d)
\end{equation}
with the corresponding product topology of narrow convergence and $BV$ weak-$*$-convergence (actually $L^1$-convergence suffices),
\begin{equation*}
  X_N := \mathcal{P}^N(\mathbb{R}^d), \quad q_N(\mu) := (\mu,Q_{h(N)}[\mu]).
\end{equation*}
and consider the limit \( \widehat{\mathcal{E}}^\lambda \) to be defined on the diagonal
\begin{equation*}
  \mathcal{D} := \left\{ (\mu,\mu) : \mu \in \mathcal{P}(\mathbb{R}^d) \cap BV(\mathbb{R}^d) \right\}.
\end{equation*}

Since we will be extracting convergent subsequences of pairs \( (\mu_N,Q_{h(N)}[\mu]) \) in order to obtain existence of minimizers, we need the following lemma to ensure that the limit is in the diagonal set \( \mathcal{D} \).

\begin{lemma}[Consistency of the embedding \( Q_{h(N)} \)]
  \label{lem:29}
  If \( (\mu_N)_N \) is a sequence such that \( \mu_N \in \mathcal{P}^N(\mathbb{R}^d) \), \( \mu_N \to \mu \in \mathcal{P}(\mathbb{R}^d) \) narrowly and \( Q_{h(N)}[\mu_N] \to \widetilde{\mu} \in BV(\mathbb{R}^d) \) in \( L^1(\mathbb{R}^d) \) as well, as \( h(N) \to 0 \), then \( \mu = \widetilde{\mu} \).
\end{lemma}

\begin{proof}
  To show \( \mu = \widetilde{\mu} \), by the metrizability of \( \mathcal{P} \) it suffices to show that \( Q_{h(N)}[\mu_N] \to \mu \) narrowly. For this, as in the proof of Lemma \ref{lem:13}, we can restrict ourselves to test convergence of the integral against bounded and Lipschitz-continuous functions. Hence, let \( \varphi \in C_b(\mathbb{R}^d) \cap \mathrm{Lip}(\mathbb{R}^d) \) with Lipschitz constant \( L \). Then,
  \begin{align*}
    \leadeq{\left| \int_{\mathbb{R}^d} \varphi(x) Q_{h(N)}[\mu_N](x) \diff x - \int_{\mathbb{R}^d} \varphi(x) \diff \mu(x) \right|} \\
    \leq {} & \left| \int_{\mathbb{R}^d} \varphi(x) K_{h(N)} \ast \mu_N(x) \diff x - \int_{\mathbb{R}^d} \varphi(x) \diff \mu_N(x) \right|\\
    {} & + \left| \int_{\mathbb{R}^d} \varphi(x) \diff \mu_N(x) - \int_{\mathbb{R}^d} \varphi(x) \diff \mu(x) \right|,
  \end{align*}
  where the second term goes to zero by \( \mu_N \to \mu \) narrowly. For the first term, by Fubini we get that
  \begin{equation*}
    \int_{\mathbb{R}^d} \varphi(x) K_{h(N)} \ast \mu_N(x) \diff x = \int_{\mathbb{R}^d} (\varphi \ast K_{h(N)}(-\cdot)) (x) \diff \mu_N(x)
  \end{equation*}
  and therefore
  \begin{align*}
    \leadeq{\left| \int_{\mathbb{R}^d} \varphi(x) K_{h(N)} \ast \mu_N(x) \diff x - \int_{\mathbb{R}^d} \varphi(x) \mu_N(x) \diff x \right|} \\
    = {} & \left| \int_{\mathbb{R}^d} \int_{\mathbb{R}^d} \left( \varphi(x + y) - \varphi(x) \right) K_{h(N)}(y) \diff y \diff \mu_N(x) \right| \\
    = {} & \left| \int_{\mathbb{R}^d} \int_{\mathbb{R}^d} \left( \varphi(x + h(N)y) - \varphi(x) \right) K(y) \diff y \diff \mu_N(x) \right| \\
    \leq {} & Lh \left\| K \right\|_{L^1} \mu_N(\mathbb{R}^d) \to 0, \quad N \to 0
  \end{align*}
  by \( h(N) \to 0 \), proving \( Q_{h(N)}[\mu_N] \to \mu \) and hence the claim.
\end{proof}

\begin{theorem}
  [Consistency of the kernel estimate]
  \label{thm:con-kern-est}
  Under the assumption \eqref{eq:172} on \( h(N) \), the functionals \((\widehat{\mathcal{E}}_{N}^{\lambda})_{N \in \mathbb{N}}\) are equi-coercive and
  \begin{equation*}
    \widehat{\mathcal{E}}_{N}^{\lambda} \xrightarrow{\Gamma(q,\tau^-)} \widehat{\mathcal{E}}^{\lambda} \quad \text{for } N \rightarrow \infty,
  \end{equation*}
  with respect to the topology of \( Y \) defined above, \ie, weak convergence of \(\mu_{N}\) together with \(L^{1}\)-convergence of \(Q_{h(N)}[\mu_N]\). In particular, every sequence of minimizers of \( \widehat{\mathcal{E}}^{\lambda}_{N} \)  admits a subsequence converging to a minimizer of \(\widehat{\mathcal{E}}^{\lambda}\).
\end{theorem}

\begin{proof}
  \emph{Ad \( \liminf \)-condition:} This follows from the lower semi-continuity of \( \widehat{\mathcal{E}} \) and \( \mu \mapsto \left| D\mu \right|(\mathbb{R}^d) \) \wrt narrow convergence and \(L^1\)-convergence, respectively.

  \emph{Ad \( \limsup \)-condition:}  We use a diagonal argument to find the recovery sequence. An arbitrary \(\mu \in BV(\mathbb{R}^d) \cap \mathcal{P}(\mathbb{R}^d)\) can by Proposition \ref{prp:rec-seq} be approximated by probability measures \( \mu_n \) with existing second moment such that \( \widehat{\mathcal{E}}[\mu_n] \rightarrow \widehat{\mathcal{E}}[\mu] \), namely
  \begin{equation*}
    \mu_n = \widehat{\eta}_{n^{-1}} \cdot \mu + \left( 1 - \widehat{\eta}_{n^{-1}} \cdot \mu(\mathbb{R}^d) \right)\delta_0.
  \end{equation*}
  By Lemma \ref{lem:11}, we can also smooth the approximating measures by convolution with a Gaussian \( \eta_{\varepsilon(n)} \) to get a narrowly convergent sequence \( \mu_n' \to \mu \),
    \begin{equation*}
      \mu_n' = \eta_{\varepsilon(n)} \ast \mu_n = \eta_{\varepsilon(n)} \ast (\widehat{\eta}_{n^{-1}}\cdot \mu) + \left(1 - (\widehat{\eta}_{n^{-1}}\cdot \mu)(\mathbb{R}^{d})\right)\eta_{\varepsilon(n)},
    \end{equation*}
    while still keeping the continuity in \( \widehat{\mathcal{E}} \). Since \( \left(1 - (\widehat{\eta}_{n^{-1}}\cdot \mu)(\mathbb{R}^{d})\right) \rightarrow 0 \), we can replace its factor \( \eta_{\varepsilon(n)} \) by \( \eta_1 \) to get
    \begin{equation*}
      \mu_n'' = \eta_{\varepsilon(n)} \ast (\widehat{\eta}_{n^{-1}}\cdot \mu) + \left(1 - (\widehat{\eta}_{n^{-1}}\cdot \mu)(\mathbb{R}^{d})\right)\eta_{1},
    \end{equation*}
    and still have convergence and continuity in \( \widehat{\mathcal{E}} \). These \( \mu_n'' \) can then be (strictly) cut-off by a smooth cut-off function \( \chi_M \) such that
    \begin{alignat*}{2}
      \chi_M(x) & = 1 &&\text{ for } \left| x \right| \leq M, \\
      \chi_M(x) & \in [0,1] &&\text{ for } M < \left| x \right| < M+1, \\
      \chi_M(x) &= 0 &&\text{ for } \left| x \right| \geq M+1.
    \end{alignat*}
    Superfluous mass can then be absorbed in a normalized version of \( \chi_1 \), summarized yielding
    \begin{equation*}
      \mu_n''' = \chi_{M(n)} \cdot \mu_n'' + (1-\chi_{M(n)} \cdot \mu_n'')(\mathbb{R}^d) \frac{\chi_1}{\left\| \chi_1 \right\|_1},
    \end{equation*}
    which for fixed \( n \) and \( M(n) \to \infty  \) is convergent in the \( 2 \)-Wasserstein topology, hence we can keep the continuity in \( \widehat{\mathcal{E}} \) by choosing \( M(n) \) large enough.

    Moreover, the sequence \( \mu_n''' \) is also strictly convergent in \( BV \): for the \( L^1 \)-convergence, we apply the dominated convergence theorem for \( M(n) \to \infty \) when considering \( \mu_n''' \), and the approximation property of the Gaussian mollification of \( L^1 \)-functions for \( \mu_n'' \). Similarly, for the convergence of the total variation, consider
    \begin{align}
      \leadeq{\left| \left| D \mu_n''' \right|(\mathbb{R}^d) - \left| D \mu \right|(\mathbb{R}^d) \right|}\\
      \leq {} & \left| \int_{\mathbb{R}^d} \chi_{M(n)}(x) \left| D \mu_n''(x) \right| \diff x - \int_{\mathbb{R}^d} \left| D \mu_n''(x) \right| \diff x \right| \label{eq:391} \\
      {} & + \int_{\mathbb{R}^d} \left| \nabla \chi_{M(n)}(x) \right| \mu_n''(x) \diff x \label{eq:392} \\
      {} & + \left| \left| D \mu_n''(x) \right| - \left| D \mu \right|(\mathbb{R}^d) \right| \label{eq:383} \\
      {} & + (1-\chi_{M(n)} \cdot \mu_n'')(\mathbb{R}^d) \frac{\left\| \nabla \chi_1 \right\|_1}{\left\| \chi_1 \right\|_1}, \label{eq:388}
    \end{align}
    where the terms \eqref{eq:391}, \eqref{eq:392} and \eqref{eq:388} tend to \( 0 \) for \( M(n) \) large enough by Dominated Convergence. For the remaining term \eqref{eq:383}, we have
    \begin{align}
      \leadeq{\left| \left| D \mu_n'' \right| - \left| D \mu \right|(\mathbb{R}^d) \right|}\\
      \leq {} & \left| \left| \eta_{\varepsilon(n)} \ast D (\widehat{\eta}_n \cdot \mu) \right|(\mathbb{R}^d) - \left| D (\widehat{\eta}_n \cdot \mu) \right|(\mathbb{R}^d) \right| \label{eq:393} \\
      {} & + \int_{\mathbb{R}^d} \left| \nabla \widehat{\eta}_n(x) \right| \diff \mu(x) \label{eq:394} \\
      {} & + \int_{\mathbb{R}^d} (1 - \widehat{\eta}_n(x)) \diff \left| D \mu \right|(x) \label{eq:395} \\
      {} & + \left(1 - (\widehat{\eta}_{n^{-1}}\cdot \mu)(\mathbb{R}^{d})\right)\left| D\eta_{1} \right|(\mathbb{R}^d). \label{eq:396}
    \end{align}
    Here, all terms vanish as well: \eqref{eq:393} for \( \varepsilon(n) \) large enough by the approximation property of the Gaussian mollification for \( BV \)-functions and \eqref{eq:394}, \eqref{eq:395} and \eqref{eq:396} by the dominated convergence theorem for \( n \to \infty \). Finally, Lemma \ref{lem:13} applied to the \( \mu_n''' \) yields the desired sequence of point approximations.

  \emph{Ad equi-coercivity and existence of minimizers:} Equi-coercivity and compactness strong enough to ensure the existence of minimizers follow from the coercivity and compactness of level sets of \( \widehat{\mathcal{E}} \) and by \( \| Q_{h(N)}(\mu_N) \|_{L^1} = 1 \) together with compactness arguments in \( BV \), similar to Proposition \ref{prp:consist-cont-bv}. Since Lemma \ref{lem:29} ensures that the limit is in \( \mathcal{D} \), standard \( \Gamma \)-convergence arguments then yield the convergence of minimizers.
\end{proof}

\subsubsection{Discretization by point-differences}
\label{sec:discrete-version-tv}

In one dimension, the geometry is sufficiently simple to avoid the use of kernel density estimators to allow us to explicitly see the intuitive effect the total variation regularization has on point masses (similar to the depiction in Figure \ref{fig:discrete-tv} in the previous section). In particular, formula \eqref{eq:192} below shows that the total variation acts as an additional attractive-repulsive force which tends to promote equi-spacing between the points masses.

In the following, let \(d = 1\) and \(\lambda > 0\) fixed.

Let \( N\in\mathbb{N} \), \( N \geq 2 \) and \(\mu_{N} \in \mathcal{P}^{N}(\mathbb{R})\) with
\begin{equation*}
  \mu_{N} = \frac{1}{N} \sum_{i=1}^{N}\delta_{x_{i}} \quad \text{for some } x_{i} \in \mathbb{R}.
\end{equation*}
Using the ordering on \(\mathbb{R}\), we can assume the \( (x_i)_i \) to be ordered, which allows us to associate to \(\mu_{N}\) a unique vector
\begin{equation*}
  x := x(\mu_{N}) := (x_{1},\dotsc,x_{N}), \quad x_{1} \leq \dotsc \leq x_{N}.
\end{equation*}
If \(x_{i} \neq x_{j}\) for all \(i \neq j \in \{1,\dotsc,N\}\), we can further
define an \(L^{1}\)-function which is piecewise-constant by
\begin{equation*}
  \widetilde{Q}_{N}[\mu_N] := \frac{1}{N}\sum_{i = 2}^{N} \frac{1}{x_{i} - x_{i-1}}1_{[x_{i-1},x_{i}]}
\end{equation*}
and compute explicitly the total variation of its derivative to be
\begin{align}
  \label{eq:192}
  \leadeq{\left| D\widetilde{Q}_{N}[\mu_N] \right|(\mathbb{R})} \\
  = {} & \frac{1}{N}\left[\sum_{i = 2}^{N-1}
    \left| \frac{1}{x_{i+1}-x_{i}} - \frac{1}{x_{i}-x_{i-1}} \right| +
    \frac{1}{x_{2}-x_{1}} + \frac{1}{x_{N}-x_{N-1}}\right],
\end{align}
if no two points are equal, and \(\infty\) otherwise. This leads us to the following definition of the regularized functional using piecewise constant functions:
\begin{align*}
  \mathcal{P}^{N}_{\times}(\mathbb{R}) := {} &\left\{\mu \in \mathcal{P}^{N}(\mathbb{R}) : \mu = \frac{1}{N}
    \sum_{i=1}^{N}\delta_{x_{i}} \text{ with } x_{i} \neq x_{j} \text{ for } i \neq j
  \right\},\\
  \widehat{\mathcal{E}}^{\lambda}_{N,\mathrm{pwc}}[\mu] := {} &
  \begin{cases}
    \widehat{\mathcal{E}}[\mu] + \lambda \left| D\widetilde{Q}_{N}[\mu] \right|(\mathbb{R}), &\mu\in\mathcal{P}_{\times}^{N}(\mathbb{R});\\
    \infty,&\mu \in \mathcal{P}^N(\mathbb{R}) \setminus \mathcal{P}^N_\times(\mathbb{R}).
  \end{cases}
\end{align*}

\begin{remark}
  The functions \(\widetilde{Q}_{N}[\mu_N]\) as defined above are not probability densities, but instead have mass \((N-1)/N\).
\end{remark}

We shall again prove \( \Gamma(q,\tau^-) \)-convergence as in Section \ref{sec:discr-kern-estim}, this time with the embeddings \( q_N \) given by \( \widetilde{Q}_{N} \). The following lemma yields the necessary recovery sequence.

\begin{lemma}
  \label{lem:15}
  If \(\mu \in \mathcal{P}_c(\mathbb{R}) \cap C_{c}^{\infty}(\mathbb{R}) \) is the density of a compactly supported probability measure, then there is a sequence \(\mu_{N} \in \mathcal{P}^{N}(\mathbb{R})\), \(N \in \mathbb{N}_{\geq 2}\) such that
  \begin{equation*}
    \mu_{N} \rightarrow  \mu \quad \text{narrowly for } N\rightarrow \infty
  \end{equation*}
  and
  \begin{equation*}
    \widetilde{Q}_{N}[\mu_N]\rightarrow \mu \quad \text{in } L^{1}(\mathbb{R}), \quad \left| D\widetilde{Q}_{N}[\mu_N] \right|(\mathbb{R}) \rightarrow  \int_{\mathbb{R}} \left| \mu'(x) \right|
    \dd x \quad \text{for } N\rightarrow \infty.
  \end{equation*}
\end{lemma}

\begin{proof}
  \emph{1. Definition and narrow convergence:} Let \( \supp \mu \subseteq [-R_\mu, R_\mu] \) and define the vector \(x^{N} \in \mathbb{R}^{N}\) as an \(N\)th quantile of \( \mu \), i.e.,
  \begin{equation*}
    \int_{x^{N}_{i-1}}^{x^{N}_{i}} \mu(x) \dd x = \frac{1}{N} \quad \text{with
    } x^{N}_{i-1} < x^{N}_{i} \text{ for all } i = 1,\dotsc,N-1,
  \end{equation*}
  where we set \( x^{N}_{0} = -R_\mu \) and \( x^{N}_{N} = R_\mu \). Narrow convergence of the corresponding measure then follows by the same arguments used in the proof of Lemma \ref{lem:15}.
  
  \emph{2. \(L^1\)-convergence:} We want to use the dominated convergence theorem: Let \(x \in \mathbb{R}\) with \(\mu(x) > 0\). Then, by the continuity of \( \mu \), there are \(x_{i-1}^{N}(x)\), \( x_{i}^{N}(x) \) such that \( x \in [x_{i-1}^N(x), x_{i}^N(x)] \) and
  \begin{align}
    \mu(x) - \widetilde{Q}_{N}[\mu_N](x) = {} & \mu(x) - \frac{1}{N(x^{N}_{i}(x)-x^{N}_{i-1}(x))} \nonumber \\
    = {} & \mu(x) - \frac{1}{x^{N}_{i}(x)-x^{N}_{i-1}(x)}
    \int_{x^{N}_{i-1}(x)}^{x^{N}_{i}(x)}\mu(y) \dd y. \label{eq:390}
  \end{align}
  Again by \(\mu(x) > 0\) and the continuity of \(\mu\),
  \begin{equation*}
    x^{N}_{i}(x) - x^{N}_{i-1}(x) \rightarrow  0 \quad \text{for } N\rightarrow \infty,
  \end{equation*}
  and therefore
  \begin{equation*}
    \widetilde{Q}_{N}[\mu_N](x) \to \mu(x) \quad \text{for all } x \text{ such that } \mu(x) > 0.
  \end{equation*}
  On the other hand, if we consider an \( x \in [-R_\mu,R_\mu] \) such that \( x \notin \supp \mu \), say \( x \in [a,b] \), where the interval $[a,b]$ is such that \( \mu(\xi) = 0 \) for all \( \xi \in [a,b] \), and again denote by \( x^N_{i-1}(x), x^N_i(x) \) the two quantiles for which \( x \in [x^N_{i-1}(x),x^N_i(x)] \), then \( x_i^N(x) - x_{i-1}^N(x) \) stays bounded from below because \( x_{i-1}^N(x) \leq a \) and \( x_i^N(x) \geq b \), together with \( N \to \infty \) implying that for such an \( x \),
  \begin{equation*}
    \widetilde{Q}_{N}[\mu_N](x) = \frac{1}{N(x^N_i - x^N_{i-1})} \leq \frac{1}{N(b-a)} \to 0. 
  \end{equation*}
  Taking into account that $\mu$ can vanish on \(\supp \mu \) only on a subset of measure $0$, we thus have
  \begin{equation*}
    \widetilde{Q}_{N}[\mu_N](x) \to \mu(x) \quad \text{for almost every } x \in \mathbb{R}.
  \end{equation*}
  Furthermore, by \eqref{eq:390} and the choice of the \( (x^N_i)_i \), we can estimate the difference by 
  \begin{equation*}
    \left| \mu(x) - \widetilde{Q}_{N}[\mu_N](x) \right| \leq 2 \left\| \mu \right\|_\infty \cdot 1_{[-R_\mu,R_\mu]}(x),
  \end{equation*}
  yielding an integrable dominating function for \( \left| \mu(x) - \widetilde{Q}_{N}[\mu_N](x) \right| \) and therefore justifying the \( L^1 \)-convergence
  \begin{equation*}
    \int_{\mathbb{R}}\left| \mu(x) - \widetilde{Q}_{N}[\mu_N](x) \right| \dd x \to 0, \quad N \to \infty.
  \end{equation*}
  
  \emph{3. Strict BV-convergence:} For strict convergence of \(\widetilde{Q}_{N}[\mu_N]\) to \(\mu\), we additionally have to check that \(\limsup_{N\rightarrow \infty} \left| D\widetilde{Q}_{N}[\mu_N] \right|(\mathbb{R}) \leq \left| D\mu \right|(\mathbb{R}) \). To this end, consider
  \begin{align*}
    \leadeq{\left| D\widetilde{Q}_{N}[\mu_N] \right|(\mathbb{R})}\\
    = {} &\sum_{i = 2}^{N-1} \left| \frac{1}{N}\frac{1}{x^{N}_{i+1}-x^{N}_{i}} - \frac{1}{N}\frac{1}{x^{N}_{i}-x^{N}_{i-1}} \right| + \frac{1}{N(x^{N}_{2}-x^{N}_{1})} + \frac{1}{N(x^{N}_{N}-x^{N}_{N-1})}\\
    = {} & \sum_{i = 2}^{N-1} \left| \frac{1}{x^{N}_{i+1}-x^{N}_{i}}\int_{x^{N}_{i}}^{x^{N}_{i+1}}\mu(x) \dd x - \frac{1}{x^{N}_{i}-x^{N}_{i-1}}\int_{x^{N}_{i-1}}^{x^{N}_{i}}\mu(x) \dd x \right|\\
    &+ \frac{1}{x^{N}_{2}-x^{N}_{1}}\int_{x^{N}_{1}}^{x^{N}_{2}}\mu(x) \dd x + \frac{1}{x^{N}_{N}-x^{N}_{N-1}}\int_{x^{N}_{N-1}}^{x^{N}_{N}}\mu(x) \dd x\\
    = {} &\sum_{i = 1}^{N}\left| \mu(t_{i+1})-\mu(t_{i}) \right|
  \end{align*}
  for \(t_{i} \in [x^{N}_{i},x^{N}_{i-1}]\), \(i = 2,\dotsc,N\) chosen by the mean value theorem (for integration) and \(t_{1},t_{N+1}\) denoting \( -R_\mu \) and \( R_\mu \), respectively. Hence,
  \begin{equation*}
    \left| D\widetilde{Q}_{N}[\mu_{N}] \right|(\mathbb{R}) \leq \sup \left\{\sum_{i=1}^{n - 1}\left| \mu(t_{i+1})-\mu(t_{i}) \right| : n \geq 2,\, t_{1} < \dotsb < t_{n}\right\} = V(\mu),
  \end{equation*}
  the pointwise variation of \(\mu\), and the claim now follows from \(V(\mu) = \left| D\mu \right|(\mathbb{R})\) by \cite[Theorem 3.28]{00-Amb-Fusco-Pallara-BV}, because by the smoothness of \( \mu \), it is a good representative of its equivalence class in \( BV(\mathbb{R}) \), \ie, one for which the pointwise variation coincides with the measure theoretical one. \qedhere
\end{proof}

As in the previous section, we have to verify that a limit point of a sequence \( (\mu_N, \widetilde{Q}_N[\mu_N]) \) is in the diagonal \( \mathcal{D} \).

\begin{lemma}[Consistency of the embedding \( \widetilde{Q}_N \)]
  \label{lem:30}
  Let \( (\mu_N)_N \) be a sequence where \( \mu_N \in \mathcal{P}^N(\mathbb{R}) \), \( \mu_N \to \mu \) narrowly and \( \widetilde{Q}_N[\mu_N] \to \widetilde{\mu} \) in \( L^1(\mathbb{R}) \). Then \( \mu = \widetilde{\mu} \).
\end{lemma}

\begin{proof}
  Denote the cumulative distribution functions \( \widetilde{F}_N \), \( F_N \), and \( F \) of \( \widetilde{Q}_N[\mu_N] \), \( \mu_N \), and \( \mu \), respectively. We can deduce \( \mu = \widetilde{\mu} \) if \( \widetilde{F}_N(x) \to F(x) \) for every \( x \in \mathbb{R} \)  (even if the measures \( \widetilde{Q}_N[\mu_N] \) have only mass \( (N-1)/N \), this is enough to show that the limit measures have to coincide, for example by rescaling the measures to have mass \( 1 \)). Note that the construction of \( \widetilde{Q}_N[\mu_N] \) precisely consists of replacing the piecewise constant functions \( F_N \) by piecewise affine functions interpolating between the points \( (x^N_i)_i \). Now, taking into account that the jump size \( F_N(x^N_i)-F_N(x^N_{i-1}) \) is always \( 1/N \) we see that
  \begin{align*}
    | \widetilde{F}_N(x) - F(x)  | \leq {} & | \widetilde{F}_N(x) - F_N(x) | + \left| F_N(x) - F(x) \right|\\
    \leq {} & \frac{1}{N} + \left| F_N(x) - F(x) \right|\to 0, \quad N \to 0,
  \end{align*}
  which is the claimed convergence.
\end{proof}

\begin{theorem}
  [Consistency of \( \widehat{\mathcal{E}}^{\lambda}_{N,\mathrm{pwc}} \)]
  \label{thm:cons-discrete-tv}
Assume $d=1$. Then for $N \to \infty$ we have \(\widehat{\mathcal{E}}^{\lambda}_{N,\mathrm{pwc}} \xrightarrow{\Gamma(q,\tau^-)} \widehat{\mathcal{E}}^{\lambda}\) with respect to the topology of \( Y \) in \eqref{eq:379}, \ie, the topology induced by the narrow convergence together with the \(L^{1}\)-convergence of the associated densities, and the sequence \((\widehat{\mathcal{E}}^{\lambda}_{N,\mathrm{pwc}})_{N}\) is equi-coercive. In particular, every sequence of minimizers of \( \widehat{\mathcal{E}}^{\lambda}_{N,\mathrm{pwc}} \) admits a subsequence converging to a minimizer of \(\widehat{\mathcal{E}}^{\lambda}\).
\end{theorem}

\begin{proof}
  \emph{1. \( \liminf \)-condition:} Let \(\mu_{N} \in \mathcal{P}_{N}(\mathbb{R})\) and
  \(\mu \in BV(\mathbb{R}) \cap \mathcal{P}(\mathbb{R}) \) with \(\mu_{N}\rightarrow \mu\) narrowly and \(\widetilde{Q}_{N}[\mu_N]\rightarrow \mu\) in \(L^{1}\).
  Then,
  \begin{equation*}
    \liminf_{N\rightarrow \infty} \widehat{\mathcal{E}}^{\lambda}_{N,\mathrm{pwc}}[\mu_{N}] = \liminf_{N\rightarrow \infty} \left[\widehat{\mathcal{E}}[\mu_{N}] + \left| D\widetilde{Q}_{N}[\mu_N] \right|(\mathbb{R}) \right] \geq \widehat{\mathcal{E}}[\mu] + \left| D\mu \right|(\mathbb{R})
  \end{equation*}
  by the lower semi-continuity of the summands with respect to the involved topologies.
  
  \emph{2. \( \limsup \)-condition:} We use the same diagonal argument used in the proof of Theorem \ref{thm:con-kern-est}, replacing the final application of Lemma \ref{lem:13} there by Lemma \ref{lem:15}, which serves the same purpose, but uses the point differences instead of the kernel estimators.
  
  \emph{3. Equi-coercivity and existence of minimizers:} The coercivity follows analogously to the proof of Theorem \ref{thm:con-kern-est}, which also justifies the existence of minimizers for each \( N \). The convergence of minimizers to an element of \( \mathcal{D} \) then follows by standard arguments together with Lemma \ref{lem:30}.
\end{proof}

\begin{remark}
  \label{rem:non-diagonal-topology}
  In both cases, instead of working with two different topologies, we could also consider
  \begin{equation*}
    \widehat{\mathcal{E}}^\lambda_{N,\mathrm{alt}} := \widehat{\mathcal{E}}[Q[\mu]] + \lambda \left| DQ[\mu] \right| (\mathbb{R}^d),
  \end{equation*}
  for a given embedding \( Q \), which in the case of point differences would have to be re-scaled to keep mass \( 1 \). Then, we would obtain the same results by identical arguments, but without the need to worry about narrow convergence separately, since it is implied by the \( L^1 \)-convergence of \( Q[\mu_N] \).
\end{remark}

\section{Numerical experiments}
\label{sec:numer-exper}

\begin{figure}[t]
  \centering
  \subfigure[\( \omega_1 \)  \label{fig:data1}
]{\includegraphics[width=6cm]{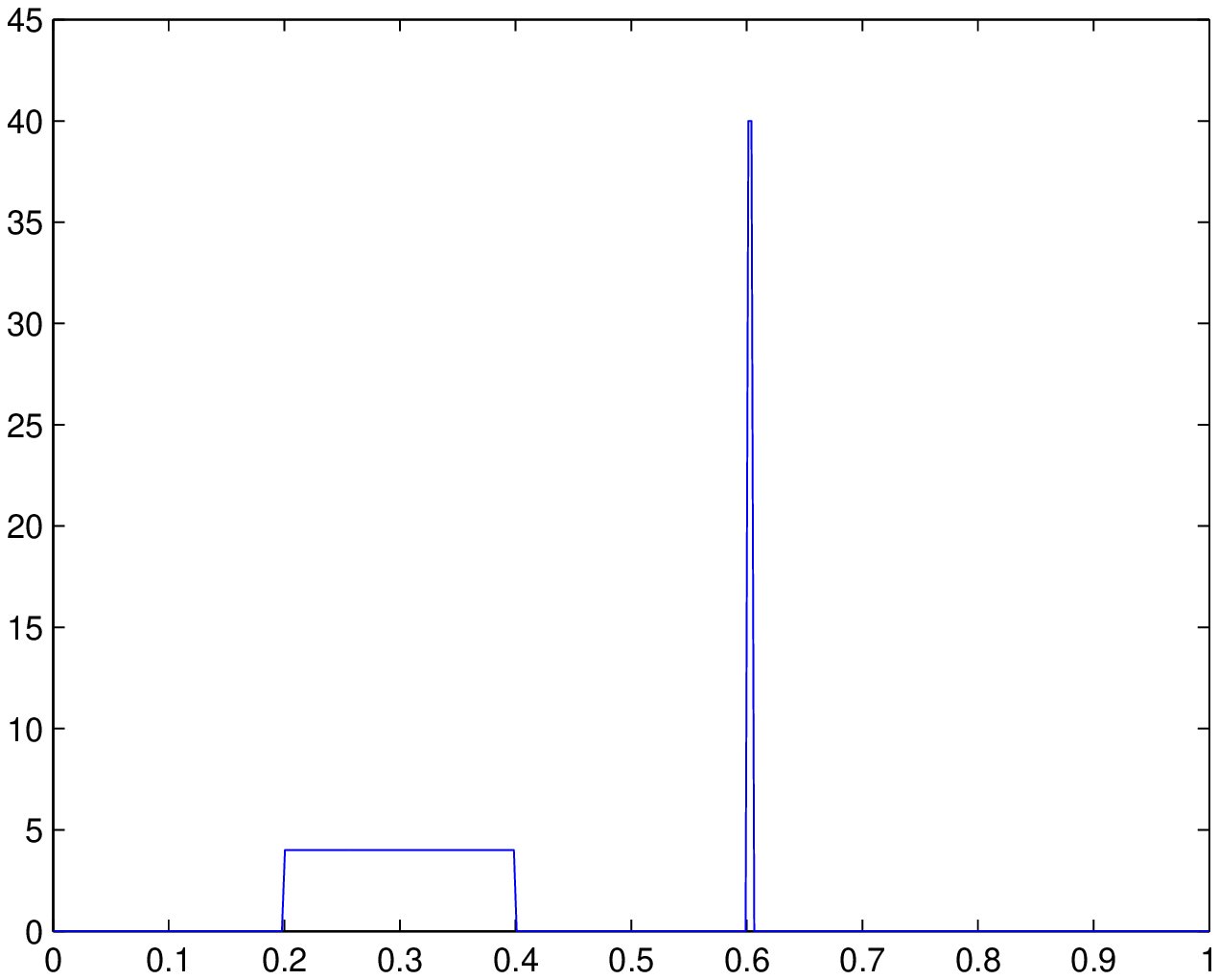}}
  \hspace{0.5cm}
  \subfigure[\( \omega_2 \)  \label{fig:data2} ]{\includegraphics[width=6cm]{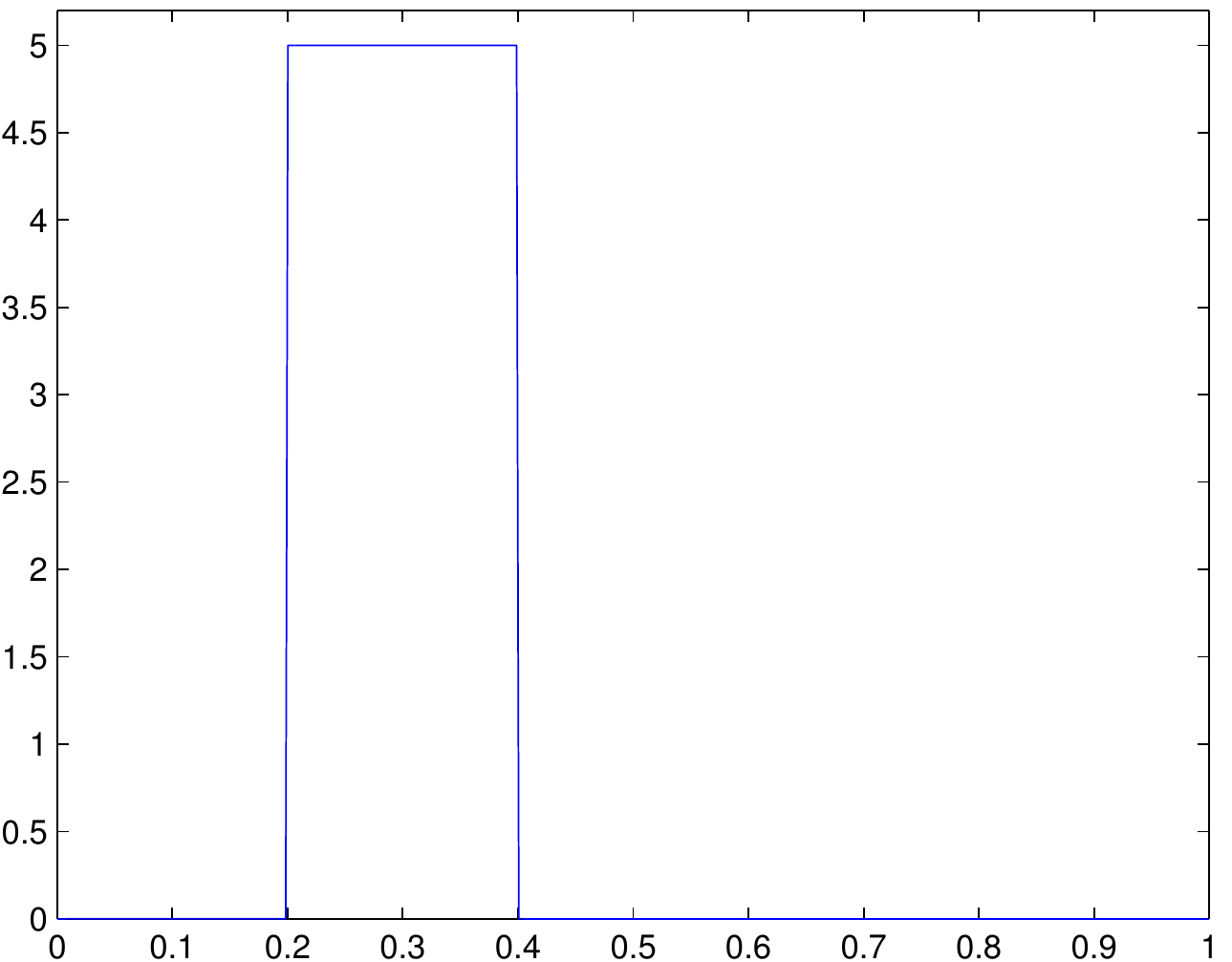}}
  \caption{The data \( \omega_1 \) and \( \omega_2 \)}
\end{figure} 

In this section, we shall show a few results of the numerical computation of minimizers to \( \widehat{\mathcal{E}}^\lambda \) and \( \widehat{\mathcal{E}}^\lambda_N \) in one dimension in order to numerically demonstrate the \( \Gamma \)-convergence result in Theorem \ref{thm:con-kern-est}.

\subsection{Grid approximation}
\label{sec:continuous-case}

\begin{figure}[t]
  \centering
  \subfigure[\( u \approx \mu \in L^1, \, \lambda = 10^{-4} \)]{\includegraphics[width=6cm]{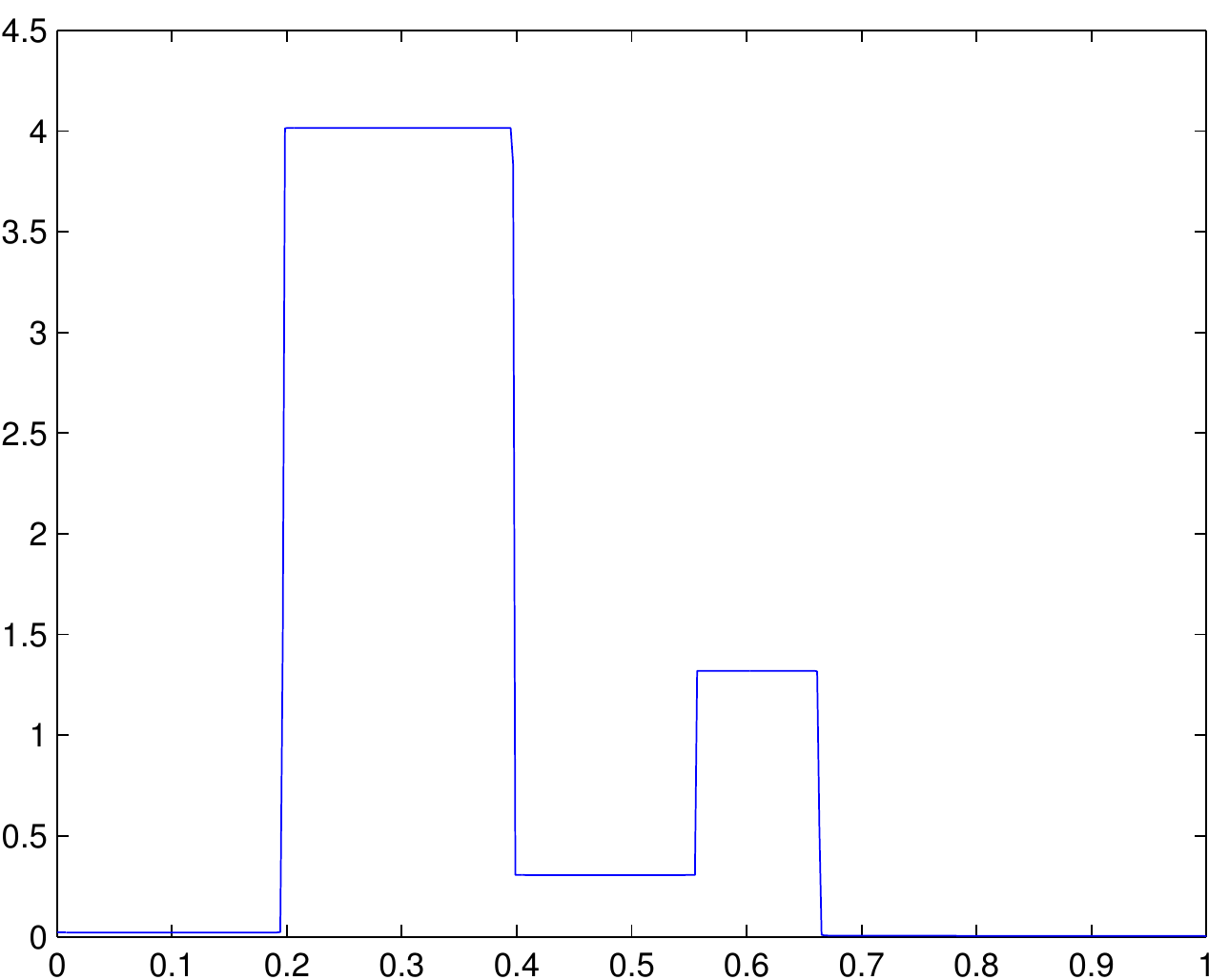}}
  \hspace{0.5cm}
  \subfigure[Particles supporting \( \mu_N \), \( \lambda = 10^{-4} \)]{\includegraphics[width=6cm]{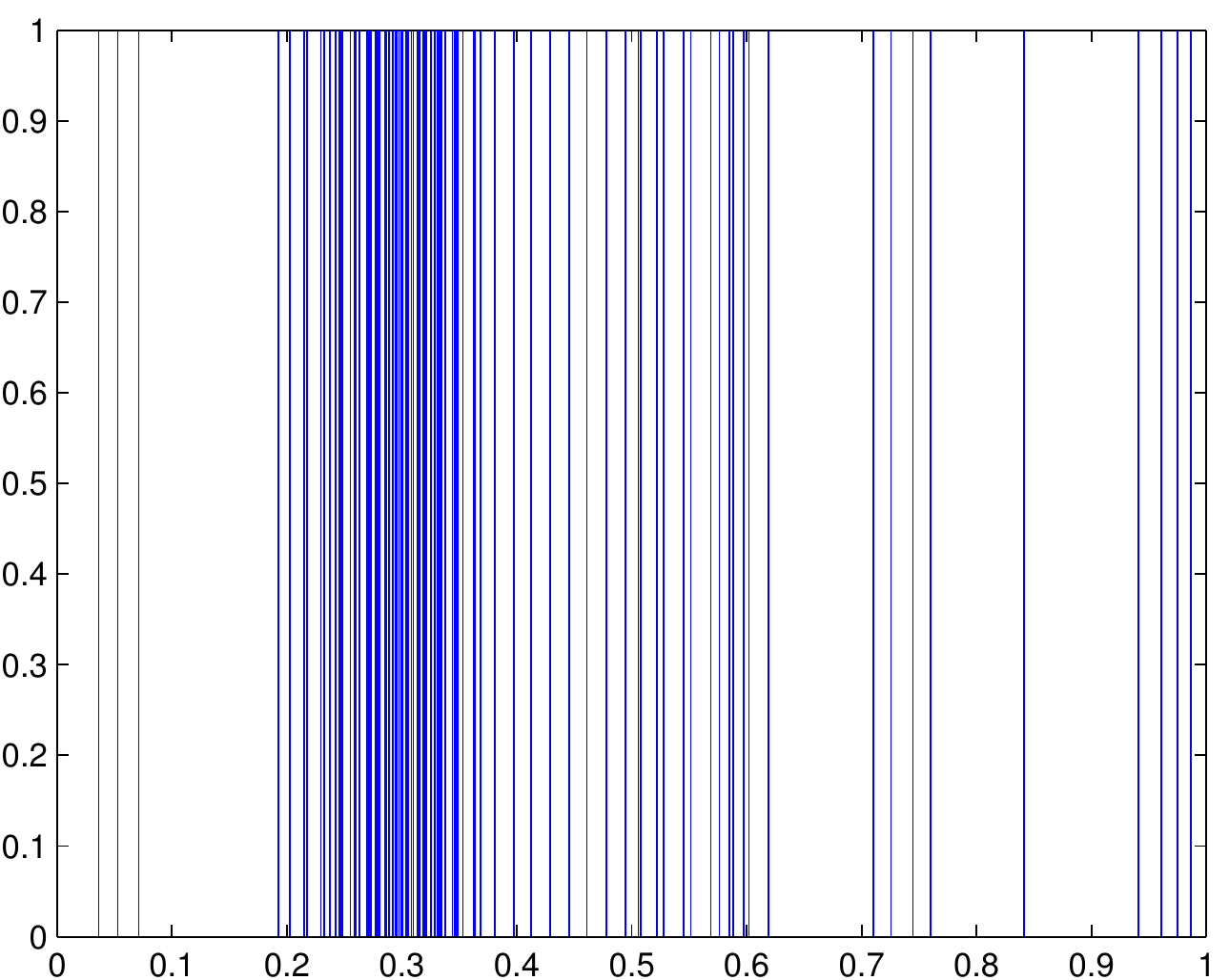}}

  \subfigure[\( u \approx \mu \in L^1, \, \lambda = 10^{-6} \)]{\includegraphics[width=6cm]{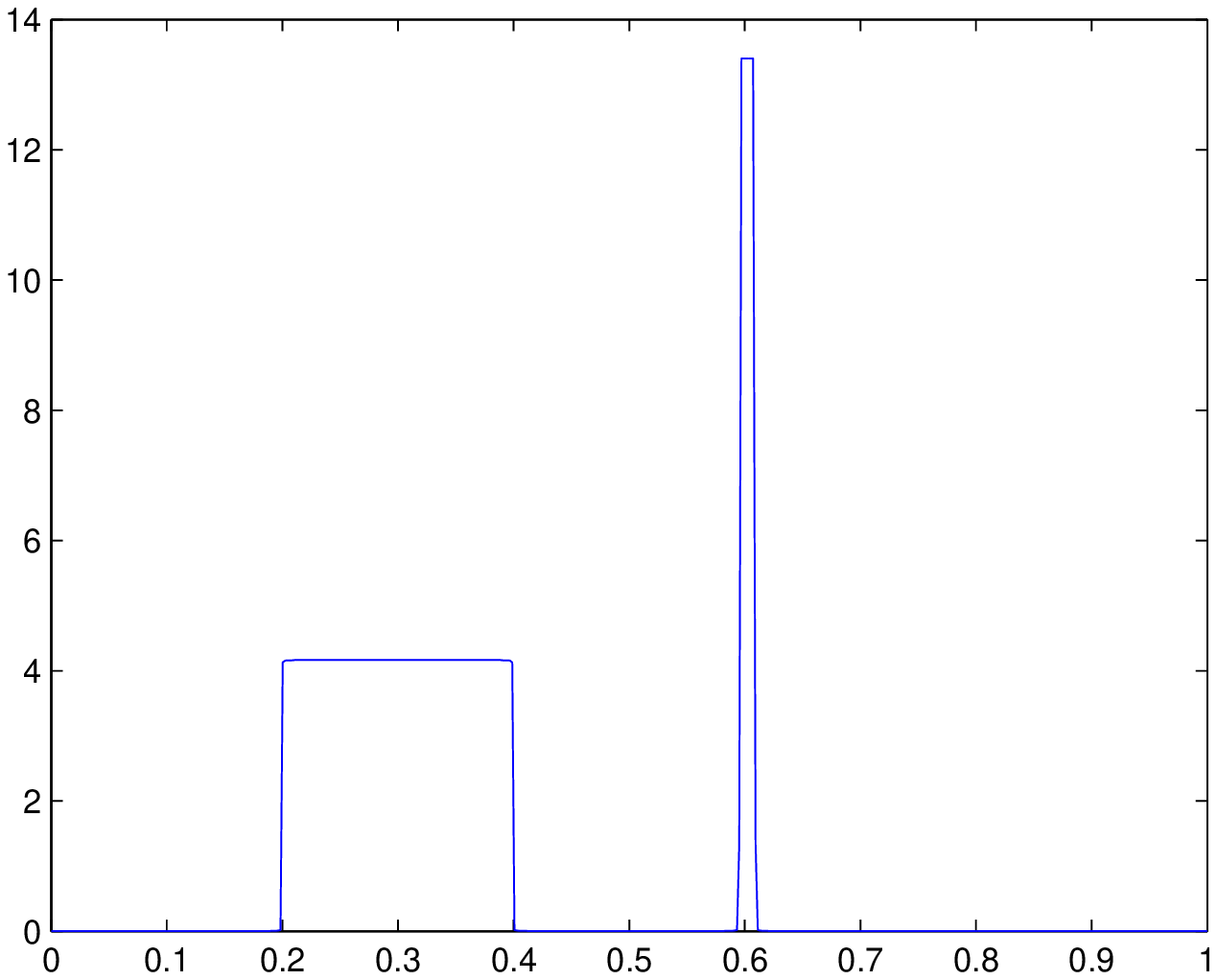}}
  \hspace{0.5cm}
  \subfigure[Particles supporting \( \mu_N \), \( \lambda = 10^{-6} \)]{\includegraphics[width=6cm]{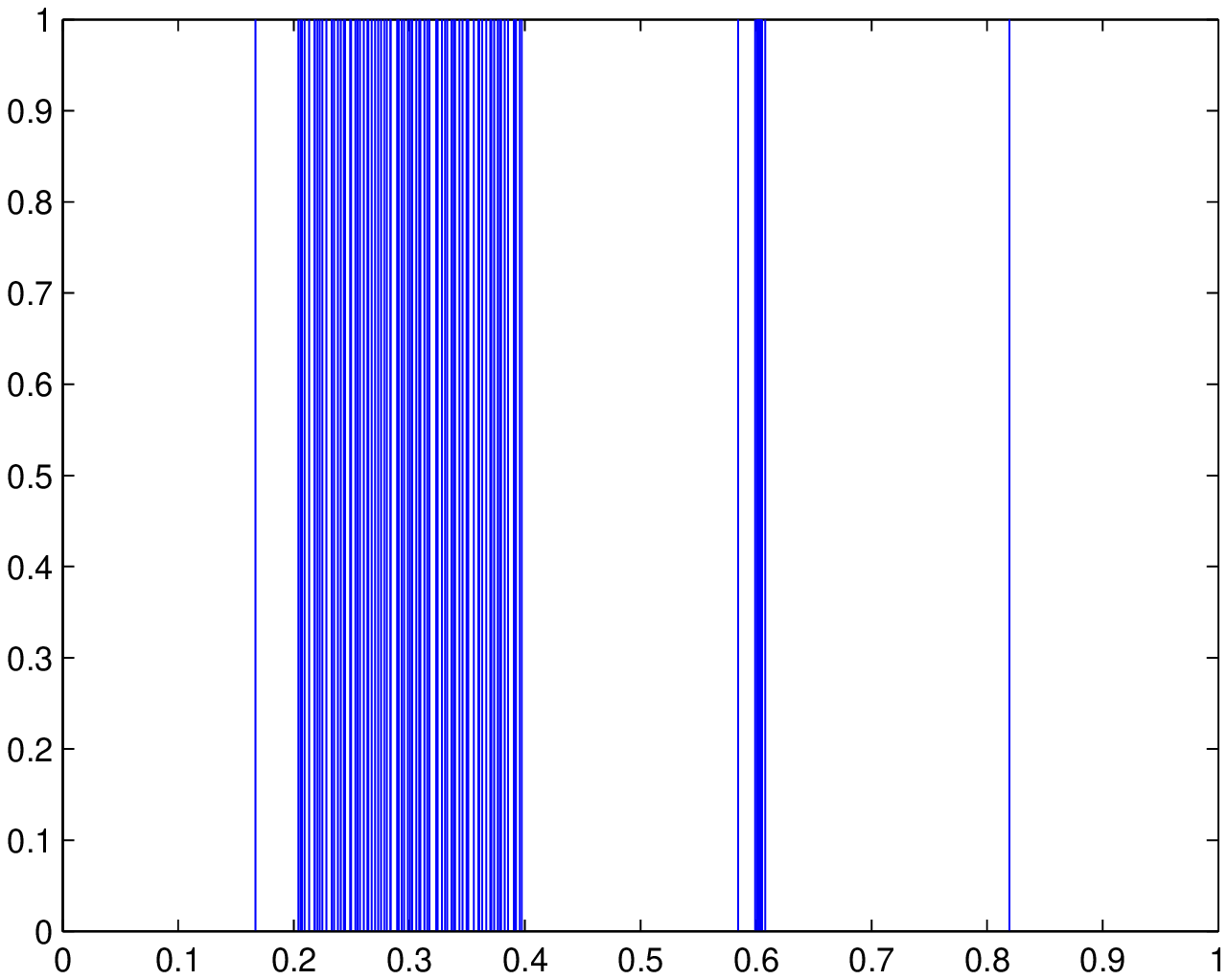}}
  \caption{Minimizers \( u \) of \eqref{eq:323} and minimizers \( \mu_N \) of \( \widehat{\mathcal{E}}^\lambda_N \) for \( \omega_1 \) as in Figure \ref{fig:data1} and parameters \( q = 1.0, \, N = 100 \).}
  \label{fig:spikes}
\end{figure}

By Theorem \ref{thm:con-kern-est}, we know that \( \widehat{\mathcal{E}}^\lambda_N \xrightarrow{\Gamma} \widehat{\mathcal{E}}^\lambda \), telling us that the particle minimizers of \( \widehat{\mathcal{E}}^\lambda \) will be close to a minimizer of the functional \( \widehat{\mathcal{E}}^\lambda \), which will be a \( BV \) function. Therefore, we would like to compare the particle minimizers to minimizers which were computed by using a more classical approximation method which in contrast maintains the underlying \( BV \) structure. One such approach is to approximate a function in \( BV \) by interpolation by piecewise constant functions on an equispaced discretization of the interval \( \Omega = [0, 1] \). Denoting the restriction of \( \widehat{\mathcal{E}}^\lambda \) to the space of these functions on a grid with \( N \) points by \( \widehat{\mathcal{E}}^\lambda_{N,\mathrm{grid}} \), it can be seen that we have \( \widehat{\mathcal{E}}^\lambda_{N,\mathrm{grid}} \xrightarrow{\Gamma} \widehat{\mathcal{E}}^\lambda \), hence it makes sense to compare minimizers of \( \widehat{\mathcal{E}}^\lambda_{N,\mathrm{grid}} \) and \( \widehat{\mathcal{E}}^\lambda_N \) for large \( N \).

If we denote by \( u \in \mathbb{R}^N \) the approximation to \( \mu \) and by \( w \in \mathbb{R}^N \) the one to \( \omega \), then the problem to minimize \( \widehat{\mathcal{E}}^\lambda_{N,\mathrm{grid}} \) takes the form
\begin{equation}
  \label{eq:323}
  \begin{aligned}
    &\text{minimize }& &(u-w)^T A_{q,\Omega} (u-w) + \lambda \sum_{i = 1}^{N-1} \left| u_{i + 1}-u_{i} \right| \\
    &\text{subject to } & &u \geq 0, \quad \sum_{i = 1}^{m} u_i = N,
  \end{aligned}
\end{equation}
where \( A_{q,\Omega} \) is the corresponding discretization matrix of the quadratic integral functional \( \widehat{\mathcal{E}} \), which is positive definite on the set \( \left\{ v : \sum v = 0 \right\} \) by the theory of Appendix \ref{cha:cond-posit-semi}. Solving the last condition \( \sum_{i=1}^{N} u_i = N \) for one coordinate of \( u \), we get a reduced matrix \( \widetilde{A}_{q, \Omega} \) which is positive definite. Together with the convex approximation term to the total variation, problem \eqref{eq:323} is a convex optimization problem which can be easily solved, e.g., by the \textsc{cvx} package \cite{cvx}, \cite{08-Grant-Boyd-Convex-Programs}.

As model cases to study the influence of the total variation, the following data were considered (see Figure \ref{fig:data1} and Figure \ref{fig:data2} for their visual representation)
\begin{enumerate}
\item \( \omega_1 = 4 \cdot 1_{[0.2,0.4]} + 40 \cdot 1_{[0.6,0.605] } \), the effect of the regularization being that the second bump gets smaller and more spread out with increasing parameter \( \lambda \), see Figure \ref{fig:spikes};
\item \( \tilde \omega_2=  \frac{1}{1+\|\eta|_{>0}\|_1} \left (  \omega_2 + \eta|_{>0} \right ) \), where \( \eta \) is Gaussian noise affecting the reference measure \( \omega_2 = 5 \cdot 1_{[0.2,0.4]} \), where we cut off the negative part and re-normalized the datum to get a probability measure. The effect of the regularization here is a filtering of the noise, see Figure \ref{fig:noise}.
\end{enumerate} 

\subsection{Particle approximation}
\label{sec:part-appr-1}

\begin{figure}[t]
  \centering
  \subfigure[\(  u \approx \mu \in L^1, \, \lambda = 0 \)]{\includegraphics[width=6cm]{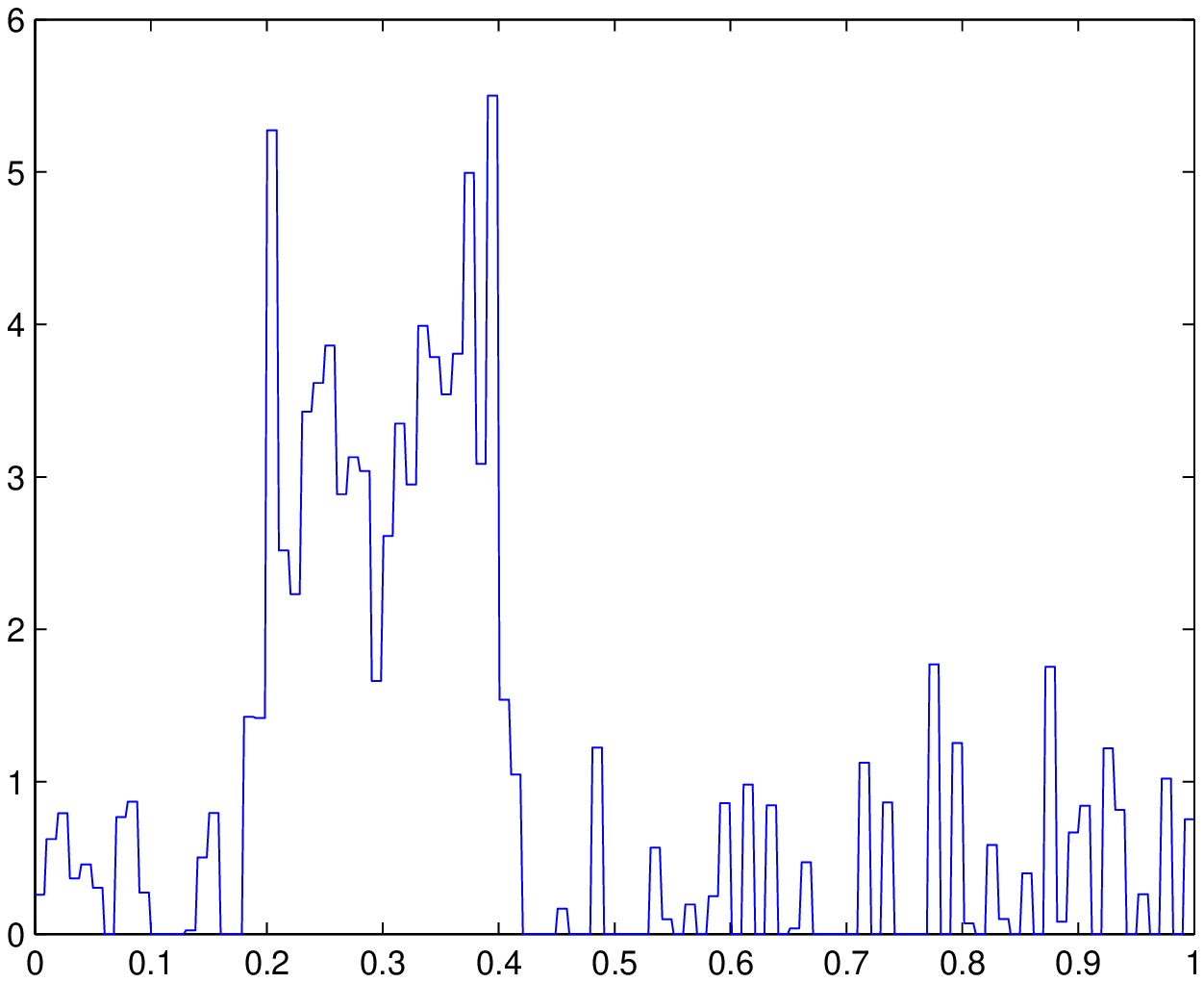}}
  \hspace{0.5cm}
  \subfigure[ Particles supporting \( \mu_N \), \, \(\lambda = 0 \)]{\includegraphics[width=6cm]{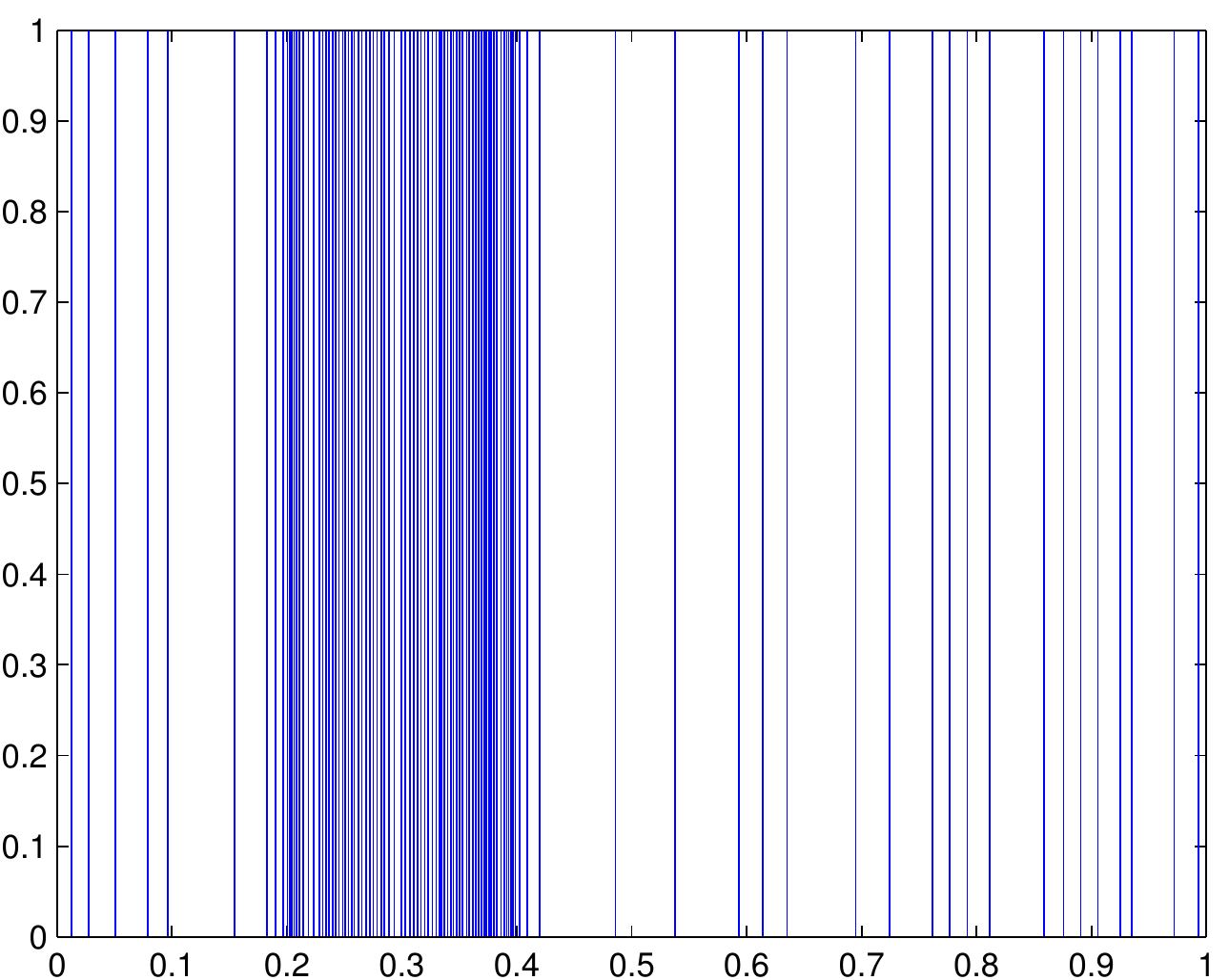}}

  \subfigure[\(  u \approx \mu \in L^1, \, \lambda = 10^{-5} \)]{\includegraphics[width=6cm]{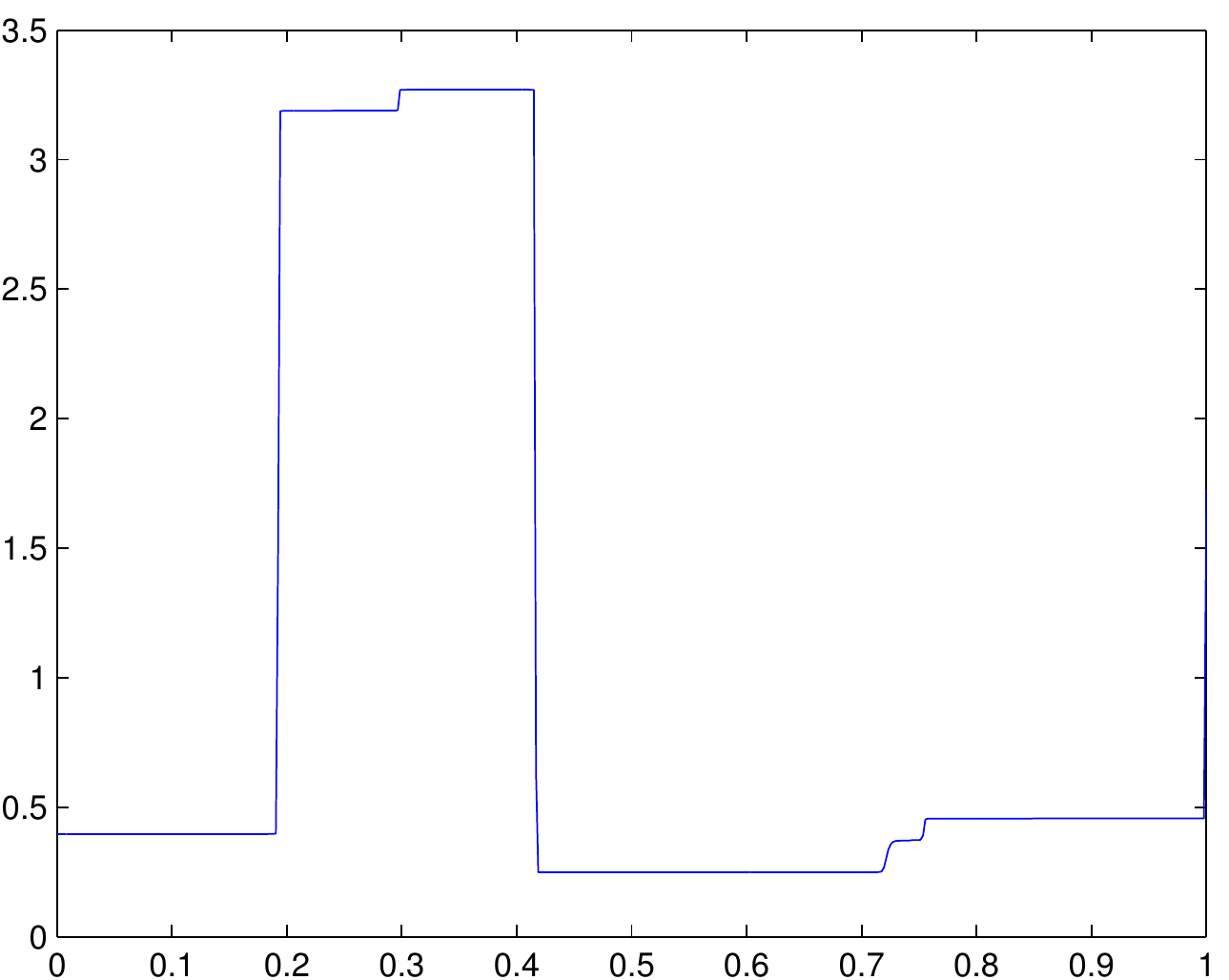}}
  \hspace{0.5cm}
  \subfigure[Particles supporting \( \mu_N \), \,\( \lambda = 10^{-5} \)]{\includegraphics[width=6cm]{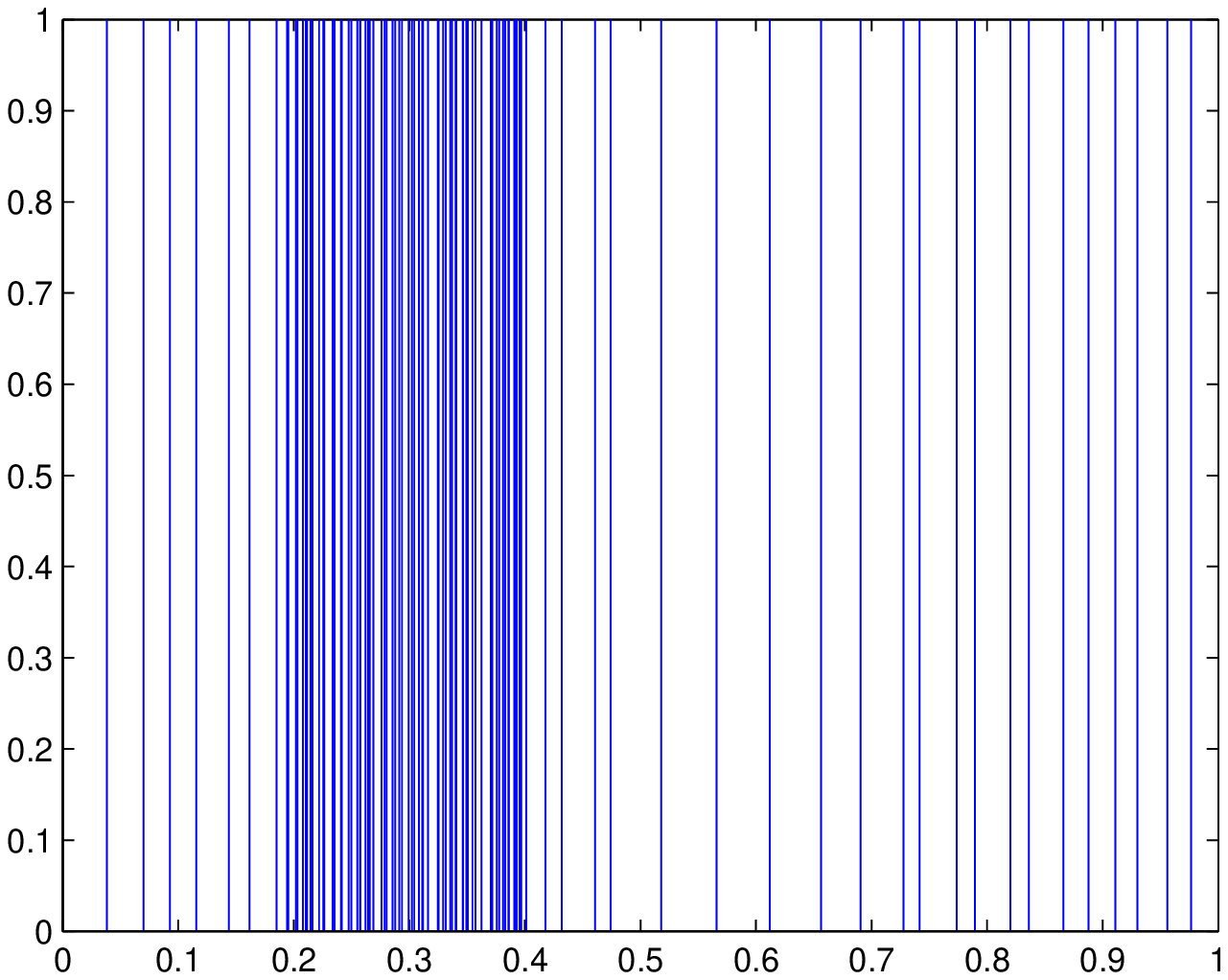}}
  \caption{Minimizers \( u \) of \eqref{eq:323} and minimizers \( \mu_N \) of \( \widehat{\mathcal{E}}^\lambda_N \) for a projection of \( \omega_2 + \eta \) as in Figure \ref{fig:data2} and parameters \( q = 1.5, \, N = 100 \).}
  \label{fig:noise}
\end{figure}

The solutions in the particle case were computed by the \textsc{matlab} optimization toolbox, in particular the Quasi-Newton method available via the \textsc{fminunc} command. The corresponding function evaluations were computed directly in the case of the repulsion functional and by a trapezoidal rule in the case of the attraction term. For the kernel estimator, we used the one sketched in Figure \ref{fig:discrete-tv},
\begin{equation*}
  K(x) = \left( 1 - \left| x \right| \right) \cdot 1_{[-1,1]}(x), \quad x \in \mathbb{R}.
\end{equation*}

\subsection{Results}
\label{sec:num-results}

As for the \( L^1 \) case, we see that the total variation regularization works well as a regularizer and allows us to recover the original profile from a datum disturbed by noise.

When it comes to the particle case, we numerically confirm the theoretical results of convergence for \( N \rightarrow \infty \) of Section \ref{sec:discrete-version-tv-1}, since the minimizers of the particle system behave roughly like the quantizers of the problem in \( L^1 \). 

\section{Conclusion}
\label{sec:conclusion}

Apart from the relatively simple results on existence for asymmetric exponents \( q_a \neq q_r \) in Section \ref{sec:prel-obs}, the Fourier representation of Section \ref{sec:prop-funct-mathbbrd}, building upon the theory of conditionally positive semi-definite functions reported in Appendix \ref{cha:cond-posit-semi}, proved essential to establish the well-posedness of the problem for equal exponents $1 \leq q_a=q_r< 2$, in terms of the lower semi-continuous envelope of the energy $\mathcal E$. This allowed us to use classical tools of the calculus of variations, in particular the machinery of \( \Gamma \)-convergence, to prove statements concerning the consistency of the particle approximation, Theorem \ref{thm:cons-part-appr}, and the moment bound, Theorem \ref{thm:moment-bound-symmetric}, which would be otherwise not at all obvious when just considering the original spatial definition of \( \mathcal{E} \).
Moreover, it enabled us to easily analyze the regularized version of the functional in Section \ref{sec:tv-reg}, which on the particle level allowed us to present a novel interpretation of the total variation as a nonlinear attractive-repulsive potential, translating the regularizing effect of the total variation in the continuous case into an energy which promotes a configuration of the particles which is as homogeneous as possible.

\subsection*{Acknowledgement}
Massimo Fornasier~is supported by the ERC-Starting Grant for the project ``High-Dimensional Sparse Optimal Control''. Jan-Christian H\"utter acknowledges the partial financial support of the START-Project ``Sparse Approximation and Optimization in High-Dimensions''  during the early preparation of this work.
\appendix
\section{Conditionally positive definite functions}
\label{cha:cond-posit-semi}

In order to compute the Fourier representation of the energy functional \( \mathcal{E} \) in Section \ref{sec:four-repr-gener}, we used the notion of \emph{generalized Fourier transforms} and \emph{conditionally positive definite functions} from \cite{Wend05}, which we shall briefly recall here for the sake of completeness. In fact, the main result reported below, Theorem \ref{thm:cond-ft-power} is shown in a slightly modified form with respect to \cite[Theorem 8.16]{Wend05} in order to allow us also for
the proof of the moment bound in Section \ref{sec:moment-bound-symm}. The representation formula \eqref{eq:39} is a consequence of Theorem \ref{thm:repr-thm-cond-semi} below, which serves as a characterization in the theory of conditionally positive definite functions.

\begin{definition}
  \label{def:cond-definit}
  \cite[Definition 8.1]{Wend05}
  Let \( \mathbb{P}_{k}(\mathbb{R}^d) \) denote the set of polynomial functions on \( \mathbb{R}^d \) of degree less or equal than \( k \). We call a continuous function \( \Phi \colon \mathbb{R}^d \rightarrow \mathbb{C} \) \emph{conditionally positive semi-definite of order} \( m \) if for all \( N\in\mathbb{N} \), pairwise distinct points \( x_1,\ldots,x_N \in \mathbb{R}^d \), and \( \alpha \in \mathbb{C}^N \) with
  \begin{equation}
    \label{eq:281}
    \sum_{j=1}^{N} \alpha_j p(x_j) = 0, \quad \text{for all } p \in \mathbb{P}_{m - 1}(\mathbb{R}^d),
  \end{equation}
  the quadratic form given by \( \Phi \) is non-negative, i.e.,
  \begin{equation*}
    \sum_{j,k=1}^{N} \alpha_j \overline{\alpha_k} \Phi(x_j - x_k) \geq 0.
  \end{equation*}
  Moreover, we call \( \Phi \) \emph{conditionally positive definite of order} \( m \) if the above inequality is strict for \( \alpha \neq 0 \).
\end{definition}

\subsection{Generalized Fourier transform}
\label{sec:gener-four-transf}

When working with distributional Fourier transforms, which can serve to characterize the conditionally positive definite functions defined above, it can be opportune to further reduce the standard Schwartz space \( \mathcal{S} \) to functions which in addition to the polynomial decay for large arguments also exhibit a certain decay for small ones. In this way, one can elegantly neglect singularities in the Fourier transform at $0$, which could otherwise arise.

\begin{definition}
  [Restricted Schwartz class \( \mathcal{S}_m \)]
  \label{def:restr-schwartz}
  \cite[Definition 8.8]{Wend05}
  Let \( \mathcal{S} \) be the Schwartz space of functions in \( C^\infty(\mathbb{R}^d) \) which for \( \left| x \right| \rightarrow \infty \) decay faster than any fixed polynomial. Then, for \( m \in \mathbb{N} \), we denote by \( \mathcal{S}_m \) the set of those functions in \( \mathcal{S} \) which additionally fulfill
  \begin{equation}
    \label{eq:283}
    \gamma(\xi) = O(\left| \xi \right|^m) \quad \text{for } \xi \rightarrow 0.
  \end{equation}

  Furthermore, we shall call an (otherwise arbitrary) function \( \Phi \colon \mathbb{R}^d \rightarrow \mathbb{C} \) \emph{slowly increasing} if there is an \( m \in \mathbb{N} \) such that
  \begin{equation*}
    \Phi(x) = O \left( \left| x \right|^m \right) \quad \text{for } \left| x \right| \rightarrow \infty.
  \end{equation*}
\end{definition}

\begin{definition}
  [Generalized Fourier transform]
  \label{def:gen-fourier-transform}
  \cite[Definition 8.9]{Wend05}
  For \( \Phi \colon \mathbb{R}^d \rightarrow \mathbb{C} \) continuous and slowly increasing, we call a measurable function \(\widehat{\Phi} \in L_{\mathrm{loc}}^2(\mathbb{R}^d \setminus \left\{ 0 \right\})\) the \emph{generalized Fourier transform} of \( \Phi \) if there exists an integer \( m \in \mathbb{N}_0 \) such that
  \begin{equation}
    \label{eq:285}
    \int_{\mathbb{R}^d} \Phi(x)\widehat{\gamma}(x) \diff x = \int_{\mathbb{R}^d} \widehat{\Phi}(\xi) \gamma(\xi) \diff \xi \quad \text{for all } \gamma \in \mathcal{S}_{2m}.
  \end{equation}
  Then, we call \( m \) the \emph{order} of \( \widehat{\Phi} \).
\end{definition}

Note that the order here is defined in terms of \( 2m \) instead of \( m \).

The consequence of this definition is that we can ignore additive polynomial terms in \( \Phi \) which would result in Dirac distributions in the Fourier transform.

\begin{proposition}
  \label{prp:poly-vanish}
  \cite[Proposition 8.10]{Wend05} If \( \Phi \in \mathbb{P}_{m-1}(\mathbb{R}^d) \), then \( \Phi \) has the generalized Fourier transform \( 0 \) of order \( m/2 \). Conversely, if \( \Phi \) is a continuous function which has generalized Fourier transform \( 0 \) of order \( m/2 \), then \( \Phi \in \mathbb{P}_{m-1} \left( \mathbb{R}^d \right) \).
\end{proposition}

\begin{proof}[Sketch of proof]
  The first claim follows from the fact that multiplication by polynomials corresponds to computing derivatives of the Fourier transform: by condition \eqref{eq:283}, all derivatives of order less than \( m \) of a test function \( \gamma \in \mathcal{S}_m \) have to vanish.

  The second claim follows from considering the pairing \( \int_{\mathbb{R}^d} \Phi(x) \widehat{g}(x) \diff x \) for a general \( g \in \mathcal{S} \) and projecting it into \( \mathcal{S}_m \) by setting
  \begin{equation*}
    \gamma(x) := g(x) - \sum_{\left| \beta \right| < m} \frac{D^{\beta}g(0)}{\beta!}x^\beta \chi(x), \quad x \in \mathbb{R}^d,
  \end{equation*}
  where  \( \chi \in C_0^\infty(\mathbb{R}^d) \) is \( 1 \) close to \( 0 \).
\end{proof}

\subsection{Representation formula for conditionally positive definite functions}
\label{sec:repr-form-cond}

Before proceeding to prove Theorem \ref{thm:repr-thm-cond-semi}, we need two Lemmata. The first one is the key to applying the generalized Fourier transform in our case, namely that functions fulfilling the decay condition \eqref{eq:283} can be constructed as Fourier transforms of point measures satisfying condition \eqref{eq:281}. The second one recalls some basic facts about the Fourier transform of the Gaussian, serving to pull the exponential functions in Lemma \ref{lem:19} into \( \mathcal{S}_m \).

\begin{lemma}
  \cite[Lemma 8.11]{Wend05}
  \label{lem:19}
  Given pairwise distinct points \( x_1,\ldots,x_N \in \mathbb{R}^d \) and \( \alpha \in \mathbb{C}^N \setminus \left\{ 0 \right\} \) such that
  \begin{equation}
    \label{eq:287}
    \sum_{j=1}^{N} \alpha_j p(x_j) = 0, \quad \text{for all } p \in \mathbb{P}_{m-1}(\mathbb{R}^d),
  \end{equation}
  then
  \begin{equation*}
    \sum_{j=1}^{N} \alpha_j \e^{\i x_j \cdot \xi} = O \left( \left| \xi \right|^m \right) \quad \text{for } \left| \xi \right| \rightarrow 0.
  \end{equation*}
\end{lemma}

\begin{proof}
  Expanding the exponential function into its power series yields
  \begin{equation*}
    \sum_{j=1}^{N} \alpha_j \e^{\i x_j \cdot \xi} = \sum_{k = 0}^{\infty} \frac{\i^k}{k!} \sum_{j=1}^{N} \alpha_j \left( x_j \cdot \xi \right)^k,
  \end{equation*}
  and by condition \eqref{eq:287} its first \( m \) terms vanish, giving us the desired behavior.
\end{proof}

\begin{lemma}
  \label{lem:20}
  \cite[Theorem 5.20]{Wend05}
  Let \( l > 0 \) and \( g_l (x) := (l/\pi)^{d/2} \e^{-l \left| x \right|^2} \). Then,
  \begin{enumerate}
  \item \label{lem:21} \( \widehat{g}_l(\xi) = \e^{-\left| \xi \right|^2/(4l)} \);
  \item \label{lem:22} for \( \Phi \colon \mathbb{R}^d \rightarrow \mathbb{C}\) continuous and slowly increasing, we have
    \begin{equation*}
      \Phi(x) = \lim_{l \rightarrow \infty} (\Phi \ast g_l)(x).
    \end{equation*}
  \end{enumerate}
\end{lemma}

\begin{theorem}
  \label{thm:repr-thm-cond-semi}
  \cite[Corollary 8.13]{Wend05}
  Let \( \Phi \colon \mathbb{R}^d \rightarrow \mathbb{C} \) be a continuous and slowly increasing function with a non-negative, non-vanishing generalized Fourier transform \( \widehat{\Phi} \) of order \( m \) that is continuous on \( \mathbb{R}^d \setminus \left\{ 0 \right\} \). Then, we have
  \begin{equation}
    \label{eq:291}
    \sum_{j,k=1}^N \alpha_j \overline{\alpha}_k \Phi \left( x_j-x_k \right) = \int_{\mathbb{R}^d} \left| \sum_{j=1}^{N} \alpha_j \e^{\i x_j \cdot \xi}\right|^2 \widehat{\Phi}(\xi) \diff \xi.
  \end{equation}
\end{theorem}

\begin{proof}
  Let us start with the right-hand side of the claimed identity \eqref{eq:291}: By Lemma \ref{lem:19}, the function
  \begin{equation*}
    f(\xi) := \left| \sum_{j=1}^{N} \alpha_j \e^{\i x_j \cdot \xi} \right|^2 \widehat{g}_l(\xi)
  \end{equation*}
  is in \( \mathcal{S}_{2m} \) for all \( l > 0 \). Moreover, by the monotone convergence theorem,
  \begin{align*}
    \int_{\mathbb{R}^d} \left| \sum_{j=1}^{N} \alpha_j \e^{\i x_j \cdot \xi}\right|^2 \widehat{\Phi}(\xi) \diff \xi = {} &\lim_{l \rightarrow \infty} \int_{\mathbb{R}^d} \left| \sum_{j=1}^{N} \alpha_j \e^{\i x_j \cdot \xi}\right| \widehat{g}_l(\xi)\, \widehat{\Phi}(\xi) \diff \xi\\
    = {} &\lim_{l \rightarrow\infty} \int_{\mathbb{R}^d} \left( \left| \sum_{j=1}^{N} \alpha_j \e^{\i x_j \cdot .}\right|^2 \widehat{g}_l(.) \right)^{\wedge}(x)\, \Phi(x) \diff x.
  \end{align*}
  Now, by Lemma \ref{lem:20}, \ref{lem:21},
  \begin{align*}
    \left( \left| \sum_{j=1}^{N} \alpha_j \e^{\i x_j \cdot .}\right|^2 \widehat{g}_l(.) \right)^{\wedge}(x) = {} &\widehat{\widehat{g\,}}_l \ast \left( \sum_{j=1}^{N} \alpha_j \delta_{x_j} \right) \ast \left( \sum_{j=1}^{N} \overline{\alpha_j} \delta_{-x_j} \right)(x)\\
    = {} & g_l\ast \left( \sum_{j=1}^{N} \alpha_j \delta_{x_j} \right) \ast \left( \sum_{j=1}^{N} \overline{\alpha_j} \delta_{-x_j} \right) (x)
  \end{align*}
  and therefore
  \begin{align*}
    \leadeq{\lim_{l \rightarrow\infty} \int_{\mathbb{R}^d} \left( \left| \sum_{j=1}^{N} \alpha_j \e^{\i x_j \cdot .}\right|^2 \widehat{g}_l(.) \right)^{\wedge}(x)\, \Phi(x) \diff x}\\
    = {} & \lim_{l \rightarrow\infty} \int_{\mathbb{R}^d}  \Phi(x) \, g_l \ast \left( \sum_{j=1}^{N} \alpha_j \delta_{x_j} \right) \ast \left( \sum_{j=1}^{N} \overline{\alpha_j} \delta_{-x_j} \right)(x) \diff x\\
    = {} & \lim_{l \rightarrow\infty} \sum_{i,j=1}^{N} \int_{\mathbb{R}^d} \alpha_i \overline{\alpha_j}\,  \Phi(x) \, g_l(x-(x_i-x_j)) \diff x\\
    = {} & \lim_{l \rightarrow\infty} \sum_{i,j=1}^{N} \int_{\mathbb{R}^d} \alpha_i \overline{\alpha_j}\,  \Phi(x-(x_i-x_j)) \, g_l(x) \diff x\\
    = {} & \sum_{i,j=1}^{N} \alpha_i \overline{\alpha_j} \, \Phi(x_i-x_j)
  \end{align*}
   by Lemma \ref{lem:20}, \ref{lem:22}.
\end{proof}

\subsection{Computation for the power function}
\label{sec:comp-power-funct}

Given Theorem \ref{thm:repr-thm-cond-semi}, we are naturally interested in the explicit formula of the generalized Fourier transform for the power function \( x \mapsto \left| x \right|^q \) for \( q \in [1,2) \). It is a nice example of how to pass from an ordinary Fourier transform to the generalized Fourier transform by extending the formula by means of complex analysis. Our starting point will be the multiquadric \( x \mapsto \left( c^2 + \left| x \right|^2 \right)^{\beta} \) for \( \beta < -d/2 \), whose Fourier transform involves the modified Bessel function of the third kind:

\begin{definition}
  [Modified Bessel function]
  \cite[Definition 5.10]{Wend05}
  For \( \nu \in \mathbb{C} \), \( z \in \mathbb{C} \) with \( \left| \operatorname{arg} z \right| < \pi/2 \), define
  \begin{equation*}
    K_\nu(z) := \int_{0}^\infty \exp(-z \cosh(t)) \cosh(\nu t)  \diff t,
  \end{equation*}
  the \emph{modified Bessel function of the third kind of order} \( \nu \in \mathbb{C} \).
\end{definition}

\begin{theorem}
  \label{poly:cond-ft-multquad1}
  \cite[Theorem 6.13]{Wend05}
    For \( c > 0 \) and \( \beta < -d/2 \),
    \begin{equation*}
      \Phi(x) = (c^2+\left| x \right|^2)^{\beta}, \quad x \in \mathbb{R}^d,
    \end{equation*}
    has Fourier transform given by    \begin{equation}
      \label{eq:300}
      \widehat{\Phi}(\xi) =(2\pi)^{d/2}  \frac{2^{1+\beta}}{\Gamma(-\beta)} \left( \frac{\left| \xi \right|}{c} \right)^{-\beta-d/2}K_{d/2+\beta}(c \left| \xi \right|).
    \end{equation}
\end{theorem}

The next lemma provides the asymptotic behavior of the involved Bessel function for large and small values, which we need for the following proof.

\begin{lemma}
  [Estimates for \( K_\nu \)]
  \label{lem:23}
  1. \cite[Lemma 5.14]{Wend05}
  For \( \nu \in \mathbb{C}, r > 0 \),
  \begin{equation}
    \label{eq:301}
    \left| K_\nu(r) \right| \leq
    \begin{cases}
      2^{\left| \Re (\nu) \right| - 1}\Gamma \left( \left| \Re(\nu) \right| \right) r^{-\left| \Re(\nu) \right|}, &\Re(\nu) \neq 0,\\
      \frac{1}{\e}-\log \frac{r}{2},&r < 2, \Re(\nu) = 0.
    \end{cases}
  \end{equation}

  2. For large \( r \), \( K_\nu \) has the asymptotic behavior
  \begin{equation}
    \label{eq:302}
    \left| K_\nu(r) \right| \leq \sqrt{\frac{2\pi}{r}} \e^{-r} \e^{\left| \Re(\mu) \right|^2/(2r)}, \quad r > 0.
  \end{equation}
\end{lemma}

\begin{theorem}
  \label{thm:cond-ft-power}
  1. \cite[Theorem 8.15]{Wend05}
  \( \Phi(x) = (c^2 + \left| x \right|^2)^\beta \), \( x \in \mathbb{R}^d \) for \( c > 0 \) and \( \beta \in \mathbb{R} \setminus \frac{1}{2} \mathbb{N}_0 \) has the generalized Fourier transform
  \begin{equation}
    \label{eq:303}
    \widehat{\Phi}(\xi) = (2\pi)^{d/2} \frac{2^{1+\beta}}{\Gamma(-\beta)} \left( \frac{\left| \xi \right|}{c} \right)^{-\beta-d/2} K_{d/2+\beta}(c \left| \xi \right|), \quad \xi \neq 0
  \end{equation}
  of order \( m = \max(0, \lceil 2\beta \rceil/2) \).
  
  2. \cite[Theorem 8.16]{Wend05}
  \( \Phi(x) = \left| x \right|^\beta \), \( x \in \mathbb{R}^d \) with \( \beta \in \mathbb{R}_+ \setminus \mathbb{N} \) has the generalized Fourier transform
  \begin{equation*}
    \widehat{\Phi}(\xi) = (2\pi)^{d/2}\frac{2^{\beta+d/2}\Gamma((d+\beta)/2)}{\Gamma(-\pi/2)} \left| \xi \right|^{-\beta-d}, \quad \xi \neq 0.
  \end{equation*}
  of order \( m = \lceil \beta \rceil/2 \).
\end{theorem}

\begin{proof}
  1. We can pass from formula \eqref{eq:300} to \eqref{eq:303} by analytic continuation, where the exponent \( m \) serves to give us the needed integrable dominating function, see formula \eqref{eq:307} below.

  Let \( G = \left\{ \lambda \in \mathbb{C} : \Re(\lambda) < m \right\} \) and
  \begin{align*}
    \varphi_\lambda(\xi) := {} &(2\pi)^{d/2} \frac{2^{1+\lambda}}{\Gamma(-\lambda)} \left( \frac{\left| \xi \right|}{c} \right)^{-\lambda-d/2} K_{d/2+\lambda}(c \left| \xi \right|)\\
    \Phi_\lambda(\xi) := {} & \left( c^2 + \left| \xi \right|^2 \right)^\lambda.
  \end{align*}
  We want to show
  \begin{equation*}
    \int_{\mathbb{R}^d} \Phi_\lambda(\xi)\widehat{\gamma}(\xi) \diff \xi = \int_{\mathbb{R}^d} \varphi_\lambda(\xi) \gamma(\xi) \diff \xi, \quad \text{for all } \gamma \in \mathcal{S}_{2m},
  \end{equation*}
  which is so far true for $\lambda$ real and \( \lambda < d/2 \) by \eqref{eq:300}. As the integrands \( \Phi_\lambda \widehat{\gamma} \) and \( \varphi_\lambda \gamma  \) are analytic, the integral functions are also analytic by Cauchy’s integral formula and Fubini's theorem if we can find a uniform dominating function for each of them on an arbitrary compact set \( \mathcal{C} \subseteq G \). As this is clear for \( \Phi_\lambda \) by the decay of \( \gamma \in \mathcal{S} \), it remains to consider \( \varphi_\lambda \).

  Setting \( b := \Re(\lambda) \), for \( \xi \) close to \( 0 \) we get by estimate \eqref{eq:301} of Lemma \ref{lem:23} that
  \begin{equation}
    \label{eq:307}
    \left| \varphi_\lambda(\xi) \gamma(\xi) \right| \leq C_\gamma \frac{2^{b+\left| b + d/2 \right|}\Gamma(\left| b + d/2 \right|)}{\left| \Gamma(-\lambda) \right|}c^{b+d/2-\left| b+d/2 \right|}\left| \xi \right|^{-b-d/2-\left| b+d/2 \right|+2m}
  \end{equation}
  for \( b \neq -d/2 \) and
  \begin{equation*}
    \left| \varphi_\lambda(\xi)\gamma(\xi) \right|\leq C_\lambda \frac{2^{1-d/2}}{\left| \Gamma(-\lambda) \right|}\left( \frac{1}{\e} - \log \frac{c \left| \xi \right|}{2} \right).
  \end{equation*}
  for \( b = -d/2 \). Taking into account that \( \mathcal{C} \) is compact and \( 1/\Gamma \) is an entire function, this yields
  \begin{equation*}
    \left| \varphi_\lambda(\xi)\gamma(\xi) \right| \leq C_{\lambda,m,c,\mathcal{C}} \left( 1 + \left| \xi \right|^{-d+2\varepsilon}-\log \frac{c \left| \xi \right|}{2} \right),
  \end{equation*}
  with \( \left| \xi \right| < \min \left\{ 1/c,1 \right\} \) and \( \varepsilon := m-b \), which is locally integrable.

  For \( \xi \) large, we similarly use estimate \eqref{eq:302} of Lemma \ref{lem:23} to obtain
  \begin{equation*}
    \left| \varphi_\lambda(\xi)\gamma(\xi) \right| \leq C_\lambda \frac{2^{1+b}\sqrt{2\pi}}{\left| \Gamma(-\lambda) \right|}c^{b+(d-1)/2} \left| \xi \right|^{-b-(d+1)/2} \e^{-c \left| \xi \right|} \e^{\left| b+d/2 \right|^2/(2c \left| \xi \right|)}
  \end{equation*}
  and consequently
  \begin{equation*}
    \left| \varphi_\lambda(\xi)\gamma(\xi) \right| \leq C_{\gamma,m,\mathcal{C},c} \e^{-c \left| \xi \right|},
  \end{equation*}
  which certainly is integrable.
  
  2. We want to pass to \( c \rightarrow 0 \) in formula \eqref{eq:303}. This can be done by applying the dominated convergence theorem in the definition of the generalized Fourier transform \eqref{eq:285}. Writing \( \Phi_c(x) := \left( c^2+ \left| x \right|^2 \right)^{\beta/2} \) for \( c>0 \), we know that
  \begin{equation*}
    \widehat{\Phi}_c(\xi) = \varphi_c(\xi) := (2\pi)^{d/2} \frac{2^{1+\beta/2}}{\left| \Gamma(-\beta/2) \right|} \left| \xi \right|^{-\beta-d}(c \left| \xi \right|)^{(\beta+d)/2}K_{(\beta+d)/2}(c \left| \xi \right|).
  \end{equation*}
  By using the decay properties of a \( \gamma \in \mathcal{S}_{2m} \) in the estimate \eqref{eq:307}, we get
  \begin{equation}
    \label{eq:1}
    \left| \varphi_c(\xi)\gamma(\xi) \right| \leq C_\gamma \frac{2^{\beta+d/2}\Gamma((\beta+d)/2}{\left| \Gamma(-\beta/2) \right|} \left| \xi \right|^{2m-\beta-d} \quad \text{for } \left| \xi \right| \rightarrow 0
  \end{equation}
  and
  \begin{equation*}
    \left| \varphi_c(\xi)\gamma(\xi) \right| \leq C_\gamma \frac{2^{\beta+d/2}\Gamma((\beta+d)/2)}{\left| \Gamma(-\beta/2) \right|} \left| \xi \right|^{-\beta-d},
  \end{equation*}
  yielding the desired uniform dominating function. The claim now follows by also taking into account that
  \begin{equation*}
    \lim_{r\rightarrow 0} r^\nu K_\nu(r) = \lim_{r\rightarrow 0} 2^{\nu-1} \int_{0}^{\infty} \e^{-t} \e^{-r^2/(4t)} t^{\nu-1} \diff t = 2^{\nu-1} \Gamma(\nu).
  \end{equation*}
\end{proof}

\begin{remark}[Fractional orders]
  \label{rem:frac-ord}
  In Theorem \ref{thm:cond-ft-power}, we have slightly changed the statement compared to the reference \cite[Theorem 8.16]{Wend05} in order to allow orders which are a multiple of \( 1/2 \) instead of just integers. This made sense in \cite{Wend05} because the definition of the order involves the space \( \mathcal{S}_{2m} \) due to its purpose in the representation formula of Theorem \ref{thm:repr-thm-cond-semi}, involving a quadratic functional. However, in Section \ref{sec:moment-bound-symm} we needed the generalized Fourier transform in a linear context. Fortunately, one can easily generalize the proof in \cite{Wend05} to this fractional case, as all integrability arguments remain true when permitting multiples of \( 1/2 \), in particular the estimates in \eqref{eq:307} and \eqref{eq:1}.
\end{remark}

\bibliographystyle{abbrv}
\bibliography{BibKineticDitheringVar}

\end{document}